\documentclass[10pt]{article}
\pdfoutput=1

\makeatletter

\@addtoreset{equation}{section}
\makeatother

\usepackage[cp1252]{inputenc}
\usepackage[OT1]{fontenc}
\usepackage[english]{babel}

\usepackage{amssymb, amsmath, 
amsthm,
amscd,
graphicx,
blackdvi,
comment
,cite
}

\input xy
\xyoption{all}

\graphicspath{{figures_pdf/}}

\def\ds{\displaystyle}

\def\R{{\mathbb R}}
\def\N{{\mathbb N}}
\def\Z{{\mathbb Z}}

\def\tcut{t_{\operatorname{cut}}}

\def\tt{\mathbf t}
\def\then{\quad\Rightarrow\quad}

\newcommand{\E}{\operatorname{E}\nolimits}

\newcommand{\Exp}{\operatorname{Exp}\nolimits}
\newcommand{\sgn}{\operatorname{sgn}\nolimits}
\newcommand{\am}{\operatorname{am}\nolimits}
\newcommand{\sn}{\operatorname{sn}\nolimits}
\newcommand{\cn}{\operatorname{cn}\nolimits}
\newcommand{\dn}{\operatorname{dn}\nolimits}
\newcommand{\Max}{\operatorname{Max}\nolimits}

\newcommand{\cl}{\operatorname{cl}\nolimits}
\newcommand{\Id}{\operatorname{Id}\nolimits}
\newcommand{\Cut}{\operatorname{Cut}\nolimits}
\newcommand{\Conj}{\operatorname{Conj}\nolimits}
\newcommand{\glob}{\operatorname{glob}\nolimits}
\newcommand{\loc}{\operatorname{loc}\nolimits}
\newcommand{\SE}{\operatorname{SE}\nolimits}
\newcommand{\intt}{\operatorname{int}\nolimits}
\newcommand{\SO}{\operatorname{SO}\nolimits}

\def\a{\alpha}
\def\b{\beta}
\def\g{\gamma}
\def\th{\theta}
\def\lam{\lambda}
\def\f{\varphi}
\def\p{\psi}
\def\eps{\varepsilon}
\def\Del{\Delta}
\def\G{\Gamma}

\def\dq{\dot q}

\def\hk{\widehat k}
\def\wht{\widehat t}
\def\hq{\widehat q}
\def\hlam{\widehat \lambda}
\def\hf{\widehat \f}
\def\hp{\widehat p}
\def\htau{\widehat \tau}
\def\hN{\widehat N}
\def\hM{\widehat M}

\def\tN{\widetilde{N}}
\def\tM{\widetilde{M}}

\def\bnu{\bar{\nu}}

\def\vh{\vec h}
\def\vH{\vec H}

\def\eqdef{:=}

\def\ts{\,{\sn \tau}\,}
\def\tc{\,{\cn \tau}\,}
\def\td{\,{\dn \tau}\,}
\def\ss{\,{\sn  p}\,}
\def\cc{\,{\cn  p}\,}
\def\dd{\,{\dn  p}\,}
\def\tsp{\,{\sn^2 \tau}\,}
\def\tcp{\,{\cn^2 \tau}\,}

\def\ssp{\,{\sn^2  p}\,}
\def\ccp{\,{\cn^2  p}\,}
\def\ddp{\,{\dn^2  p}\,}
\def\ssf{\,{\sn^4  p}\,}
\def\Eo{\,{\E(p)}\,}

\def\bu{\bar{u}}
\def\bv{\bar{v}}
\def\tu{\widetilde{u}}

\def\Ho{$(\mathbf{H1})$ } 
\def\Ht{$(\mathbf{H2})$ } 
\def\Hth{$(\mathbf{H3})$ } 
\def\Hf{$(\mathbf{H4})$ } 

\def\tconj{t_1^{\operatorname{conj}}}
\def\Ncut{N_{\operatorname{cut}}}
\def\Nconj{N_{\operatorname{conj}}}
\def\NMax{N_{\operatorname{Max}}}
\def\Nrest{N_{\operatorname{rest}}}
\def\Mcut{M_{\operatorname{cut}}}
\def\Mconj{M_{\operatorname{conj}}}
\def\MMax{M_{\operatorname{Max}}}
\def\Mrest{M_{\operatorname{rest}}}
\def\pal{p_1^{\a_1}}

\newcommand{\pder}[2]{\frac{\partial \, #1}{\partial \, #2} }
\newcommand{\der}[2]{\frac{d \, #1}{d\, #2} }
\newcommand{\pdder}[2]{\frac{\partial^2 #1}{\partial \, {#2}^2} }
\newcommand{\be}[1]{\begin{equation}\label{#1}}
\newcommand{\ee}{\end{equation}}
\newcommand{\eq}[1]{$(\protect\ref{#1})$}
\newcommand{\map}[3]{#1 \, : \, #2 \to #3}

\newcommand{\ddef}[1]{{#1}}
\newcommand{\hypoth}[2]{\medskip\noindent{#1}{\em #2}\medskip}
\newcommand{\restr}[2]{\left. #1 \right|_{#2}}
\newcommand{\partproof}[1]{\medskip\textbf{(#1)}{\ }}

\newcommand{\arth}{\operatorname{artanh}\nolimits}

\newtheorem{theorem}{Theorem}[section]
\newtheorem{lemma}{Lemma}[section]
\newtheorem{corollary}{Corollary}[section]
\newtheorem{proposition}{Proposition}[section]

\theoremstyle{remark}

\newcommand{\onefiglabel}[3]
{
\begin{figure}[htbp]
\begin{center}
\includegraphics[width=0.5\textwidth]{#1}
\\
\parbox[t]{0.5\textwidth}{\caption{#2}\label{#3}}
\end{center}
\end{figure}
}

\newcommand{\onefiglabelrotate}[4]
{
\begin{figure}[htbp]
\begin{center}
\includegraphics[width=0.5\textwidth, angle = #4]{#1}
\\
\parbox[t]{0.5\textwidth}{\caption{#2}\label{#3}}
\end{center}
\end{figure}
}

\newcommand{\onefiglabelsizen}[4]
{
\begin{figure}[htbp]
\begin{center}
\includegraphics[height=#4cm]{#1}
\\
\parbox[t]{0.7\textwidth}{\caption{#2}\label{#3}}
\end{center}
\end{figure}
}

\newcommand{\twofiglabel}[6]
{
\begin{figure}[htbp]
\includegraphics[width=0.47\textwidth]{#1}
\hfill
\includegraphics[width=0.47\textwidth]{#4}
\\
\parbox[t]{0.45\textwidth}{\caption{#2}\label{#3}}
\hfill
\parbox[t]{0.45\textwidth}{\caption{#5}\label{#6}}
\end{figure}
}

\newcommand{\twofiglabelh}[7]
{
\begin{figure}[htbp]
\includegraphics[height=#7cm]{#1}
\hfill
\includegraphics[height=#7cm]{#4}
\\
\parbox[t]{0.45\textwidth}{\caption{#2}\label{#3}}
\hfill
\parbox[t]{0.45\textwidth}{\caption{#5}\label{#6}}
\end{figure}
}

\newcommand{\twofiglabelsize}[8]
{
\begin{figure}[htbp]
\includegraphics[height=#4cm]{#1}
\hfill
\includegraphics[height=#8cm]{#5}
\\
\hfill
\parbox[t]{0.45\textwidth}{\caption{#2}\label{#3}}
\hfill
\parbox[t]{0.45\textwidth}{\caption{#6}\label{#7}}
\hfill
\end{figure}
}

\newcommand{\twotablabel}[6]
{
\begin{table}[htbp]
\hfill
\parbox[t]{0.45\textwidth}{#1}
\hfill
\parbox[t]{0.45\textwidth}{#4}
\hfill
\\
\hfill
\parbox[t]{0.45\textwidth}{\caption{#2}\label{#3}}
\hfill
\parbox[t]{0.45\textwidth}{\caption{#5}\label{#6}}
\hfill
\end{table}
}

\title{Cut time and optimal synthesis \\ in sub-Riemannian problem \\ on the group of motions of a 
plane\thanks{Work supported by Cariplo Foundation and 
Landau Network --- Centro Volta 2007/2008 Russia/CIS Fellowship
}} 

\author{
Yu. L. Sachkov\\[0.5cm]
Program Systems Institute\\
Pereslavl-Zalessky,  Russia\\
E-mail: sachkov@sys.botik.ru} 

\date{March 4, 2009}

\begin{document}

\maketitle

\begin{abstract}
A solution to the left-invariant sub-Riemannian problem on the group of motions (rototranslations) of a plane $\SE(2)$ is obtained. Local and global optimality of extremal trajectories is characterized.

Lower and upper bounds on the first conjugate time are proved.

The cut time is shown to be equal to the first Maxwell time corresponding to the group of discrete symmetries of the exponential mapping. An explicit global description of the cut locus is obtained. 

Optimal synthesis is described.
\end{abstract}



\newpage
\tableofcontents

\newpage
\section{Introduction}
This work is devoted to the study of the left-invariant sub-Riemannian problem on the group of motions of a plane. This problem can be stated as follows: given two unit vectors $v_0 = (\cos \th_0, \sin \th_0)$, $v_1 =  (\cos \th_1, \sin \th_1)$ attached respectively at two given points $(x_0, y_0)$, $(x_1, y_1)$ in the plane, one should find an optimal motion in the plane that transfers the vector $v_0$ to the vector $v_1$, see Fig.~\ref{fig:prob_state}. The vector can move forward or backward and rotate simultaneously. The required motion should be optimal in the sense of minimal length in the space $(x,y,\th)$, where $\th$ is the slope of the moving vector. 

\begin{figure}[htbp]
\begin{center}
\includegraphics[width=0.5\textwidth]{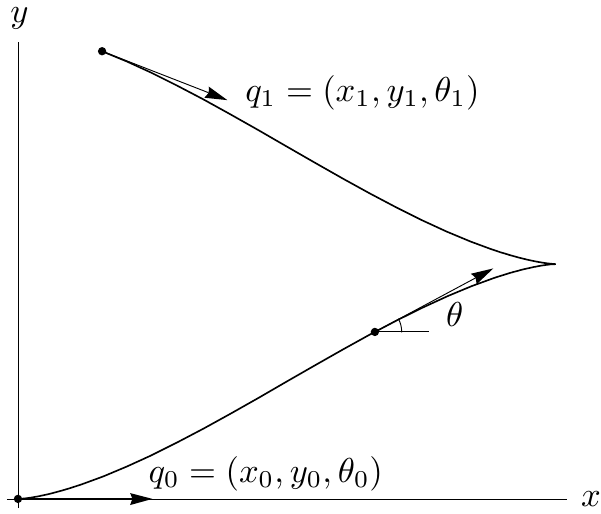}
\\
{\caption{Problem statement}\label{fig:prob_state}}
\end{center}
\end{figure}

The corresponding optimal control problem reads as follows: 
\begin{align}
&\dot x = u_1 \cos{\theta}, \quad  \dot y = u_1 \sin{\theta}, \quad  \dot \theta = u_2, \label{sys1} \\
&q = (x,y,\th) \in M \cong \R^2_{x,y} \times S^1_{\th}, 
\quad u = (u_1, u_2) \in \R^2,
\label{qu} \\
&q(0) = q_0 = (0,0,0), \qquad q(t_1) = q_1 = (x_1,y_1,\th_1),  \label{bound} \\
&l = \int_0^{t_1} \sqrt{u_1^2 + u_2^2} \, d t \to \min, \label{l}\\
\intertext{or, equivalently,}
&J = \frac 12 \int_0^{t_1} (u_1^2 + u_2^2) \, d t \to \min. \label{J}
\end{align}

This work is an immediate continuation of the previous work~\cite{max_sre2}. We use extensively the results obtained in that paper, and now we recall the most important of them. 

Problem~\eq{sys1}--\eq{J} is a left-invariant sub-Riemannian problem on the group of motions of a plane 
$\SE(2) = \R^2 \ltimes \SO(2)$.
The normal Hamiltonian system of
Pontryagin Maximum Principle   becomes triangular in appropriate coordinates on cotangent bundle $T^*M$, and its vertical subsystem is the equation of mathematical pendulum:
\begin{align}
&\dot \g = c, \quad \dot c = -\sin \g, \qquad (\g, c) \in C \cong (2 S^1_{\g}) \times \R_c, \label{ham_vert}\\
&\dot x = \sin \frac{\g}{2} \cos \th, \quad \dot y = \sin \frac{\g}{2} \sin \th, \quad \dot \th = - \cos \frac{\g}{2}. \label{ham_hor}
\end{align}
In elliptic coordinates $(\f, k)$, where $\f$ is the phase, and $k$ a reparametrized energy of pendulum~\eq{ham_vert}, the system ~\eq{ham_vert}, \eq{ham_hor} was integrated in 
Jacobi's functions~\cite{whit_watson}.
The equation of pendulum~\eq{ham_vert} has a discrete group of symmetries 
$G = \{ \Id, \eps^1, \dots, \eps^7\} = \Z_2 \times \Z_2 \times \Z_2$ generated by reflections in the axes of coordinates $\g$, $c$, and translations $(\g,c) \mapsto(\g+2 \pi, c)$. Action of the group $G$ is naturally extended to extremal trajectories $(x_t, y_t)$, this action modulo rotations is represented at Figs.~\ref{fig:eps1xy}--\ref{fig:eps5xy}.

\begin{figure}[htbp]
\includegraphics[height=3cm]{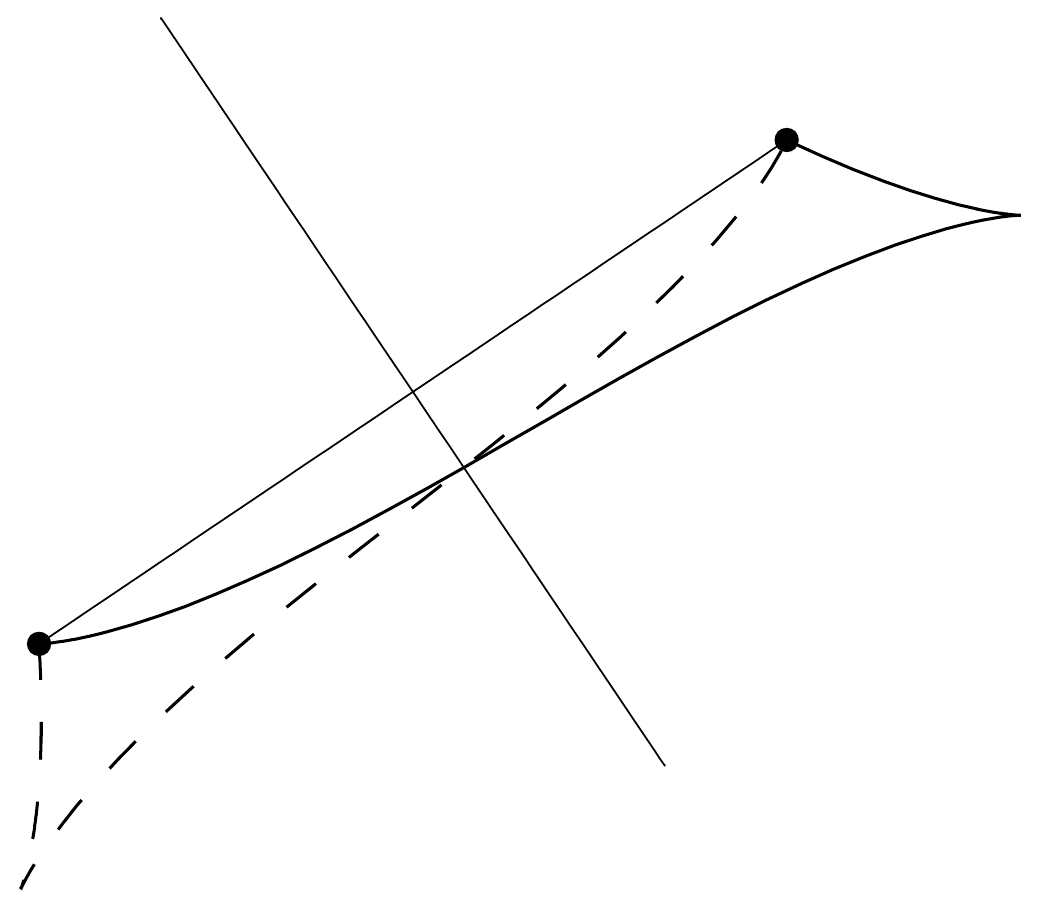}
\hfill
\includegraphics[height=2.5cm]{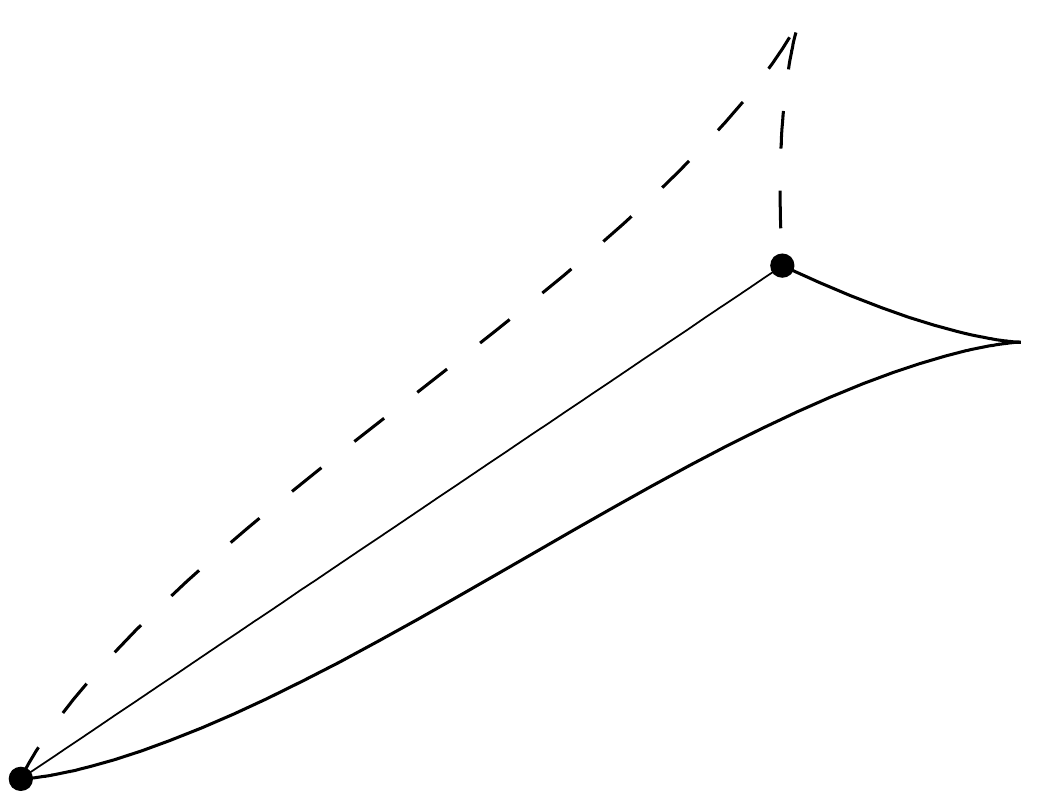}
\hfill
\includegraphics[height=1.7cm]{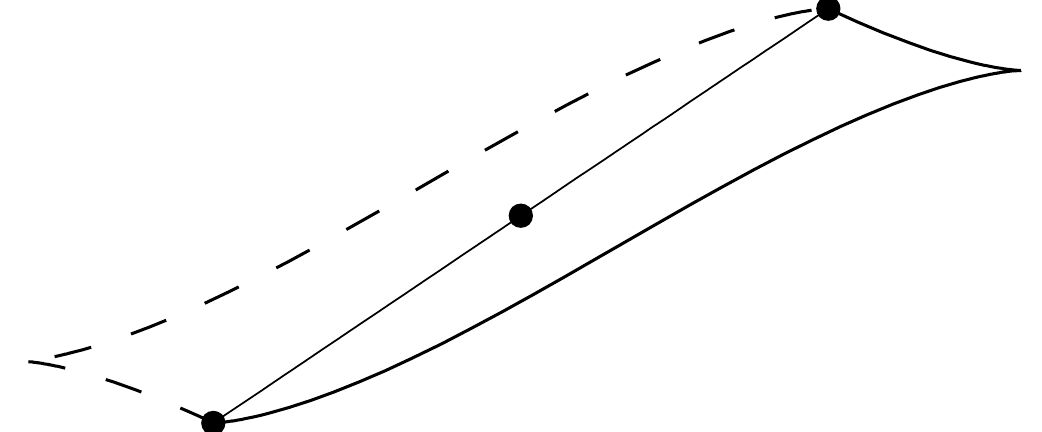}
\\
\parbox[t]{0.27\textwidth}{\caption{Action of $\eps^1$, $\eps^2$ on $(x_t, y_t)$}\label{fig:eps1xy}}
\hfill
\parbox[t]{0.27\textwidth}{\caption{Action of $\eps^4$, $\eps^7$ on $(x_t, y_t)$}\label{fig:eps4xy}}
\hfill
\parbox[t]{0.27\textwidth}{\caption{Action of $\eps^5$, $\eps^6$ on $(x_t, y_t)$}\label{fig:eps5xy}}
\end{figure}

Reflections $\eps^i$ are  symmetries of the exponential mapping $\map{\Exp}{N = C \times \R_+}{M}$, $\Exp(\lam,t) = q_t$.

The main result of work~\cite{max_sre2} is an upper bound on cut time 
$$
\tcut = \sup \{ t_1 > 0 \mid q_s \text{ is optimal for } s \in [0,t_1]\}
$$
along extremal trajectories $q_s$. It is  based on the fact that a sub-Riemannian geodesic cannot be optimal after a Maxwell point, i.e., a point where two distinct geodesics of equal sub-Riemannian length meet one another. A natural idea is  to look for Maxwell points corresponding to discrete symmetries of the exponential mapping. For each extremal trajectory $q_s = \Exp(\lam,s)$, we described Maxwell times $t_{\eps^i}^n(\lam)$, $i = 1, \dots, 7$, $n = 1, 2, \dots$, corresponding to discrete symmetries $\eps^i$. The following upper bound is the main result of work~\cite{max_sre2}:
\be{tcutbound}
\tcut(\lam) \leq \tt(\lam), \qquad \lam \in C,
\ee
where $\tt(\lam) = \min(t_{\eps^i}^1(\lam))$ is the first Maxwell time corresponding to the group of symmetries $G$. We recall the explicit definition of the function $\tt(\lam)$ below in Eqs.~\eq{ttC1}--\eq{ttC5}. 

In this work we obtain a complete solution to problem~\eq{sys1}--\eq{J}. 

First we study the local optimality of sub-Riemannian geodesics (Section~\ref{sec:conj}). 
We show that extremal trajectories corresponding to oscillating pendulum~\eq{ham_vert} do not have conjugate points, thus they are locally optimal forever. In the case of rotating pendulum we prove that the first conjugate time is bounded from below and from above by the first Maxwell times $t_{\eps^2}^1$ and $t_{\eps^5}^1$ respectively. For critical values of energy of the pendulum, there are no conjugate points. 

In Sections~\ref{sec:cut}--\ref{sec:sol} we study the global optimality of geodesics. We construct decompositions of the preimage and image of exponential mapping into smooth stratas of dimensions $0, \dots, 3$ and prove that the exponential mapping transforms these stratas diffeomorphically. As a consequence, we show that inequality~\eq{tcutbound} is in fact an equality. Further, we describe explicitly the Maxwell strata and cut locus in the problem. For some special terminal points $q_1$, we provide explicit optimal solutions. Finally, we present plots of the conjugate and cut loci, and of sub-Riemannian spheres.

\section{Conjugate points}
\label{sec:conj}

In this section we obtain bounds on conjugate time in the sub-Riemannian problem on $\SE(2)$, see Th.~\ref{th:conjC}.

\subsection{General facts}

First we recall some known facts from the theory of conjugate points in optimal control problems. For details see, e.g., \cite{notes, cime, sar}.

Consider an optimal control problem of the form
\begin{align}
&\dq = f(q,u), \qquad q \in M, \quad u \in U \subset \R^m, \label{sys} \\
&q(0) = q_0, \qquad q(t_1) = q_1, \qquad t_1 \text{ fixed}, \label{bound1} \\
&J = \int_0^{t_1} \f(q(t),u(t)) \, dt \to \min, \label{J1}
\end{align}
where $M$ is a  finite-dimensional analytic manifold,  $f(q,u)$ and $\f(q,u)$ are respectively analytic in $(q,u)$  families of vector fields   and   functions  on $M$ depending on the control parameter $u \in U$, and $U$ an open subset of $\R^m$. Admissible controls are $u(\cdot) \in L_{\infty}([0, t_1],U)$, and admissible trajectories $q(\cdot)$ are Lipschitzian.
Let 
$$
h_u(\lam) = \langle \lam, f(q,u)\rangle - \f(q,u), 
\qquad \lam \in T^*M, \quad q = \pi(\lam) \in M, \quad u \in U,
$$
be the \ddef{normal Hamiltonian of PMP} for the problem~\eq{sys}--\eq{J1}. Fix a triple $(\tu(t), \lam_t, q(t))$ consisting of a normal extremal control $\tu(t)$, the corresponding extremal $\lam_t$, and the extremal trajectory $q(t)$ for the problem~\eq{sys}--\eq{J1}.

Let the following hypotheses hold:

\hypoth{\Ho}{For all $\lam \in T^*M$ and  $u \in U$, the quadratic form $\ds \pdder{h_u}{u}(\lam)$ is negative definite.}

\hypoth{\Ht}
{For any $\lam \in T^* M$, the function $u \mapsto h_u(\lam)$, $u \in U$, has a maximum point $\bu(\lam) \in U$:
$$
h_{\bu(\lam)}(\lam) = \max_{u \in U} h_u(\lam), \qquad \lam \in T^*M.
$$}%

\hypoth{\Hth}
{The extremal control $\tu(\cdot)$ is a corank one critical point of the endpoint mapping.}%

\hypoth{\Hf}
{All trajectories of the Hamiltonian vector field $\vH(\lam)$, $\lam \in T^*M$, are continued for $t \in [0, +\infty)$. 
}

An instant $t_* > 0$ is called a \ddef{conjugate time} (for the initial instant $t = 0$) along the extremal $\lam_t$ if the restriction of the second variation of the endpoint mapping to the kernel of its first variation   is degenerate, see~\cite{notes} for details. In this case the point $q(t_*) = \pi(\lam_{t_*})$ is called \ddef{conjugate} for the initial point  $q_0$ along the extremal trajectory $q(\cdot)$.

Under hypotheses \Ho--\Hf, we have the following:

\begin{enumerate}
\item
Normal extremal trajectories lose their local optimality (both strong and weak) at the first conjugate point, see~\cite{notes}.
\item
An instant $t > 0$ is a conjugate time iff the exponential mapping $\Exp_t = \pi \circ e^{t \vH}$ is degenerate, see~\cite{cime}.
\item
Along each normal extremal trajectory, conjugate times are isolated one from another, see~\cite{sar}.
\end{enumerate}

We will apply  the following statement for the proof of absence of conjugate points via homotopy.

\begin{proposition}[Corollaries 2.2, \ 2.3~\cite{el_conj}]\label{propos:cor2.2-2.3}
Let $(u^s(t), \lam^s_t)$, $t \in [0, + \infty)$, $s \in [0, 1]$, be a continuous in parameter $s$ family of normal extremal pairs in the optimal control problem~\eq{sys}--\eq{J1} satisfying hypotheses \Ho--\Hf.

\begin{itemize}
\item[$(1)$]
Let $s \mapsto t_1^s$ be a continuous function, $s \in [0, 1]$, $t_1^s \in (0, + \infty)$.
Assume that for any $s \in [0, 1]$ the instant $t = t_1^s$ is not a conjugate time along the extremal $\lam_t^s$.

If the extremal trajectory $q^0(t) = \pi(\lam_t^0)$, $t \in (0, t_1^0]$, does not contain conjugate points, then the extremal trajectory $q^1(t) = \pi(\lam_t^1)$, $t \in (0, t_1^1]$, also does not contain conjugate points.

\item[$(2)$]
Let for any $s \in [0, 1]$ and $T > 0$ the extremal $\lam_t^s$ have no conjugate points for $t \in (0, T]$.
Then for any $T> 0$, the extremal $\lam_t^1$ also has no conjugate points for $t \in (0, T]$.
\end{itemize}
\end{proposition}

One easily checks that the sub-Riemannian problem~\eq{sys1}--
\eq{J} satisfies all hypotheses \Ho--\Hf, so the results cited in this subsection are applicable to this problem.

We denote the first conjugate time along an extremal trajectory $q(t) = \Exp(\lam,t)$ as $\tconj(\lam)$.

\subsection{Conjugate points for the case of oscillating pendulum}
\label{subsec:conjC1}
In this subsection we assume that $\lam \in C_1$ and prove that the corresponding extremal trajectories do not contain conjugate points, see Th.~\ref{th:conjC1}.

Using the parametrization of extremal trajectories obtained in Subsec.~3.3~\cite{max_sre2}, we compute explicitly Jacobian of the exponential mapping:
\begin{align}
&J = \frac{\partial(x_t, y_t, \th_t)}{\partial(t, \f, k)} = 
\frac{4}{k^3 (1-k^2) (1 - k^2 \ssp \tsp)} J_1, \nonumber\\
&p = t/2, \qquad \tau = \f + t/2, \label{ptauC1}\\
&J_1(\tau, p, k) = v_1 \tsp + v_2 \tcp, \label{J1v12}\\
&v_1 = (1-k^2) (p - \Eo)(\Eo - (1 - k^2)p), \nonumber\\
&v_2 = (p - \Eo)(\Eo - (1 - k^2)p)+k^2 \cc \dd (2\Eo + (k^2-2)p)\ss \nonumber\\
&\qquad + k^2((\Eo - p)(\Eo-(1-k^2)p)-k^2 )\ssp + k^4 \ssf, \nonumber
\end{align}
so that $\sgn J = \sgn J_1$.

\subsubsection{Preliminary lemmas}

\begin{lemma}
\label{lem:v1}
For any $k \in (0,1)$ and $p> 0$ we have $v_1(p,k) > 0$.
\end{lemma}
\begin{proof}
The statement follows from the relations 
\be{p-E>0}
p - \Eo = k^2 \int_0^p \sn^2 t \, dt > 0, 
\qquad 
\Eo - (1-k^2) p  = k^2 \int_0^p \cn^2 t \, dt > 0.
\ee
\end{proof}

\begin{lemma}
\label{lem:v2}
For any $k \in (0,1)$, $n \in \N$, and $\tau \in \R$ we have $J_1(\tau, 2 Kn,k) > 0$.
\end{lemma}
\begin{proof}
If $p = 2 Kn$, $n \in \N$, then $v_2(p,k) = (p - \Eo)(\Eo-(1 - k^2)p)   > 0$ by inequalities~\eq{p-E>0}.

By virtue of Lemma~\ref{lem:v1} and decomposition~\eq{J1v12}, we obtain the inequality $J_1(\tau, 2 Kn,k) > 0$.
\end{proof}

\begin{lemma}
\label{lem:Jk->0}
$\forall \ p_1 > 0 \quad \exists \  \hk \in(0,1) \quad \forall \ k \in(0, \hk) \quad \forall p \in(0, p_1) \quad \forall \tau \in \R$ $J_1(\tau, p, k) > 0$.
\end{lemma}
\begin{proof}
The statement of the lemma follows from the Taylor expansions:
\begin{align}
&J_1 = \frac{k^4}{16}(4p^2 - \sin^2 2p) + o(k^4), \qquad k \to 0, 
\label{J1k->0} \\
&J_1 = \frac{1}{3}k^4 p^4 + o(k^2+p^2)^4, \qquad k^2 + p^2 \to 0. 
\label{J1kp->0}
\end{align}
By contradiction, if the statement is not verified, then there exists a converging sequence $(\tau_n, p_n, k_n) \to (\tau_0, p_0, 0)$ such that $J(\tau_n, p_n, k_n) \leq 0$ for all $n \in N$.
If $p_0 > 0$, then a standard calculus argument yields contradiction with~\eq{J1k->0}. And if $p_0=0$, then similarly one obtains a contradiction with~\eq{J1kp->0}.
\end{proof}

\subsubsection{Absence of conjugate points in $C_1$}
\begin{theorem}
\label{th:conjC1}
If $\lam \in C_1$, then the extremal trajectory $q(t) = \Exp(\lam,t)$, $t>0$, does not contain conjugate points.
\end{theorem}
\begin{proof}
We choose any $\hlam \in C_1$, $\wht > 0$, and prove that the extremal trajectory $\hq(t) = \Exp(\hlam, t)$ does not contain conjugate points for $t \in (0, \wht\ ]$.

Find the elliptic coordinates $(\hk, \hf)$ corresponding to the covector $\hlam \in C_1$ according to Subsec.~3.2~\cite{max_sre2}, and let $\hp = \wht /2$, $\htau = \hf + \hp$. Find $n \in \N$ such that $p_1 = 2 K(\hk) n > \hp$. Choose the following continuous curve in the plane $(k,p)$:
$$
\{(k^s, p^s) \mid s \in [0, 1]\}, \qquad k^s = s \hk, \quad p^s = 2 K(k^s) n,
$$
with the endpoints $(k^0, p^0) = (0, \pi n)$ and $(k^1, p^1) = (\hk, 2 K(\hk) n)$.

Consider the following family of extremal trajectories:
\begin{align*}
&\g^s = \{q^s(t) = \Exp(\f^s, k^s, t) \mid t \in [0, t^s]\}, \qquad s \in [0, 1],\\
&t^s = 2 p^s, \qquad \f^s = \htau - p^s.
\end{align*}

The endpoint $q^s(t^s)$ of each trajectory $\g^s$, $s \in [0, 1]$, corresponds to the values of parameters $(\tau, p, k) = (\htau, 2 K(k^s)n, k^s)$. Thus Lemma~\ref{lem:v2} implies that for any $s \in [0, 1]$ the endpoint $q^s(t^s)$ is not a conjugate point.

Further, Lemma~\ref{lem:Jk->0} states that 
\be{Jtaupk0}
\exists \ k_0 \in (0, \hk) \quad \forall \tau \in \R \quad \forall \ p \in(0, p_1) \qquad J(\tau, p, k) > 0.
\ee
Denote $s_0 = k_0/\hk \in(0,1)$, so that $k^{s_0} = k_0$. Condition~\eq{Jtaupk0} means that the extremal trajectory $\g^{s_0}$ does not contain conjugate points for all $t \in [0, t^{s_0}]$.

Then Proposition~\ref{propos:cor2.2-2.3} yields that for any $s \in [s_0, 1]$, the extremal trajectory $q^s(t)$ does not contain conjugate points for all $t \in [0, t^s]$. In particular, the trajectory $\hq(t) = q^1(t)$, $t \in (0, \wht]$, is free of conjugate points. 
\end{proof}

So we proved that extremal trajectories $q(t) = \Exp(\lam,t)$ with $\lam \in C_1$ (i.e., corresponding to oscillating pendulum) are locally optimal at any segment $[0, t_1]$, $t_1 >0$.

\subsection{Conjugate points for the case of rotating pendulum}
\label{subsec:conjC2}
In this subsection we obtain bounds on conjugate points in the case $\lam \in C_2$.

Using the formulas for extremal trajectories of Subsec.~3.3~\cite{max_sre2}, we get:
\begin{align}
&J = \frac{\partial(x_t, y_t, \th_t)}{\partial(t,\f,k)} = - \frac{4k}{(1-k^2)(1-k^2 \ssp  \tsp)} J_2, \nonumber\\
&p = t/(2k), \qquad \tau = \p + t/(2k) = (2 \f + t)/(2k), \label{ptauC2}\\
&J_2 = \a \tsp + \b \tcp, \label{J2alphabeta} \\
&\a = (1-k^2) \ss \a_1, \label{alphaalpha1}\\
&\a_1 = \cc \dd (p - 2 \Eo) + \ss(\ddp + \Eo(p-\Eo)), \label{alpha1}\\
&\b = f_1(p) \b_1, \qquad \b_1 = \cc \Eo - \dd \ss, \label{betabeta1}
\end{align}
where $f_1(p,k) =  \cc (\E(p)-p)   - \dd\ss$, see Eq.~(5.12)~\cite{max_sre2}. 

\subsubsection{Preliminary lemmas}

Recall that we denoted the first positive root of the function $f_1(p)$ by $p_1^1(k)$, see Lemma~5.3~\cite{max_sre2}. 

\begin{lemma}
\label{lem:J2p=p11}
If $k\in(0,1)$ and $p = p_1^1(k)$, then $\a(p,k) > 0$.

If additionally $\ts \neq 0$, then $J_2 > 0$ and $J < 0$.
\end{lemma}
\begin{proof}
In terms of the auxiliary function 
\be{phipk}
\f(p,k) = \ss \dd - (2 \Eo - p) \cc,
\ee
 we have a decomposition
\be{alpha1dphi}
\a_1 = \dd \f(p) + \ss \Eo (p - \Eo).
\ee

Let $k \in (0, 1)$ and $p = p_1^1(k)$. Then $f_1(p) = 0$, i.e., $\ss \dd = \cc(\Eo - p)$. Thus $\f(p) = \cc(\Eo-p) - (2 \Eo - p) \cc = - \Eo \cc$. By virtue of Cor.~5.1~\cite{max_sre2},
we have $\cc < 0$, so $\f(p) > 0$. Moreover, $\ss > 0$. Then decomposition~\eq{alpha1dphi} yields $\a_1(p) > 0$, consequently, $\a(p) > 0$.

If additionally $\ts \neq 0$, then it is obvious that $J_2 > 0$ and $J < 0$.
\end{proof}

\begin{lemma}
\label{lem:alphak->0}
$\exists \ \hk \in (0,1) \quad \forall \ k \in (0, \hk) \quad \forall p \in (0, p_1^1] \qquad \a(p,k) > 0$.
\end{lemma}
\begin{proof}
The statement of this lemma follows by the argument used in the proof of Lemma~\ref{lem:Jk->0} from the Taylor expansions
\begin{align*}
&\a = \sin p(\sin p - p \cos p) + o(1), \qquad k \to 0, \\
&\a = \frac{p^4}{3} + o(p^2+k^2)^2, \qquad p^2+k^2 \to 0.
\end{align*}
\end{proof}

\begin{lemma}
\label{lem:beta1}
$\forall \ k \in(0,1) \quad \forall p \in (0, 2 K] \qquad \b_1(p,k) < 0$.
\end{lemma}
\begin{proof}
Since $(\b_1(p)/\cc)' = -(1-k^2) \ssp / \ccp$, the function $\b_1(p)/\cc$ decreases at the segments $p \in[0, K)$ and $p \in (K, 2 K]$. 

We have $\b_1(0) = 0$, thus $\b_1(p)/\cc < 0$, so $\b_1(p) < 0$ for $p \in (0, K)$.

Further, $\b_1(K) = - \sqrt{1-k^2} < 0$.

Since $\b_1(p)/\cc \to + \infty$ as $p \to K + 0$, and $\b_1(2K)/\cn(2K) = \E(2K) > 0$, we have $\b_1(p)/\cc >0$, so $\b_1(p) < 0$ for $p \in(K, 2 K]$.
\end{proof}

\begin{lemma}
\label{lem:J2z=0}
Let $k \in (0,1)$.
\begin{itemize}
\item[$(1)$]
Let $\ts = 0$.
Then $J_2(\tau,p,k) > 0$ for $p \in (0, p_1^1)$, and $J_2(\tau,p,k) = 0 $ for $p = p_1^1$.
\item[$(2)$]
Let $\ts \neq 0$. Then $J_2(\tau,p,k) > 0$ for $p \in (0, p_1^1]$.
\end{itemize}
\end{lemma}
\begin{proof}
If $p\in (0,p_1^1)$, then $f_1(p,k) < 0$ (Cor.~5.1~\cite{max_sre2}), and $\b_1(p,k) < 0$ (Lemma~\ref{lem:beta1}), thus $\b(p,k) = f_1(p,k) \b_1(p,k) > 0$.

(1)
Let $\ts = 0$. If $p \in (0, p_1^1)$, then $J_2(\tau,p,k) = \b(p,k) > 0$.
And if $p = p_1^1$, then $f_1(p,k) = 0$, thus $J_2(\tau,p,k) = \b(p,k) = 0$.

(2)
Let $\ts \neq 0$. 

(2.a) We prove that the function $\f(p)$ given by~\eq{phipk} satisfies the inequality
$$
\f(p) > 0 \qquad \forall \ p \in (0,K].
$$
First, $\f(p) = p^3/3 + o(p^3) > 0 $ as $ p \to + 0$. Second,
$$
(\f(p)/\cc)' = \ddp \ssp / \ccp > 0 \qquad \forall \  p \in (0,K).
$$
Thus $\f(p) > 0$ for $p \in (0,K)$. And if $p = K$, then $\f(p) = \sqrt{1-k^2}> 0$.

(2.b)
By virtue of the decomposition $\f(p) = - f_1(p) - \E(p) \cc$, we get the inequality $\f(p) > 0 $ for all $p \in (K, p_1^1]$. We proved that 
$$
\f(p) > 0 \qquad \forall \ p \in (0,p_1^1].
$$

(2.c)
In view of~\eq{alpha1dphi}, we obtain that $\a_1(p) > 0$ for $p \in (0, p_1^1]$. Then Eq.~\eq{alphaalpha1} yields $\a(p) > 0$ for $p \in (0, p_1^1]$. Finally, Eq.~\eq{J2alphabeta} gives $J_2 > 0$ for $p \in (0, p_1^1]$.
\end{proof}

\begin{lemma}
\label{lem:J2z>0}
$\forall \ z \in (0,1] \quad \exists \ \hk \in (0,1) \quad \forall \ k \in (0, \hk) \quad \forall \ p \in (0, p_1^1]$ we have $J_2(z, p, k) > 0$.
\end{lemma}
\begin{proof}
Fix any $z \in (0,1]$. By Lemma~\ref{lem:alphak->0}, 
$$
\exists \ \hk \in (0,1) \quad \forall \ k \in (0, \hk) \quad \forall \ p \in (0, p_1^1] \qquad \a(p, k) > 0.
$$
But if $p \in (0, p_1^1]$, then $p \in (0, 2 K]$, thus $\b_1(p,k) < 0$ by Lemma~\ref{lem:beta1}, so $\beta(p,k) > 0$ by Cor.~5.1~\cite{max_sre2}. 

Then the inequalities $\a(p,k) > 0$, $\b(p,k) > 0$ imply the required inequality $J_2(\tau,p,k) > 0$.
\end{proof}

\subsubsection{Conjugate points in $C_2$}
First we obtain a lower bound on the first conjugate time. It will play a crucial role in the subsequent analysis of the global structure of the exponential mapping in Sections~\ref{sec:cut}, \ref{sec:bound}.

\begin{theorem}
\label{th:conjC2}
If $\lam \in C_2$, then $\tconj(\lam) \geq 2 k p_1^1(k)$.
\end{theorem}
\begin{proof}
Given any $\lam \in C_2$, compute the corresponding elliptic coordinates $(\f, k)$. If additionally we have $t>0$, find the corresponding parameters $p = t/(2k)$, $\tau = \f/k + t/(2k)$ and denote $z = \tsp$.

We should prove that for any $\lam \in C_2$ the interval $t \in (0, 2 k p_1^1(k))$ does not contain conjugate times for the extremal trajectory $q(t) = \Exp(\lam,t)$.

Take any $\lam^1 \in C_2$ and denote the corresponding elliptic coordinates $(\f^1, k^1)$. For $t^1 = 2 k^1 p_1^1(k^1)$ we  denote the corresponding parameters $p^1$, $\tau^1$, $z^1$. In order to prove that the extremal trajectory $q^1(t) = \Exp(\lam^1,t)$ does not have conjugate points at the interval $t \in (0, t^1)$, we show that
$$
p \in (0,p^1) \then J_2(z^1, p, k^1) > 0 \then J(z^1, p, k^1) < 0.
$$

\medskip
(1) Assume first that $z^1 = \sn^2(\tau^1, k^1) \neq 0$, i.e., $z^1 \in (0, 1]$. We prove that in this case 
$$
p \in (0,p^1] \then J(z^1, p, k^1) < 0.
$$

Consider the following continuous curve in the space $(z,p,k)$:
$$
\{(z^1, p^s, k^s) \mid s \in (0,1]\}, \qquad k^s = s k^1, \quad p^s = p_1^1(k^s).
$$
The corresponding curve in the space $(\tau,p,k)$ is
$$
\{(\tau^s, p^s, k^s) \mid s \in (0,1]\}, \qquad \tau^s = F(\am(\tau^1, k^1), k^s),
$$
and in the space $(t,\f,k)$ is
$$
\{(t^s, \f^s, k^s) \mid s \in (0,1]\}, \qquad t^s = 2 k^s p^s, \quad \f^s = (\tau^s - p^s)k^s.
$$
Let $\lam^s = (\f^s, k^s)$, $s \in (0, 1]$, be the corresponding curve in $C_2$, and consider the continuous one-parameter family of extremal trajectories
$$
q^s(t) = \Exp(\lam^s, t), \qquad t \in [0, t^s], \quad s \in (0, 1].
$$

For any $s \in (0,1]$, if $t = t^s$, then by Lemma~\ref{lem:J2p=p11} we have $J_2(z^1, p^1_1(k^s), k^s) < 0$, i.e., the terminal instant $t = t^s$ is not a conjugate time along the extremal trajectory $q^s(t)$.

Further, by Lemma~\ref{lem:J2z>0}, for $z^1 \in (0,1]$ 
$$
\exists \ \hk \in (0, 1) \quad \forall \ k \in (0, \hk) \quad \forall \ p \in (0, p_1^1(k)]
\qquad J_2(z^1, p, k) > 0.
$$
Consequently,  there exists $s_0 \in (0,1)$ such that the whole trajectory $q^{s_0}(t)$, $t \in (0, t^{s_0}]$, is free of conjugate points.

Then Propos.~\ref{propos:cor2.2-2.3} implies that the trajectory $q^1(t)$, $t \in (0, t^1]$, also does not contain conjugate points.

We proved that if $z^1 \neq 0$, then the trajectory $q^1(t) = \Exp(\lam^1,t)$, $t \in (0, t^1]$, does not have conjugate points.

\medskip
(2)
Now consider the case $z^1 = \sn^2(\tau^1, k^1) = 0$. Then Lemma~\ref{lem:J2z=0} states that the terminal instant $t = 2 k^1 p_1^1(k^1)$ is a conjugate point. We prove that all the less instants are not conjugate.

Since conjugate points are isolated one from another at each extremal trajectory, there exists $p < p_1^1(k^1)$ arbitrarily close to $p_1^1(k^1)$ such that the corresponding time $t = 2 k^1 p$ is not conjugate.

Consider the continuous curve in the space $(z,p,k)$:
$$
\{(z_s, p, k^1)\mid s \in [0, \eps) \}, \qquad z_s = s z^1.
$$
By item (1) of this proof, there exists $\eps > 0$ such that for any $s \in (0, \eps)$ the extremal trajectory $q^s(t)$, $t \in (0, t^s]$, $t^s = 2 k^1 p$, does not have conjugate points. By Propos.~\ref{propos:cor2.2-2.3}, for $s =0$ the initial extremal trajectory $q^0(t)$, $t \in (0, t^0]$, also does not contain conjugate points.  The endpoint $t^0 = 2 k^1 p$ can be chosen arbitrarily close to $t^1 = 2k^2p_1^1(k^1)$, so the initial extremal trajectory does not have conjugate points for $t \in (0, t^1)$.
\end{proof}

Now we obtain the final result on the first conjugate time in the domain $C_2$ --- the following two-side bound.

\begin{theorem}
\label{th:tconjC2fin}
If $\lam \in C_2$, then 
\be{tconjbound1}
2kp_1^1(k) \leq \tconj(\lam) \leq 4 k K(k).
\ee
\end{theorem}
\begin{proof}
We proved in Th.~\ref{th:conjC2} that $2kp_1^1(k) \leq \tconj(\lam)$; moreover, if $t \in (0, 2 k p_1^1)$, then $J<0$.

Let $t = 4 k K$, then $p = 2 K$, thus $\a = 0$, $f_1 = p - \Eo>0$, $\b_1 = - \Eo < 0$, so
$$
J = - \frac{4k}{1-k^2} J_2 = \frac{4k}{1-k^2} \tcp \Eo(p - \Eo) \geq 0.
$$
It follows that for any $\lam \in C_2$ the function $t \mapsto J$ has a root at the segment $t \in [2kp_1^1, 4kK]$. Consequently, also the first root $\tconj \in [2kp_1^1, 4kK]$.
\end{proof}

One can show that the bound~\eq{tconjbound1} can be a little bit improved. The precise bound on the first conjugate time is
\be{tconjbounds}
2kp_1^1(k) \leq \tconj(\lam) \leq \g(k) = \min(4 k K, 2 k \pal(k)), 
\ee
where $p = \pal(k)$ is the first positive root of the equation $\a_1(p) = 0$, and the function $\a_1$ is given by Eq.~\eq{alpha1}. One can show that $\g(k) = 4 k K$ for $k \in (0, k_0]$ and $\g(k) = 2 k \pal(k)$ for $[k_0, 1)$, where $k_0 \approx 0.909$ is the unique root of the equation $2E(k) - K(k) = 0$, see Proposition~11.5~\cite{el_max}. Thus for $k \in (k_0,1)$ the bound~\eq{tconjbound1} is not exact and can be replaced by the following exact one:
\be{tconjbound2}
2kp_1^1(k) \leq \tconj(\lam) \leq 2 k \pal(k), \qquad k \in (k_0, 1).
\ee
The bound~\eq{tconjbound1} is illustrated at Figs.~\ref{fig:tconjk08}, \ref{fig:tconjkk0}; and the bound~\eq{tconjbound2} --- at Fig.~\ref{fig:tconjk099}. The exact bounds~\eq{tconjbounds} are plotted at Fig.~\ref{fig:tconjbounds}.

\twofiglabel
{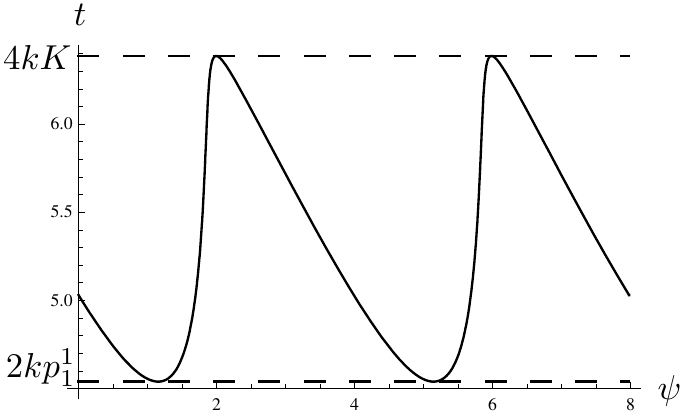}{Plot of  $\tconj(\psi,k)$, $k = 0.8 < k_0$}{fig:tconjk08}
{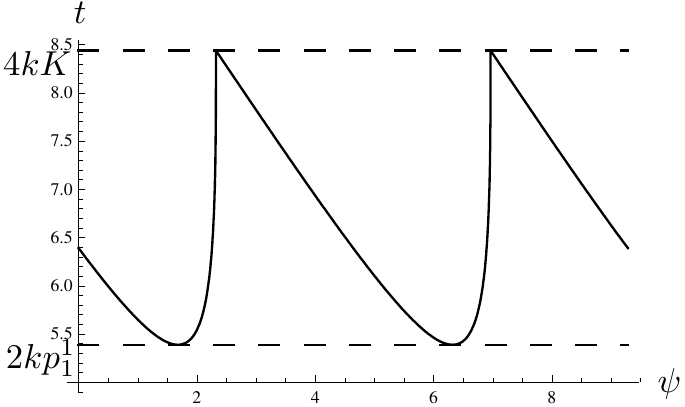}{Plot of  $\tconj(\psi,k)$, $k = k_0$}{fig:tconjkk0}

\twofiglabelh
{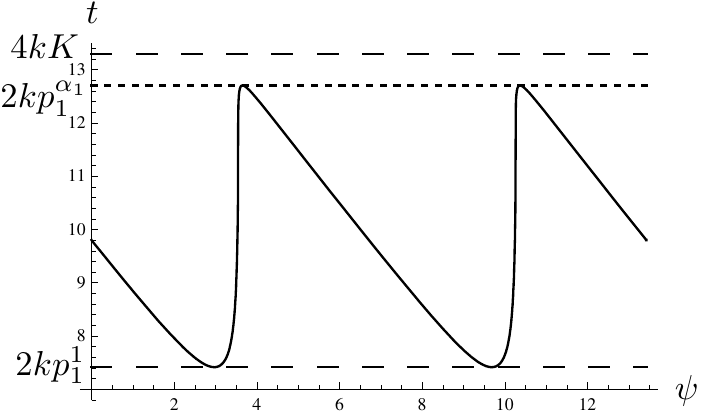}{Plot of  $\tconj(\psi,k)$, $k = 0.99 > k_0$}{fig:tconjk099}
{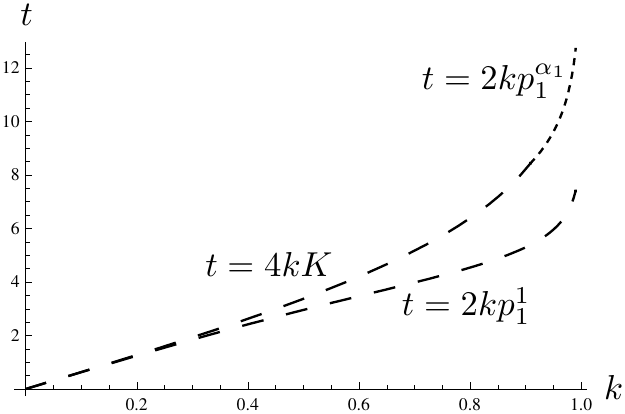}{Bounds of  $\tconj(\psi,k)$}{fig:tconjbounds}
{4}

\begin{proposition}
\label{propos:conjC2ts=0}
Let   $\lam \in C_2$ and  $\tau = (2 \f +  2 k p_1^1)/(2k)$.
\begin{itemize}
\item[$(1)$]
If $\ts = 0$, then $\tconj(\lam) = 2 k p_1^1$.
\item[$(2)$]
If $\ts \neq 0$, then $\tconj(\lam) \in (2 k p_1^1, 4 k K]$.
\end{itemize}
\end{proposition}
\begin{proof}
Notice first that by Th.~\ref{th:conjC2}, the interval $(0,2 k p_1^1)$ does not contain conjugate times. Then items (1), (2) of this proposition follow directly from the corresponding items of Lemma~\ref{lem:J2z=0}, and from Th.~\ref{th:tconjC2fin}.
\end{proof}

Now we apply Proposition~\ref{propos:conjC2ts=0} to fill the gap appearing in Th.~5.3~\cite{max_sre2}.

\begin{theorem}
\label{th:tcut_bound_fin}
There holds the bound
\be{tcutt_fin}
\tcut(\lam) \leq \tt(\lam) \qquad \forall \ \lam \in C.
\ee 
\end{theorem}
\begin{proof}
If $\lam \in C \setminus C_2$ or $(\lam,t) \in N_2$, $p = p_1^1(k)$, $\ts \neq 0$, then Th.~5.3~\cite{max_sre2}
{}  gives the required bound. And if $(\lam,t) \in N_2$, $p = p_1^1(k)$, $\ts = 0$, then
 the bound is provided by Proposition~\ref{propos:conjC2ts=0} since $\tcut(\lam) \leq \tconj(\lam)$ (local optimality is lost after or simultaneously with the global optimality).
\end{proof}

\subsection{Conjugate points for the cases of critical energy of pendulum}
\label{subsec:conjC345}

The subset $C_3 \cup C_4 \cup C_5$ of the cylinder $C$ is the boundary of the domain $C_1$, see Fig.~2~\cite{max_sre2}.
{} Thus absence of conjugate points for the corresponding extremal trajectories follows by limit passage from $C_1$.

\begin{theorem}
\label{th:conjC345}
If $\lam \in C_3 \cup C_4 \cup C_5$, then the corresponding extremal trajectory $q(t) = \Exp(\lam,t)$ does not have conjugate points for $t> 0$.
\end{theorem}
\begin{proof}
For any $\lam \in C_3 \cup C_4 \cup C_5$, there exists a continuous curve $\lam^s$, $s \in [0,1]$, such that $\lam^s \in C_1$ for $s \in [0, 1)$ and $\lam^1 = \lam$. By Theorem~\ref{th:conjC1}, the trajectories $q^s(t) = \Exp(\lam^s,t)$, $t > 0$, are free of conjugate points. Then Propos.~\ref{propos:cor2.2-2.3} implies the same for the trajectory $q^1(t) = q(t)$.
\end{proof}

\subsection{General bound of conjugate points}
\label{subsec:conjC}

We collect the bounds on the first conjugate time obtained in the previous subsections.

\begin{theorem}
\label{th:conjC}
\begin{itemize}
\item[$(1)$]
If $\lam \in C_1 \cup C_3 \cup C_4 \cup C_5$, then $\tconj(\lam) = + \infty$.
\item[$(2)$]
If $\lam \in C_2$, then $\tconj(\lam) \in [ 2 k p_1^1, 4 k K]$.
\item[$(3)$]
Consequently,
$\tconj(\lam) \geq \tt(\lam)$ for all $\lam \in C$.
\end{itemize}
\end{theorem}

\section{Exponential mapping of open stratas and cut time}
\label{sec:cut}
In this section we 
prove that $\tcut(\lam) = \tt(\lam)$ for any $\lam \in C$ and 
describe the optimal synthesis on an open dense subset of the state space.

\subsection{Decompositions in  preimage and image of exponential mapping}
Denote $\hM = M \setminus \{q_0\}$.
For any point $q \in \hM$ there exists an optimal trajectory $q(s) = \Exp(\lam, s)$ such that $q(t) = q$, $(\lam,t) \in N$. Thus the mapping $\map{\Exp}{N}{\hM}$ is surjective. By Th.~5.4~\cite{max_sre2}, the optimal instant $t$ satisfies the inequality $t \leq \tt(\lam)$.  So the restriction 
\begin{align*}
&\map{\Exp}{\hN}{\hM},\\
&\hN = \{(\lam,t) \in N \mid t \leq \tt(\lam)\},
\end{align*}
is surjective as well.

\subsubsection{Decomposition in $\hN$}
Now we select open dense subsets of $\hN$ such that restriction of $\Exp$ to these subsets will turn out to be a diffeomorphism.
Let
\begin{align}
&\tN= \{(\lam,t) \in \cup_{i=1}^3 N_i \mid t < \tt(\lam), \ \ts \tc \neq 0\}, \label{Ntilde}\\
&N'= \{(\lam,t) \in \cup_{i=1}^3 N_i \mid t = \tt(\lam) \text{ or } \ts \tc = 0\} \cup \hN_4 \cup N_5, \nonumber\\
&\hN_4 = \hN \cap N_4. \nonumber
\end{align}
We have the obvious decomposition $\hN = \tN \sqcup N'$ (we denote by $\sqcup$ the union of mutually non-intersecting sets).

There hold the following implications, see~\cite{max_sre2}:
\begin{align*}
&(\lam,t) \in N_1 \then \tt(\lam) = 2 K, \ \tau \in \R /(4 K \Z),\\
&(\lam,t) \in N_2 \then \tt(\lam) = 2 k p_1^1, \ \tau \in \R /(4 K \Z),\\
&(\lam,t) \in N_3 \then \tt(\lam) = + \infty, \ \tau \in \R.
\end{align*}
Consequently, there holds the following decomposition:
$$
\tN = \sqcup_{i=1}^8 D_i,
$$
where the sets $D_i$, $i = 1, \dots, 8$, are defined by Table~\ref{tab:Di}.

\begin{table}[htbp]
$$
\begin{array}{|c|c|c|c|c|c|c|c|c|}
\hline
D_i & D_1 & D_2 & D_3 & D_4 & D_5 & D_6 & D_7 & D_8 \\
\hline
\lam & C_1^0 & C_1^0 & C_1^0  & C_1^0 & C_1^1  & C_1^1  & C_1^1 & C_1^1 \\
\tau & (3K,4K) & (0,K) & (K,2K) & (2K,3K)  & (-K,0) & (0,K) & (K,2K) & (2K,3K) \\
p    & (0,K) & (0,K) & (0,K) & (0,K) & (0,K) & (0,K) & (0,K) & (0,K) \\
\hline
\lam & C_2^+ & C_2^+ & C_2^-  & C_2^- & C_2^+  & C_2^+  & C_2^- & C_2^- \\
\tau & (-K,0) & (0,K) & (-K,0) & (0,K)  & (K,2K) & (2K,3K) & (-3K,-2K) & (-2K,-K) \\
p    & (0,p_1^1) & (0,p_1^1) & (0,p_1^1) & (0,p_1^1) & (0,p_1^1) & (0,p_1^1) & (0,p_1^1) & (0,p_1^1) \\
\hline
\lam & C_3^{0+} & C_3^{0+} & C_3^{0-}  & C_3^{0-} & C_3^{1+}  & C_3^{1+}  & C_3^{1-} & C_3^{1-} \\
\tau & (-\infty,0) & (0,+\infty) & (-\infty,0) & (0,+\infty)  & (-\infty,0) & (0,+\infty) & (-\infty,0) & (0,+\infty) \\
p    & (0,+\infty) & (0,+\infty) & (0,+\infty) & (0,+\infty) & (0,+\infty) & (0,+\infty) & (0,+\infty) & (0,+\infty) \\
\hline
 \end{array}
$$ 
\caption{Definition of domains $D_i$}\label{tab:Di}
\end{table}

Table~\ref{tab:Di} should be read by columns. For example, the first column means that
\begin{align*}
&D_1 = (D_1 \cap N_1) \sqcup (D_1 \cap N_2) \sqcup (D_1 \cap N_3),\\
&D_1 \cap N_1 = \{(\tau,p,k) \in N_1 \mid \lam \in  C_1^0, \ \tau \in (3K, 4 K), \ p \in (0, K), \ k \in (0,1)\},\\
&D_1 \cap N_2 = \{(\tau,p,k) \in N_2 \mid \lam \in  C_2^+, \ \tau \in (-K, 0), \ p \in (0, p_1^1), \ k \in (0,1)\},\\
&D_1 \cap N_3 = \{(\tau,p,k) \in N_3 \mid \lam \in  C_3^{0+}, \ \tau \in (-\infty, 0), \ p \in (0, +\infty), \ k =1\}.
\end{align*}

Projections of the sets $D_i$ to the phase cylinder of the pendulum $(\g,c)$ are shown at Fig.~\ref{fig:Di}.

\onefiglabelsizen{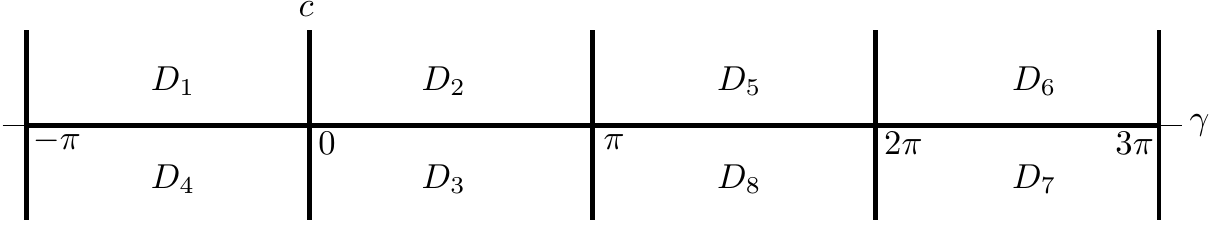}{Projections of domains $D_i$ to the phase cylinder of pendulum $(2S^1_{\g}) \times \R^1_c$}{fig:Di}{2.5}

\begin{lemma}
\label{lem:Di}
Each set $D_i$, $i = 1, \dots, 8$, is homeomorphic to $\R^3$.
\end{lemma}
\begin{proof}
We prove the statement only for the set $D_2$ since all other sets $D_i$ can be defined in the coordinates $(\tau,p,k)$ by the same inequalities as $D_2$ by a shift of origin in elliptic coordinate $\f$.
Taking into account Table~\ref{tab:Di} and Eqs.~\eq{ptauC1}, \eq{ptauC2}, we get:
\begin{align*}
&D_2 = (D_2 \cap N_1) \sqcup (D_2 \cap N_2) \sqcup (D_2 \cap N_3),\\
&D_2 \cap N_1 = \{(\lam,t) \in N_1 \mid \lam \in  C_1^0, \ k \in (0,1), \ t \in (0, 2K),\ \f \in (-t, -t + 2 K)\},\\
&D_2 \cap N_2 = \{(\lam,t) \in N_1 \mid \lam \in  C_2^+, \ k \in (0,1), \ t \in (0, 2k p_1^1),\ \f \in (-t, -t + 2 k K)\},\\
&D_2 \cap N_3 = \{(\lam,t) \in N_1 \mid \lam \in  C_3^{0+}, \ k =1, \ t \in (0, +\infty),\ \f \in (-t, +\infty)\}.
\end{align*}

As shown in~\cite{el_max}, one can choose regular system of coordinates $(k_1, \f,t)$ on the set $D_2$, where
$$
k_1 = k \text{ for } \lam \in C_1; \quad
k_1 = 1/k \text{ for } \lam \in C_2; \quad
k_1 = 1 \text{ for } \lam \in C_3.
$$
In this system of coordinates
\be{D2=nu}
D_2 = \{\nu = (k_1, \f, t) \mid k_1 \in (0, + \infty), \ t \in (0, t_1(k_1), \ \f \in (-t, -t + t_2(k_2))\},
\ee
where $t_1(k_1) = 2 K(k_1)$ for $k_1 \in (0,1)$, $t_1(k_1) = +\infty$ for $k_1 = 1$,   $t_1(k_1) = (2/k_1) p_1^1(1/k_1)$ for $k_1 \in (1, + \infty)$; and $t_2(k_1) = 2 K(k_1)$ for $k_1 \in (0,1)$, $t_2(k_1) = +\infty$ for $k_1 = 1$,   $t_2(k_1) = (2/k_1) K(1/k_1)$ for $k_1 \in (1, + \infty)$. The both functions $\map{t_i}{(0, + \infty)}{(0, + \infty]}$, $i = 1, 2$, are continuous. Thus representation~\eq{D2=nu} implies that the domain $D_2$ is homeomorphic to $\R^3$.  
\end{proof}

Consequently, all domains $D_i$ are open, connected, and simply connected.
These domains are schematically represented in the left-hand side of Fig.~\ref{fig:globExp}.

\subsubsection{Decomposition in $\hM$}
The state space of the problem admits a decomposition of the form
\begin{align*}
&M = \tM \sqcup M',\\
&\tM = \{q \in M \mid R_1(q) R_2(q) \sin \th \neq 0\},\\
&M' = \{q \in M \mid R_1(q) R_2(q) \sin \th = 0\},
\end{align*}
where 
$$
R_1 = y \cos \frac{\th}{2} - x \sin \frac{\th}{2}, \qquad 
R_2 = x \cos \frac{\th}{2} + y \sin \frac{\th}{2}.
$$
Further,
$$
\tM = \sqcup_{i=1}^8 M_i,
$$
where each of the sets $M_i$ is characterized by constant signs of the functions $\sin \th$, $R_1$, $R_2$ described in Table~\ref{tab:Mi}.

\begin{table}[htbp]
$$
\begin{array}{|c|c|c|c|c|c|c|c|c|}
\hline 
M_i      & M_1      & M_2 & M_3 & M_4 & M_5 & M_6 & M_7 & M_8 \\
\hline 
\sgn(\sin \th) & - & - & - & - & + & + & + & + \\
\hline 
\sgn(R_1) & + & + & - & - & - & - & + & + \\
\hline 
\sgn(R_2) & + & - & - & + & + & - & - & + \\
\hline 
\end{array}
$$
\caption{Definition of domains $M_i$}\label{tab:Mi}
\end{table}

For example,
$
M_1 = \{ q \in M \mid \sin \th < 0, \ R_1 > 0, \ R_2 > 0\}.
$
The numeration of the sets $M_i$ is chosen so that it correspond to numeration of the sets $N_i$ (we prove below in Th.~\ref{th:ExpDidiffeo}  that each mapping $\map{\Exp}{N_i}{M_i}$ is a diffeomorphism). It is obvious that all the sets $M_i$ are diffeomorphic to $\R^3$. 

All the domains $M_i$ are contained in the set $\{q \in M \mid \th \neq 0\}$. At this set $\th$ is a single-valued function, and we choose the branch $\th \in (0, 2 \pi)$. Thus in the sequel we assume that $\th \in (0, 2 \pi)$ on the sets $M_i$. Then $R_1$,  
$R_2$ become single-valued functions, and the last two rows of Table~\ref{tab:Mi} reflect the signs of these single-valued functions $R_i$ on the sets $M_i$. 

The boundary $M'$ of the domain $\tM$ decomposes into four mutually orthogonal surfaces: two planes $\{\th = 0\}$, $\{\th = \pi\}$ and two Moebius strips $\{R_1 = 0\}$, $\{R_2 = 0\}$, see the right-hand side of Fig.~\ref{fig:globExp}, and Fig.~7~\cite{max_sre2}.

\subsection{Diffeomorphic properties of exponential mapping}
\begin{lemma}
\label{lem:ExpDiMi}
For any $i = 1, \dots, 8$, we have $\Exp(D_i) \subset M_i$.
\end{lemma}
\begin{proof}
We prove only the inclusion $\Exp(D_1) \subset M_1$ since the rest inclusions are proved similarly.

Let $(\lam,t) \in D_1 \cap N_1 = \{(\lam,t) \in N_1 \mid \lam \in C_1^0, \ \tau \in (3 K, 4 K), \ p \in (0,K), \ k \in (0,1)\}$, see Table~\ref{tab:Di}. Since $\lam \in C_1^0$, then $s_1 = \sgn(\g_t/2) = 1$. Moreover, $\cc \ss \td  > 0$, thus $\sin \th_t < 0$ by virtue of Eq.~(5.2)~\cite{max_sre2}. So $\th_t/2 \in (\pi/2, \pi)$. Consequently, $\cos(\th_t/2) < 0$, $\sin \th_t/2 >0$ on $D_1 \cap N_1$, thus $s_3 = -1$, $s_4 = 1$ in Eq.~(5.3)--(5.6)~\cite{max_sre2}. Then we get $R_1 > 0$ from Eq.~(5.5)~\cite{max_sre2}, and $R_2 > 0$ from Eq.~(5.6)~\cite{max_sre2} and Lemma~5.2~\cite{max_sre2}. We proved that if $\nu \in D_1 \cap N_1$, then $\sin \th_t < 0$, $R_1 > 0$, $R_2 > 0$, i.e., $\Exp(\nu) \in M_1$, see Table~\ref{tab:Mi}. That is, $\Exp(D_1\cap N_1) \subset M_1$.

Since the domains $D_1$ and $M_1$ are connected, it follows that $\Exp(D_1) \subset M_1$.
\end{proof} 

\begin{lemma}
\label{lem:ExpDinondeg}
The restriction $\restr{\Exp}{\tN}$  is nondegenerate.
\end{lemma}
\begin{proof}
If $\nu = (\lam,t) \in \tN$, then $t < \tt(\lam)$. By Th.~\ref{th:conjC}, $\tconj(\lam) \geq \tt(\lam)$, thus $t < \tconj(\lam) = \inf \{ s > 0 \mid \Exp(\lam,s) \text{ is degenerate }\}$. Consequently, the exponential mapping is nondegenerate at the point $(\lam,t)$.
\end{proof}

\begin{lemma}
\label{lem:ExpDiproper}
For any $i = 1, \dots, 8$, the mapping $\map{\Exp}{D_i}{M_i}$ is proper.
\end{lemma}
\begin{proof}
Similarly to Lemma~\ref{lem:Di}, we can consider only the case $i = 2$. Let $K \subset M_2$ be a compact, we show that $S = \Exp^{-1}(K) \subset D_2$ is a compact as well, i.e., $S$ is bounded and closed.

There exists $\eps > 0$ such that 
$$
|\sin \th|\geq \eps, \qquad 
\eps \leq |R_1|, |R_2| \leq 1/\eps \qquad \text{for all } q \in K.
$$

(1) We show that $S$ is bounded. By contradiction, let $\nu_n = (k_n, \f_n, t_n) \to \infty$ for some sequence $\{\nu_n\}\subset S$. Then there exists a sequence $\{\nu_n\}\subset S \cap N_i$ for some $i = 1, 2, 3$ with  $\nu_n  \to \infty$. 

Let $S \cap N_1 \ni \nu_n = (k_n,\f_n,t_n) \to \infty$. Then $t_n = 2 p_n \in (0, K(k_n))$, $\tau_n = (\f_n + t_n)/2 \in (0, K(k_n))$. If $k_n$ is separated from 1, then $p_n$, $\tau_n$ are bounded, thus $t_n$, $\f_n$ are bounded, a contradiction. Thus $k_n \to 1$ for a subsequence (we will assume that this holds for the initial sequence).

If $(\g_n,c_n) \to (\pm \pi, 0)$, then $(\th_t, y_t) \to 0$, thus $R_1 \to 0$, a contradiction. Thus the sequence $(\g_n,c_n)$ is separated from the point $(\pm \pi,0)$.

Then there exists a sequence such that $k_n \to 1$ and $\f_n \to \f \in (-\infty, +\infty)$, thus $t_n \to + \infty$, $p_n \to + \infty$, $\tau_n \to + \infty$. Then $(p_n - \E(p_n))/(k_n \sqrt{\Delta}) \to \infty$, $f_2(p_n,k_n)  /(k_n \sqrt{\Delta}) \to \infty$. 

If $\cn(\tau_n)$ is separated from zero, then $R_1 \to \infty$ (see Eq.~(5.5)~\cite{max_sre2}).  And if $\cn(\tau_n)$ is not separated from zero, then there exists a sequence such that $\cn(\tau_n) \to 0$, thus $\sn(\tau_n)$ is separated from zero, then $R_2 \to \infty$ (see Eq.~(5.6)~\cite{max_sre2}).

So the hypothesis $S \cap C_1 \ni \nu_n = (k_n, \f_n,t_n) \to \infty$ leads to a contradiction.

Similarly the hypotheses $C \cap C_i \ni \nu_n   \to \infty$, $i = 2, 3$, lead to a contradiction.

Thus the set $S = \Exp^{-1}(K)$ is bounded.

(2) We show that $S$ is closed. Let $\{\nu_n\}\subset S$, we have to prove that there exists a subsequence $\nu_{n_k}$ converging in $D_2$. By contradiction, let $\nu_n \to \infty$ or $\nu_n \to \nu \in \partial D_2$.

Consider the case $\nu_n = (\tau_n, p_n, k_n) \in S \cap N_1$. 

If $k_n \to 0$, then $(x,y) \to 0$, thus $R_1, R_2 \to 0$, a contradiction. 

Let $k_n \to 1$. If $(\g_n,c_n) \to (\pm\pi,0)$, then $(\th,y) \to (0,0)$, thus $R_1 \to 0$, a contradiction. 

If $(\g_n,c_n) \to (\g,c) \neq (\pm \pi, 0)$, then $\nu \in N_3$, a contradiction. 

Thus $k_n \to k \in (0,1)$. Then $\tau_n \to \tau \in [3K(k), 4 K(k)]$. If $\tau = 3 K$, then $R_1 \to 0$, and if $\tau = 4 K$, then $R_2 \to 0$, a contradiction. Thus $\tau_n \to \tau \in (3K(k), 4 K(k))$.

Further, $p_n \to p \in [0, K(k)]$. If $p = 0$, then $t = 0$ and $R_1, R_2 \to 0$. If $p = K$, then $R_1 \to 0$. Thus $p_n \to p \in (0,K)$. 

So $(\tau_n, p_n, k_n) \to (\tau,p,k) \in N_1$, a contradiction.

We proved that any sequence $\nu_n \in S \cap N_1$ contains a subsequence converging in $D_2$. Similarly one proves the same for a sequence $\nu_n \in S \cap N_i$, $i = 2,3$.

Thus any sequence $\nu_n \in S$ contains a subsequence converging in $D_2$, thus converging in $S$. So the set $S = \Exp^{-1}(K)$ is closed.
\end{proof}

\begin{theorem}
\label{th:ExpDidiffeo}
For any $i = 1, \dots, 8$, we have $\Exp(D_i) \subset M_i$, and the mapping $\map{\Exp}{D_i}{M_i}$ is a diffeomorphism.
\end{theorem}
\begin{proof}
The inclusion $\Exp(D_i) \subset M_i$ was proved in Lemma~\ref{lem:ExpDiMi}. The mapping $\map{\Exp}{D_i}{M_i}$ is smooth, nondegenerate (Lemma~\ref{lem:ExpDinondeg}), and proper (Lemma~\ref{lem:ExpDiproper}), thus it is a covering. Since $M_i$ is simply connected, the mapping $\map{\Exp}{D_i}{M_i}$ is a diffeomorphism.
\end{proof}

\begin{lemma}
\label{lem:ExpN45}
$\Exp(N_4) = \{q \in M \mid R_1 = R_2 = 0\} = \{q \in M \mid x = y = 0\}$, $\Exp(N_5) = \{q \in M \mid R_1  = 0, \ R_2 \neq 0, \ \th = 0\} = \{q \in M \mid x \neq 0, \ y = 0, \ \th  = 0\}$.
\end{lemma}
\begin{proof}
Follows immediately from the corresponding formulas for extremal trajectories of Subsec.~3.3~\cite{max_sre2}.
\end{proof}

\begin{lemma}
\label{lem:ExpN'}
$\Exp(N') \subset M'$.
\end{lemma}
\begin{proof}
Follows from formulas~(5.2)--(5.11)~\cite{max_sre2}.
\end{proof}

Theorem~\ref{th:ExpDidiffeo} implies the following statement.

\begin{corollary}
\label{cor:ExpNtilde}
The mapping $\map{\Exp}{\tN}{\tM}$ is a diffeomorphism.
\end{corollary}

In view of Lemma~\ref{lem:ExpN'}, for any $q \in \tM$ there exists a unique $\nu = (\lam,t) = \Exp^{-1}(q) \in \tN$, $\lam = \lam(q)$, $t = t(q)$. 

The diffeomorphism $\map{\Exp}{\tN=\cup_{i=1}^8 D_i}{\tM = \cup_{i=1}^8 M_i}$ is schematically shown at Fig.~\ref{fig:globExp}.

\onefiglabelsizen
{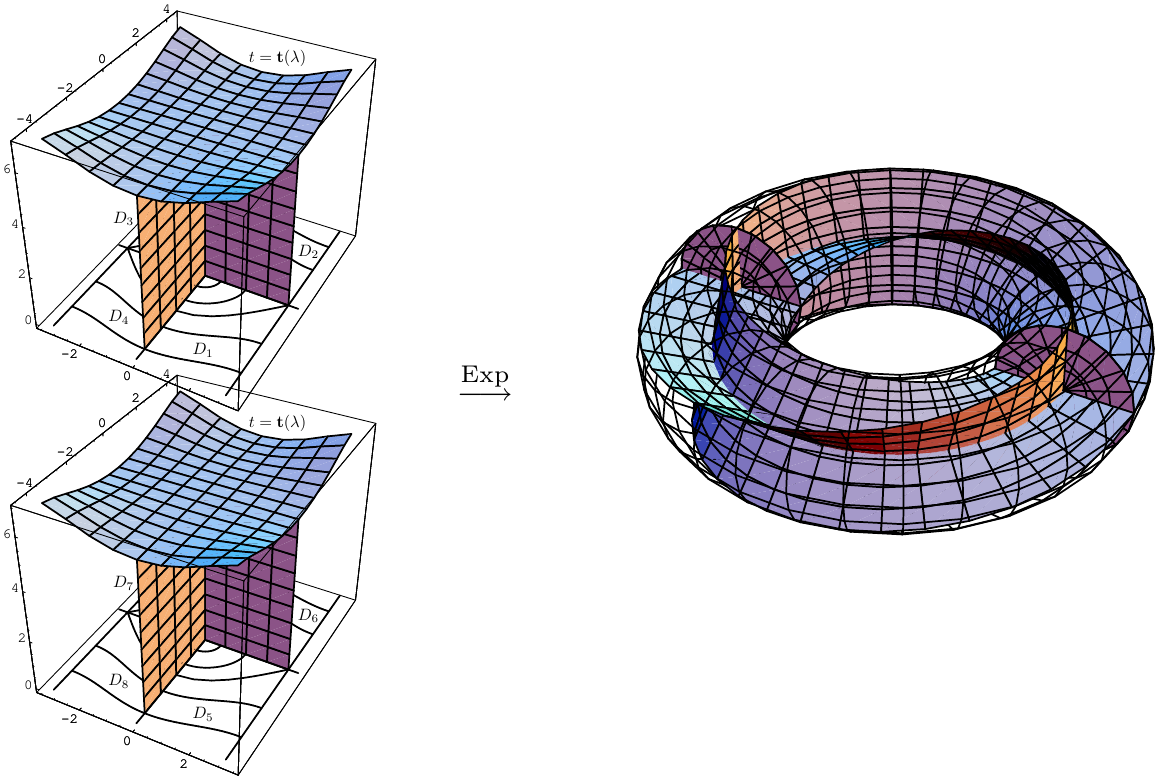}{Global structure of exponential mapping}{fig:globExp}{10}

\subsection{Cut time}

\begin{theorem}
\label{th:Mtildesynth}
For any $q_1 \in \tM$, let $(\lam_1,t_1) = \Exp^{-1}(q_1) \in \tN$. Then the extremal trajectory $q(s) = \Exp(\lam_1,s)$ is optimal with $q(t_1) = q_1$. 

Thus optimal synthesis on the domain $\tM$ is given by 
$$
u_i(q) = h_i(\lam), \quad i = 1, 2, \qquad (\lam, t) = \Exp^{-1}(q) \in \tN, \qquad q \in \tM.
$$
\end{theorem}
\begin{proof}
Let $q_1 \in \tM$. There exists $\nu_1 = (\lam_1, t_1) \in \hN = \tN \sqcup N'$ such that the trajectory $q(s) = \Exp(\lam_1,s)$ is optimal and $q(t_1) = \Exp(\nu_1) = q_1$. By Lemmas~\ref{lem:ExpN45} and~\ref{lem:ExpN'}, we have $\nu_1 \in \tN$. By Cor.~\ref{cor:ExpNtilde}, there exists a unique $\nu_1 \in \tN$ such that $\Exp(\nu_1) = q_1$. So $q(s) = \Exp(\lam_1, s)$ is a unique optimal trajectory coming to $q_1$.
\end{proof}

In work~\cite{max_sre2} we introduced the following function $\map{\tt}{C}{(0, + \infty]}$ on the phase cylinder of pendulum $(2 S^1_{\g}) \times \R_c = C = \sqcup_{i=1}^5 C_i$:
\begin{align}
&\lam \in C_1 \then \tt(\lam) = 2 K(k), \label{ttC1} \\
&\lam \in C_2 \then \tt(\lam) = 2 k p_1^1(k), \label{ttC2}\\
&\lam \in C_3 \then \tt(\lam) = +\infty, \label{ttC3}\\
&\lam \in C_4 \then \tt(\lam) = \pi, \label{ttC4}\\
&\lam \in C_5 \then \tt(\lam) = +\infty. \label{ttC5}
\end{align}and proved the inequality $\tcut(\lam) \leq \tt(\lam)$. Now we prove the corresponding equality.

\begin{theorem}
\label{th:tcutC}
For any $\lam \in C$ we have $\tcut(\lam) = \tt(\lam)$.
\end{theorem}
\begin{proof}
We proved the inequality $\tcut(\lam) \leq \tt(\lam)$ in Th.~5.4~\cite{max_sre2}.

(1) Consider first the generic case $\lam_1 \in \cup_{i=1}^3 C_i$. There exists $t_1 \in (0, \tt(\lam_1))$ and arbitrarily close to $\tt(\lam_1)$ such that $\sn \tau_1 \, \cn \tau_1 \neq 0$. Then $\nu_1 = (\lam_1, t_1) \in \tN$, thus $q_1 = \Exp(\nu_1) \in \tM$. By Th.~\ref{th:Mtildesynth}, the trajectory $q(s) = \Exp(\lam_1,s)$, $s \in [0,t_1]$, is optimal, thus $t_1 \leq \tcut (\lam_1)$. 

So there exists $t_1 \in (0, \tt(\lam_1))$ arbitrarily close to $\tt(\lam_1)$ such that $t_1 \leq \tcut(\lam_1)$. Consequently, $\tt(\lam_1) \leq \tcut(\lam_1)$.

We proved that $\tcut(\lam_1) = \tt(\lam_1)$ for any $\lam_1 \in \cup_{i=1}^3 C_i$.

(2) 
If $\lam \in C_4$, then the extremal trajectory $(x,y,\th) = (0,0, \pm t)$ is a Riemannian geodesic for the restriction of the sub-Riemannian problem on $\SE(2)$ to the circle $\{(0,0,\th) \mid \th \in S^1\}$. It is optimal up to the antipodal point, thus $\tcut(\lam) = \pi = \tt(\lam)$.

(3)
In the case $\lam \in C_5$ the extremal trajectory is a line $(x,y,\th) = (\pm t, 0, 0)$, thus it is optimal forever: $\tcut(\lam) = + \infty = \tt(\lam)$.
\end{proof}

\section[Structure of exponential mapping at the boundary of open stratas]
{Structure of exponential mapping \\at the boundary of open stratas}
\label{sec:bound}
In this section we describe the action of the exponential mapping
$$
\map{\Exp}{N' = \hN \setminus \tN}{M' = \hM \setminus \tM}.
$$

\subsection{Decomposition of the set $N'$}
\label{subsec:decompN'}
Consider the following subsets of the set
$$
N' = \hN \setminus \tN = \{  (\lam,t) \in \cup_{i=1}^3 N_i \mid t = \tt(\lam) \text{ or } \ts \tc = 0\} \cup \hN_4 \cup N_5:
$$
\begin{align}
&\Ncut = \{ (\lam,t) \in N \mid t = \tt(\lam) \}, \label{Ncutdef}\\
&\Nconj =  \{ (\lam,t) \in N_2  \mid t = \tt(\lam), \ \ts = 0  \}, \label{Nconjdef}\\
&\NMax = \Ncut \setminus \Nconj, \qquad \Nrest = N' \setminus \Ncut. \nonumber
\end{align}
So we have the following decompositions:
\be{NhatNconj}
\hN = \tN \sqcup N', \qquad 
N' = \Ncut \sqcup \Nrest, \qquad
\Ncut = \NMax \sqcup \Nconj.
\ee

In order to study the structure of the exponential mapping at the set $N'$, we need further decomposition 
into subsets $N_i'$, $i = 1, \dots, 58$, defined by Table~\ref{tab:Ni'}.

\begin{table}
$$
\begin{array}{|c|c|c|c|c|c|c|c|c|}
\hline
N_i' & N_1' & N_2' & N_3' & N_4' & N_5' & N_6'& N_7'& N_8'\\
\hline
\lam  & C_1^0 & C_1^0 &  C_1^0 & C_1^0  &C_1^1 &C_1^1 &C_1^1 &C_1^1\\
\hline
\tau & (0, K) & (K,2K) & (2K,3K) & (3K,4K)& (0, K) & (K,2K) & (2K,3K) & (3K,4K)\\
\hline
p & K & K& K& K& K& K & K & K\\
\hline
\end{array}
$$
$$
\begin{array}{|c|c|c|c|c|c|c|c|c|}
\hline
N_i' & N_9' & N_{10}' & N_{11}' & N_{12}' & N_{13}' & N_{14}'& N_{15}'& N_{16}'\\
\hline
\lam  & C_2^+ & C_2^+ &  C_2^+ & C_2^+   &C_2^- &C_2^-  &C_2^- &C_2^-\\
\hline
\tau & (3K,4K) & (0, K) & (K,2K) & (2K,3K) & (-3K, -2K) & (-2K,-K) & (-K,0) & (0,K)\\
\hline
p & p_1^1 & p_1^1 & p_1^1 & p_1^1& p_1^1& p_1^1 & p_1^1 & p_1^1\\
\hline
\end{array}
$$
$$
\begin{array}{|c|c|c|c|c|c|c|c|c|c|c|c|c|c|c|}
\hline
N_i' & N_{17}' & N_{18}' & N_{19}' & N_{20}' & N_{21}' & N_{22}'& N_{23}'& N_{24}'& N_{25}'& N_{26}'& N_{27}' & N_{28}' & N_{29}' & N_{30}'\\
\hline
\lam  & C_1^0 & C_1^0 &  C_1^0 & C_1^0  &C_1^1 &C_1^1  &C_1^1 &C_1^1&  C_2^+&  C_2^+ &  C_2^+ &  C_2^+ &  C_2^- &  C_2^- \\
\hline
\tau & 0 & K & 2K & 3K & 0 & K & 2K & 3K& 3K & 0 & K & 2K & K & -2K\\
\hline
p & K & K & K & K& K & K & K & K& p_1^1 & p_1^1 & p_1^1 & p_1^1 & p_1^1 & p_1^1\\
\hline
\end{array}
$$
$$
\begin{array}{|c|c|c|c|c|c|c|c|c|c|c|}
\hline
N_i' & N_{31}' & N_{32}' & N_{35}' & N_{36}' & N_{37}' & N_{38}'& N_{39}'& N_{40}'& N_{41}'& N_{42}'\\
\hline
\lam  & C_2^- & C_2^- &  C_1^0 & C_1^0   &C_1^0  &C_1^0  &C_1^1 & C_1^1 & C_1^1 & C_1^1\\
\hline
\tau & -K & 0 & 0 & K & 2K & 3K & 0  & K & 2K & 3K\\
\hline
p & p_1^1 & p_1^1 & (0,K) & (0,K)& (0,K)& (0,K) & (0,K)  & (0,K) & (0,K)& (0,K)\\
\hline
\end{array}
$$
$$
\begin{array}{|c|c|c|c|c|c|c|c|c|}
\hline
N_i' & N_{47}' & N_{48}' & N_{49}' & N_{50}' & N_{51}' & N_{52}'& N_{53}'& N_{54}'\\
\hline
\lam  & C_3^{0+} & C_3^{0-} &  C_3^{1+} & C_3^{1-}   &C_2^+  &C_2^+   &C_2^+ & C_2^+ \\
\hline
\tau & 0 & 0 & 0 & 0 & 3K & 0 & K  & 2K \\
\hline
p & (0, + \infty) & (0, + \infty) & (0, + \infty) & (0, + \infty) & (0,p_1^1)& (0,p_1^1) & (0,p_1^1)  & (0,p_1^1) \\
\hline
\end{array}
$$
$$
\begin{array}{|c|c|c|c|c|}
\hline
N_i' & N_{55}' & N_{56}' & N_{57}' & N_{58}'  \\
\hline
\lam  & C_2^{-} & C_2^{-} &  C_2^{-} & C_2^{-}    \\
\hline
\tau & K & -2K & -K & 0    \\
\hline
p & (0,p_1^1) & (0,p_1^1)  & (0,p_1^1)  & (0,p_1^1)   \\
\hline
\end{array}
$$
$$
\begin{array}{|c|c|c|c|c|c|c|}
\hline
N_i' & N_{33}' & N_{34}' & N_{43}' & N_{44}' & N_{45}'& N_{46}' \\
\hline
\lam  & C_4^{0} & C_4^{1} &  C_4^{0} & C_4^{1}& C_5^{0} & C_5^{1}    \\
\hline
t & \pi & \pi  & (0,\pi)  & (0,\pi) & (0,+ \infty)  & (0,+ \infty)  \\
\hline
\end{array}
$$
\caption{Definition of sets $N_i'$}\label{tab:Ni'}
\end{table}

 Images of the projections
\begin{align*}
&N_i' \cap \{ t < \tt(\lam) , \ \ts \tc = 0 \}  \to \{p = 0\}, \qquad (k,\tau,p) \mapsto (k,\tau,0), \\
&N_i' \cap \{ t = \tt(\lam) \}  \to \{p = 0\}, \qquad (k,\tau,p) \mapsto (k,\tau,0), \\
\end{align*}
are shown respectively at Figs.~\ref{fig:N'lessttt}, \ref{fig:N't=tt}.

\onefiglabelsizen{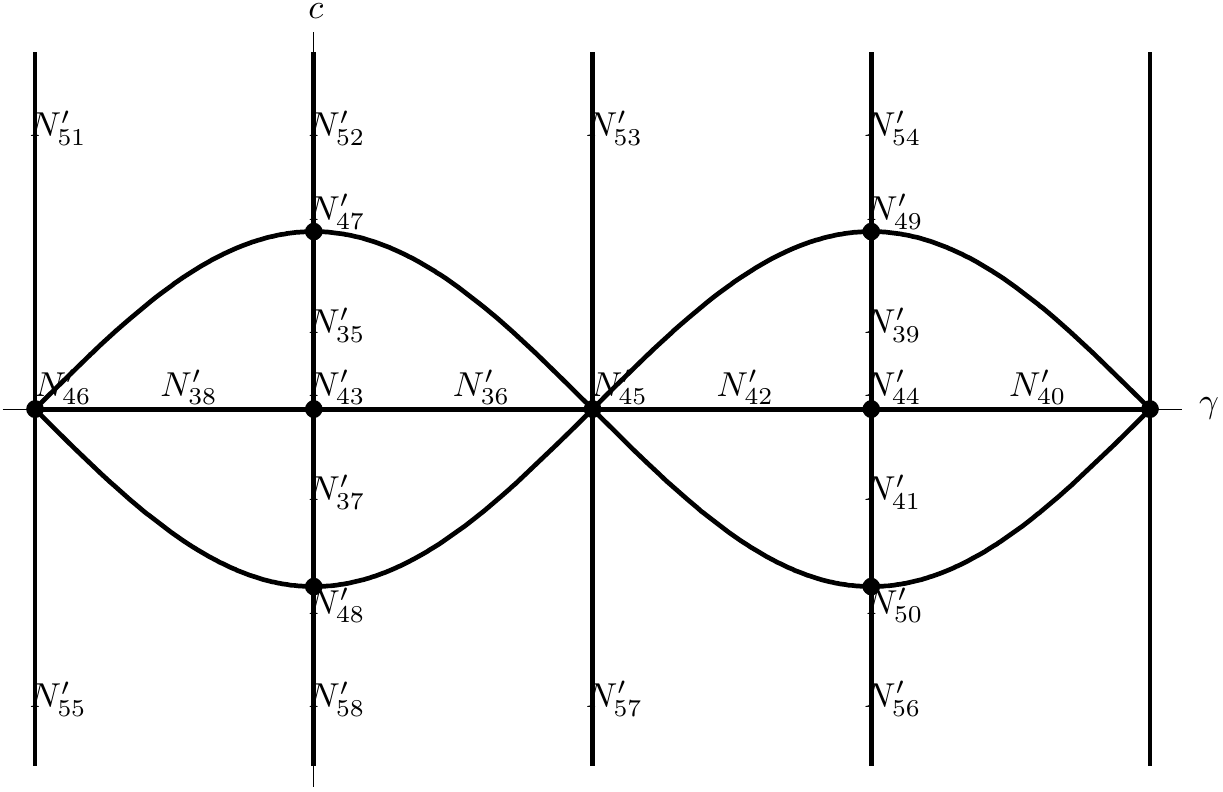}{$N_i' \cap \{ t < \tt(\lam), \ \ts \tc = 0 \} $}{fig:N'lessttt}{8}

\onefiglabelsizen{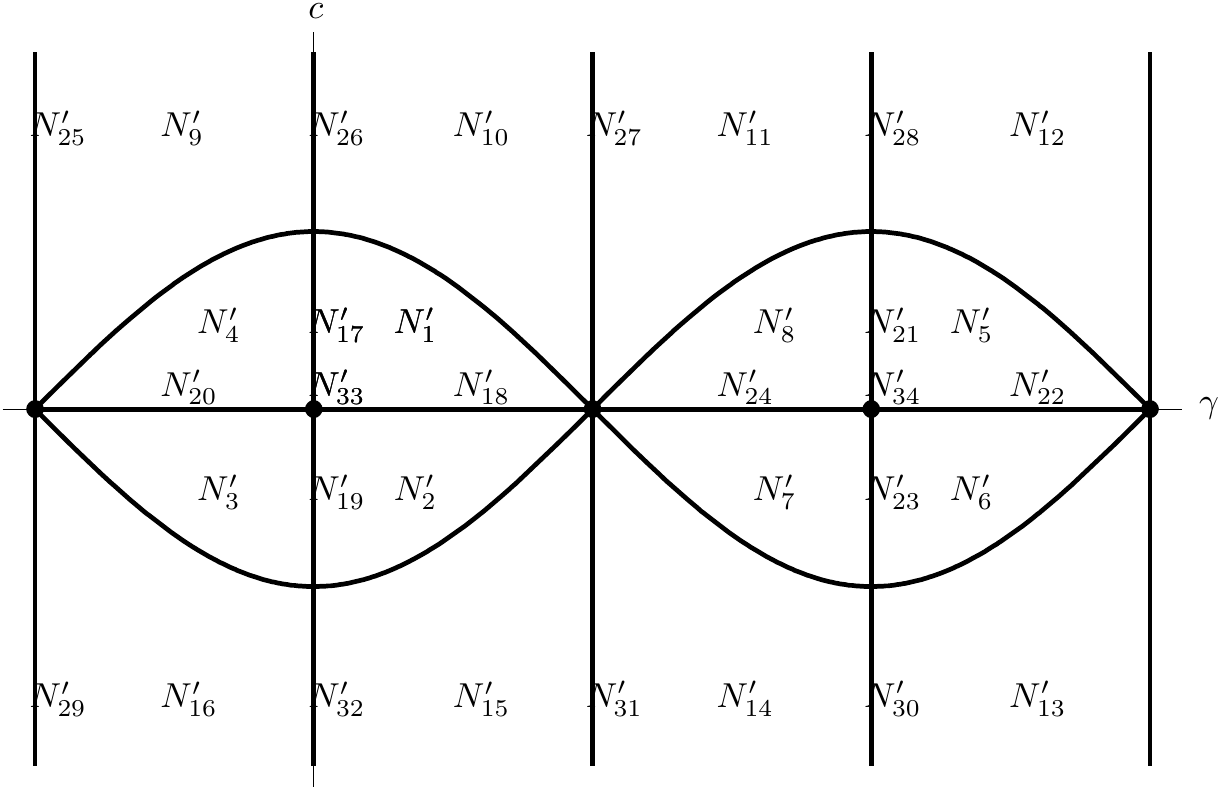}{$N_i' \cap \{ t = \tt(\lam) \} $}{fig:N't=tt}{8}

Table~\ref{tab:Ni'} provides a definition of the sets $N_i'$; e.g.,  the second column of this table means that
$$
N_1' = \{ (\lam,t) \in N \mid \lam \in C_1^0, \ \tau \in (0,K), \ p \in (0,K), \ k \in (0,1)\}.
$$

Introduce the following index sets for numeration of the subsets $N_i'$:
\begin{align}
&I = \{ 1, \dots, 58\}, \quad
C = \{ 1, \dots, 34\}, \quad
J = \{ 26, 28, 30, 32\}, \label{ICJ}\\
&R = \{ 35, \dots, 58\}, \quad
X = C \setminus J. \label{RX}
\end{align}
Notice that $I = C \sqcup R$, $J \subset C$.

\begin{lemma}
\label{lem:decompN'}
\begin{itemize}
\item[$(1)$]
We have $N_i' \cap N_j' = \emptyset$ for any distinct $i, j \in I$.
\item[$(2)$]
There are the following decompositions of subsets of the set $N'$:
$$
\Ncut = \cup _{i \in C} N_i', \qquad
\Nconj = \cup _{i \in J} N_i', \qquad
\Nrest = \cup _{i \in R} N_i', 
$$
thus
\begin{align}
&\NMax = \cup_{i \in X} N_i',  \label{NMaxd} \\
&N' = \sqcup_{i \in I} N_i'. \label{N'=N_i'}
\end{align}
\end{itemize}
\end{lemma} 
\begin{proof}
Both statements (1), (2) follow directly from Table~\ref{tab:Ni'}, definitions of the sets $N'$, $\NMax$, $\Ncut$, $\Nconj$, $\Nrest$, and decompositions~\eq{NhatNconj}.
\end{proof}

\subsection{Exponential mapping of the set $N_{35}'$}
In order to describe the image $\Exp(N_{35}')$, we will need the following function:
\be{R12}
R_1^2(\th) = 2 (\arth(\sin(\th/2)) - \sin(\th/2)), \qquad \th \in [0, \pi).
\ee 
It is obvious that $R_1^2 \in C^{\infty}[0, \pi)$, $R_1^2(0) = 0$, $R_1^2(\th) > 0$ for $\th \in (0,\pi)$, $\lim_{\th \to \pi -0} R_1^2(\th)= + \infty$, and
\be{dR12}
\der{R_1^2}{\th} (0) = 0.
\ee
A plot of the function $R_1^2(\th)$ is given at Fig.~\ref{fig:R12}.

\onefiglabel
{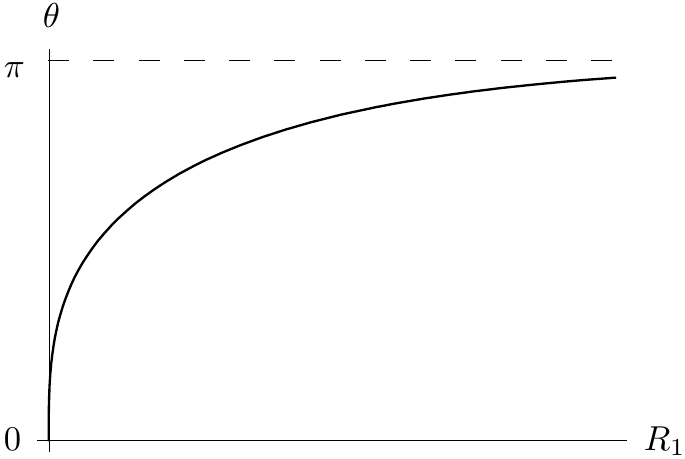}{Plot of $R_1 = R_1^2(\th)$}{fig:R12}

Define the following subset of the set $M'$:
$$
M_{35}' = \{ q \in M \mid \th \in (\pi, 2 \pi), \ R_1 \in (0, R_1^2(2 \pi - \th)), \ R_2 = 0 \}.
$$

\begin{lemma}
\label{lem:N1'M1'}
The mapping $\map{\Exp}{N_{35}'}{M_{35}'}$ is a diffeomorphism of 2-dimensional manifolds.
\end{lemma}

To be more precise, we state that $\Exp(N_{35}')  = M_{35}'$ and $\restr{\Exp}{N_{35}'}$ is a diffeomorphism of the manifold $N_{35}'$ onto the manifold $M_{35}'$. Below we will write such statements briefly as in Lemma~\ref{lem:N1'M1'}.

\begin{proof}
Formulas~(5.2)--(5.6)~\cite{max_sre2} imply that in the domain $N_{35}'$ we have the following:
\begin{align}
&\sin(\th/2) = \ss,    \qquad &&\cos(\th/2) = -\cc, \label{sinth2N1'M1'}\\
&R_1 = 2(p - \E(p))/k, \qquad &&R_2 = 0.          \label{R1N1'M1'}
\end{align}
By Th.~\ref{th:conjC}, the restriction $\restr{\Exp}{N_{35}'}$ is nondegenerate. Thus the set $\Exp(N_{35}')$ is an open connected domain in the 2-dimensional manifold
$$
S = \{ q \in M \mid \th \in (\pi, 2 \pi),\ R_1 > 0, \ R_2 = 0\}.
$$
On the other hand, the set $N_{35}'$ is an open connected simply connected domain in the 2-dimensional manifold
$$
T = \{ \nu \in N_1 \mid \tau = 0, \ p \in (0,K), \ k \in (0,1)\}.
$$ 
In the topology of $T$, we have
\begin{align*}
&\partial N_{35}' = \cup_{i=1}^4 n_i, \\
&n_1 = \{\nu \in N_1 \mid \tau = 0, \ p =0, \ k \in [0,1]\}, \\
&n_2 = \{\nu \in N_1 \mid \tau = 0, \ p \in[0,\pi/2] \ k =0\}, \\
&n_3 = \{\nu \in N_1 \mid \tau = 0, \ p = K(k), \ k \in [0,1)\}, \\
&n_4 = \{\nu \in N_1 \mid \tau = 0, \ p \in [0, + \infty), \ k =1]\}, 
\end{align*}
see Fig.~\ref{fig:N35'}.

\twofiglabelh
{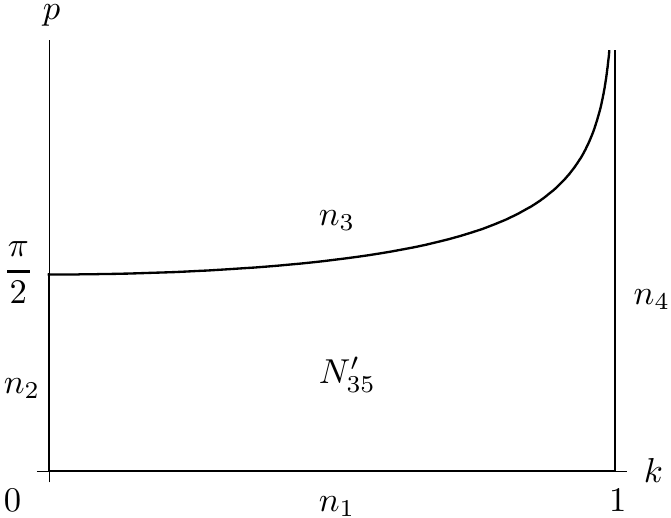}{Domain $N_{35}'$}{fig:N35'}
{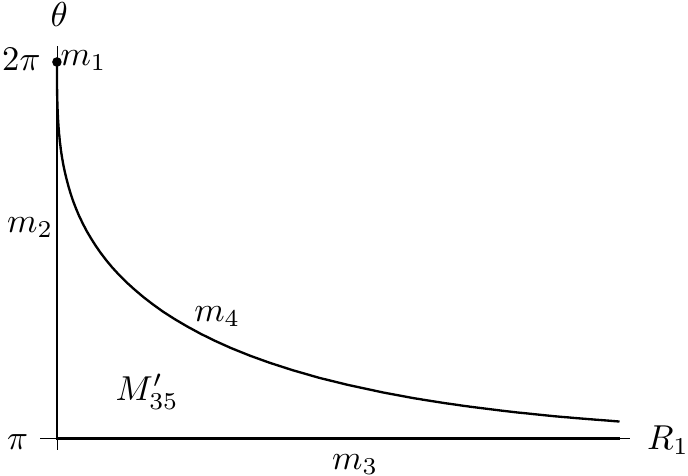}{Domain $M_{35}'$}{fig:M35'}
{4.5}

We have
\begin{align*}
&\Exp(n_1) = m_1 = \{ q \in M \mid \th = 2 \pi, \ R_1 = 0, \ R_2 = 0 \},\\
&\Exp(n_2) = m_2 = \{ q \in M \mid \th \in[\pi,  2 \pi], \ R_1 = 0, \ R_2 = 0 \},\\
&\Exp(n_3) = m_3 = \{ q \in M \mid \th = \pi, \ R_1 > 0, \ R_2 = 0 \},\\
&\Exp(n_4) = m_4 = \{ q \in M \mid \th \in[\pi,  2 \pi], \ R_1 = R_1^2(2 \pi - \th), \ R_2 = 0 \},
\end{align*}
moreover, $\partial M_{35}' = \cup _{i=1}^4 m_i$, see Fig.~\ref{fig:M35'}.

\partproof{a}
We show that $\Exp(N_{35}') \cap M_{35}' \neq \emptyset$.
Formulas~\eq{sinth2N1'M1'}, \eq{R1N1'M1'} give the following asymptotics as $k \to 0$:
$$
\th = 2 \pi - 2 p + o(1), \qquad R_1 = k(p/2 - (\sin 2p)/4) + o(k).
$$
There exists $(p,k)$ close to $(\pi/2,0)$ such that the corresponding point $(\th, R_1)$ is arbitrarily close to $(0,0)$, with $\th > 0$, $R_1 > 0$. Thus there exists $\nu \in N_{35}'$ such that $\Exp(\nu) \in M_{35}'$.

\partproof{b}
We show that $\Exp(N_{35}') \neq S$. Formulas~\eq{sinth2N1'M1'}, \eq{R1N1'M1'} yield the following chain:
$$
\th\to 2 \pi - 0 \then \ss \to 0 \then p \to 0 \then R_1 \to 0.
$$
Thus there exists $q \in S \setminus \Exp(N_{35}')$.

\partproof{c}
We prove that $\Exp(N_{35}') \subset M_{35}'$. By contradiction, suppose that there exists a point $q_1 \in \Exp(N_{35}') \setminus M_{35}'$. Since the mapping $\restr{\Exp}{N_{35}'}$ is nondegenerate, we can choose this point such that $q_1 \in \Exp(N_{35}') \setminus \cl( M_{35}')$.

Choose any point $q_2 \in S \setminus \cl(M_{35}')$. Connect the points $q_1$, $q_2$ by a continuous curve in $S$, and find at this curve a point $q_3 \in S \setminus \Exp(N_{35}')$, $q_3 \notin \cl(M_{35}')$ such that there exists a converging sequence $q^n \to q_3$, $q^n = \Exp(\nu^n) \in \Exp(N_{35}')$. Further, there exist a subsequence $\nu^{n_i} \in N_{35}'$ converging to a finite or infinite limit. If $\nu^{n_i} \to \bnu \in N_{35}'$,then $q_3 = \Exp(\bnu) \in \intt \Exp(N_{35}')$ by nondegeneracy of $\restr{\Exp}{N_{35}'}$, a contradiction. If $\nu^{n_i} \to \bnu \in \partial N_{35}'$, then 
$$
q_3 = \Exp(\bnu) \in \Exp(\partial N_{35}') = \partial M_{35}' \subset \cl(M_{35}'),
$$
a contradiction. Finally, if $\nu^{n_i} \to \infty$, then at this sequence $k^{n_i} \to 1 - 0$, $p^{n_i} \to \infty$, thus $R_1(q^{n_i}) \to \infty$, a contradiction. 

Consequently, $\Exp(N_{35}') \subset M_{35}'$.

\partproof{d}
The mapping $\map{\Exp}{N_{35}'}{M_{35}'}$ is a diffeomorphism since $\restr{\Exp}{N_{35}'}$ is nondegenerate and proper, and $N_{35}'$, $M_{35}'$ are connected and simply connected.
\end{proof}

\subsection{Exponential mapping of the set $N_{47}'$}
Define the following subset of $M'$:
$$
M_{47}' = \{q \in M \mid \th \in (\pi, 2 \pi), \ R_1 = R_1^2(2 \pi - \th), \ R_2 = 0 \}.
$$

\begin{lemma}
\label{lem:N2'M2'}
The mapping $\map{\Exp}{N_{47}'}{M_{47}'}$ is a diffeomorphism of 1-dimensional manifolds.
\end{lemma}
\begin{proof}
We pass to the limit $k \to 1 - 0$ in formulas~\eq{sinth2N1'M1'}, \eq{R1N1'M1'} and obtain for $\nu \in N_{47}'$:
$$
\sin(\th/2) = \tanh p, \quad \cos(\th/2) = - 1/\cosh p, \quad R_1 = 2(p - \tanh p), \quad R_2 = 0.
$$
This coordinate representation shows that $\map{\Exp}{N_{47}'}{M_{47}'}$ is a diffeomorphism.
\end{proof}

\subsection{Exponential mapping of the set $N_{26}'$}
Before the study of $\restr{\Exp}{N_{52}'}$, postponed till the next subsection, we need to consider the set $N_{26}'$ contained in the boundary of $N_{52}'$. In order to parametrize regularly the image $\Exp(N_{26}')$, we introduce the necessary functions.

Recall that the function $p = p_1^1(k)$, $k \in [0,1)$, is the first positive root of the function $f_1(p) = \cc (\E(p) - p) - \dd \ss$, see Eq.~(5.11) and Cor.~5.1~\cite{max_sre2}. Define the function
\be{v11}
v_1^1(k) = \am(p_1^1(k),k), \qquad k \in [0,1).
\ee

\begin{lemma}
\label{lem:v11}
\begin{itemize}
\item[$(1)$]
The number $v = v_1^1(k)$ is the first positive root of the function
$$
h_1(v,k) = E(v,k) - F(v,k) - \sqrt{1 - k^2 \sin^2 v} \tan v, \qquad k \in [0, 1).
$$
\item[$(2)$]
$v_1^1 \in C^{\infty}[0,1)$.
\item[$(3)$]
$v_1^1(k) \in (\pi/2, \pi)$ for $ k \in (0,1)$; moreover, $v_1^1(0) = \pi$.
\item[$(4)$]
The function $v_1^1(k)$ is strictly decreasing at the segment $ k \in [0, 1)$.
\item[$(5)$]
$\lim_{v \to 1 - 0} v_1^1(k) = \pi/2$, thus setting $v_1^1(1) = \pi/2$, we obtain $v_1^1 \in C[0,1]$.
\item[$(6)$]
$v_1^1(k) = \pi - (\pi/2) k^2 + o(k^2)$, $k \to + 0$.
\end{itemize}
\end{lemma}
\begin{proof}
(1) follows from~\eq{v11} since $p = p_1^1$ is the first positive root of the function $f_1(p)$.

(2) follows since $p_1^1 \in C^{\infty}[0,1)$ by Lemma~5.3~\cite{max_sre2}.

(3) follows since $p_1^1 \in (K,2K)$ and $p_1^1(0) = \pi$, see Cor.~5.1~\cite{max_sre2}.

(4) We have for $v \in (\pi/2, \pi]$:
\begin{align*}
&\pder{h_1}{v} = - \sqrt{1 - k^2 \sin^2 v} /\cos^2v < 0, \\
&\pder{h_1}{k} = - \frac{k}{1-k^2}(E(v,k) - \sqrt{1-k^2 \sin^2 v} \tan v ) < 0.
\end{align*}
Thus $\ds\der{v_1^1}{k} = - \frac{\partial h_1/\partial k} {\partial h_1/\partial v} < 0$ for $k \in [0,1)$.

(5) Monotonicity and boundedness of $v_1^1(k)$ imply that there exists a limit $\lim_{k \to 1 - 0} v_1^1(k) = \bv \in [\pi/2, \pi)$. If $\bv \in (\pi/2, \pi)$, then as $k \to 1 - 0$
$$
h_1(v_1^1(k),k) \to \int_0^{\bv} \left(|\cos t| - 1/|\cos t|\right)\, d t - \sqrt{1 - \sin^2 \bv} \tan \bv = \infty,
$$
which contradicts the identity $h_1(v_1^1(k),k) \equiv 0$, $k \in [0, 1)$. Thus $\bv = \pi/2$.

(6) As $(k,v) \to (0,\pi)$, we have $h_1(v,k) = v - \pi + (\pi/2) k^2 + o(k^2 + (v-\pi)^2)$, thus $v_1^1(k) = \pi - (\pi/2)k^2 + o(k^2)$, $k \to + 0$.
\end{proof}

A plot of the function $v_1^1(k)$ is given at Fig.~\ref{fig:v11}.

\twofiglabel
{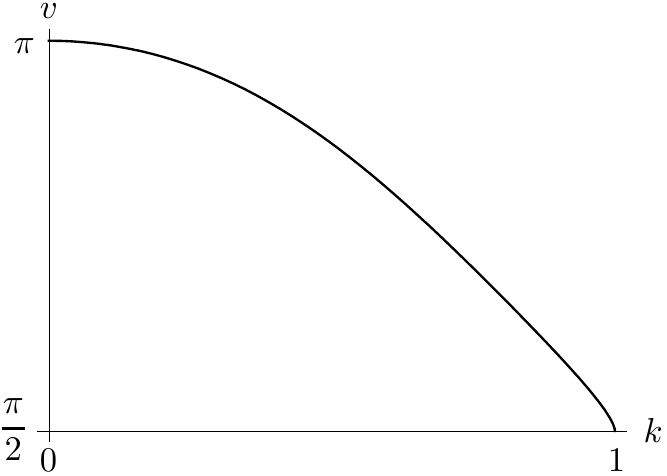}{Plot of $v = v_1^1(k)$}{fig:v11}
{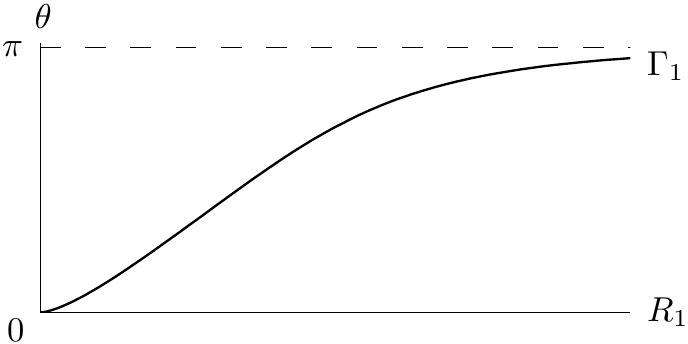}{The curve $\Gamma_1$}{fig:Gam1}

Define the curve
$
\G_1 \subset S = \{ q \in M \mid \th \in (0, \pi), \ R_1 > 0, \ R_2 = 0 \}$ given parametrically as follows:
\begin{align}
&\th = 2 \arcsin(k \sin v_1^1(k)), \label{gam1th}\\
&R_1 = 2 (F(v_1^1(k),k) - E(v_1^1(k),k)), \qquad k \in [0, 1), \label{gam1R1}
\end{align}
see Fig.~\ref{fig:Gam1}.

\begin{lemma}
\label{lem:Gam1}
\begin{itemize}
\item[$(1)$]
The function $k \sin v_1^1(k)$ is strictly increasing as $k \in [0,1]$, thus the function $\th = \th(k)$, $k \in [0,1]$, determined by~\eq{gam1th} has an inverse function $k = k_1^1(\th)$, $\th \in [0, \pi]$.
\item[$(2)$]
$k_1^1 \in C[0,\pi] \cap C^{\infty}[0,\pi)$.
\item[$(3)$]
The function $k_1^1(\th)$ is strictly increasing as $\th \in [0,\pi]$.
\item[$(4)$]
The curve $\G_1$ is a graph of the function
\begin{align}
&R_1 = R_1^1(\th), \qquad \th\in [0,\pi], \nonumber\\
&R_1^1(\th) = 2 (F(v_1^1(k),k) - E(v_1^1(k),k)), \qquad k = k_1^1(\th). \label{R11}
\end{align}
\item[$(5)$]
$R_1^1 \in C[0,\pi] \cap C^{\infty}(0,\pi)$.
\item[$(6)$]
$R_1^1(\th) = \sqrt[3]{\pi}/2 \, \th^{2/3} + o(\th^{2/3})$, $\th \to + 0$.
\end{itemize}
\end{lemma}
\begin{proof}
(1) As $k \in [0,1]$, we have:
\begin{align*}
&v_1^1(k) \ \downarrow, \qquad v_1^1 (k) \in [\pi/2, \pi], \\
&\sin v_1^1(k) \ \uparrow, \qquad k \sin v_1^1(k) \ \uparrow, \qquad 2 \arcsin(k v_1^1(k)) \ \uparrow.
\end{align*}

(2) follows from items (2), (5) of Lemma~\ref{lem:v11}.

(3) follows from item (1) of this lemma.

(4) follows from~\eq{gam1th}, \eq{gam1R1}.

(5) follows from item (2) of this lemma.

(6) As $k \to + 0$, we have
$$
v_1^1(k) = \pi - (\pi/2) k^2 + o(k^2), \qquad \sin v_1^1(k) = (\pi/2)   k^2 + o(k^2),
$$
and for the functions~\eq{gam1th}, \eq{gam1R1}
$$
\th = \pi k^3 + o(k^3), \qquad R_1 = (\pi/2)k^2 + o(k^2).
$$
Thus as $\th \to + 0$, we have
$$
k_1^1(\th) = \sqrt[3]{\th/\pi} + o(\sqrt[3]{\th}), \qquad
R_1^1(\th) = \sqrt[3]{\pi}/2 \, \th^{2/3}  + o(\th^{2/3}).
$$
\end{proof}

Define the following subset of $M'$:
$$
M_{26}' = \{ q \in M \mid \th \in (\pi, 2 \pi), \ R_1 = R_1^1(2 \pi - \th), \ R_2 = 0 \}.
$$

\begin{lemma}
\label{lem:N44'M44'}
The mapping $\map{\Exp}{N_{26}'}{M_{26}'}$ is a diffeomorphism of 1-dimensional manifolds.
\end{lemma}
\begin{proof}
For $\nu \in N_{26}'$ we obtain from formulas~(5.7)--(5.12)~\cite{max_sre2}:
\begin{align*}
&\sin(\th/2) = k \sn p_1^1(k) = k \sin v_1^1(k), \\
&\cos(\th/2) = - \dn p_1^1(k) = - \sqrt{1 - k^2 \sin^2 v_1^1(k)}, \\
&R_1 = 2 (p_1^1(k) - \E(p_1^1(k)) = 2(F(v_1^1(k), k) - E(v_1^1(k), k)), \\
&R_2 = 0.
\end{align*}
Thus $\Exp(N_{26}') = M_{26}'$. Moreover, the mapping $\map{\Exp}{N_{26}'}{M_{26}'}$ decomposes into the chain 
\begin{align*}
&N_{26}' \ \stackrel{(*)}{\to} \ \G_1 \ \stackrel{(**)}{\to} \ M_{26}',\\
&(*) \ : \ k \mapsto (\th = 2 \arcsin(k \sin v_1^1(k)), \ R_1 = 2 (F(v_1^1(k),k) - E(v_1^1(k),k)), \ R_2 = 0), \\
&(**) \ : \ (\th, R_1, R_2) \mapsto (2 \pi - \th, R_1, R_2).
\end{align*}
The mapping $(*)$ is a diffeomorphism by Lemma~\ref{lem:Gam1}. Thus $\map{\Exp}{N_{26}'}{M_{26}'}$ is a diffeomorphism.
\end{proof}

\subsection{Exponential mapping of the set $N_{52}'$}
\begin{lemma}
\label{lem:N3'M3'}
\begin{itemize}
\item[$(1)$]
The functions $R_1^1(\th)$, $R_1^2(\th)$ defined in~\eq{R11}, \eq{R12} satisfy the inequality
$$
R_1^2(\th) < R_1^1(\th), \qquad \th \in (0, \pi).
$$
\item[$(2)$]
The mapping $\map{\Exp}{N_{52}'}{M_{52}'}$ is a diffeomorphism of 2-dimensional manifolds, where
$$
M_{52}' = \{ q \in M \mid \th \in (\pi, 2 \pi),\ R_1 \in (R_1^2(2 \pi - \th), R_1^1(2 \pi - \th)), \ R_2 = 0 \}.
$$
\end{itemize}
\end{lemma}
\begin{proof}
We have
$$
N_{52}'= \{ \nu \in N_2^+ \mid \tau = 0, \ v \in (0, v_1^1(k)), \ k \in (0,1)\},
$$
where $v = \am(p,k)$. Thus
$$
N_{52}' \subset T = \{ \nu \in N_2^+ \mid \tau = 0, \ v \in [0, \pi], \ k \in [0,1]\},
$$
and in the 2-dimensional topology of $T$
\begin{align*}
&\partial N_{52}' = \cup _{i=1}^4 n_i, \\
&n_1 = \{ \nu \in N_2^+ \mid \tau = 0, \ v = 0, \ k \in [0,1]\},\\
&n_2 = \{ \nu \in N_2^+ \mid \tau = 0, \ v \in [0, \pi], \ k = 0\},\\
&n_3 = \{ \nu \in N_2^+ \mid \tau = 0, \ v = v_1^1(k), \ k \in [0,1]\},\\
&n_4 = \{ \nu \in N_2^+ \mid \tau = 0, \ v \in [0, \pi/2], \ k = 1\},
\end{align*}
see Fig.~\ref{fig:N3'}.

\twofiglabelh
{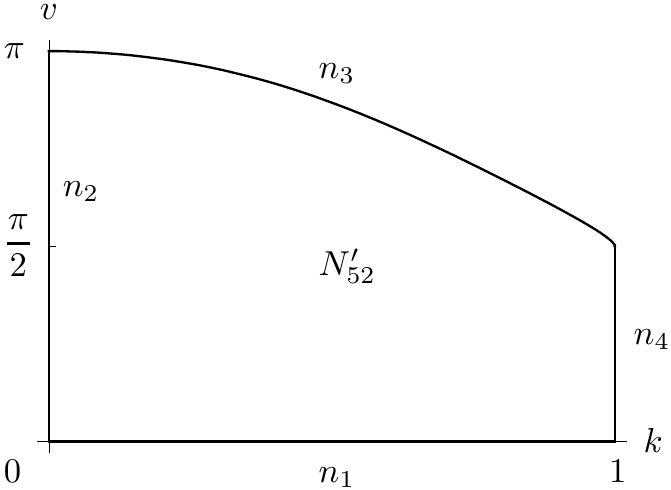}{Domain  $N_{52}'$}{fig:N3'}
{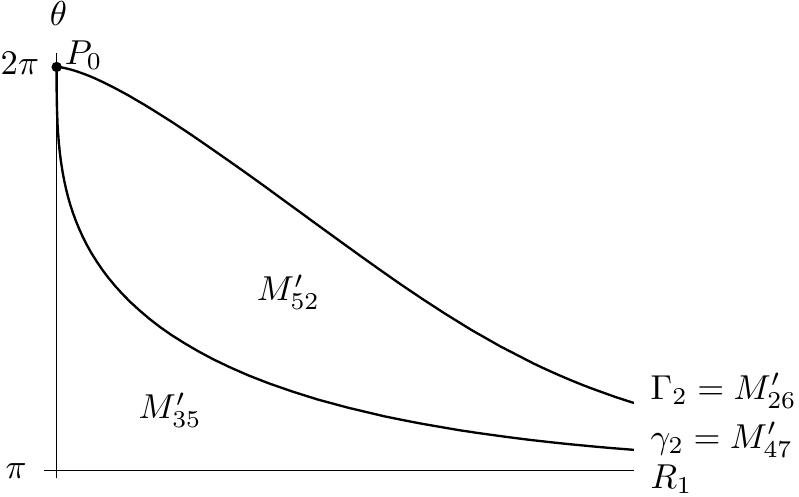}{Curves $\g_2$ and $\G_2$}{fig:gam2Gam2}
{4.5}

By formulas~(5.7)--(5.12)~\cite{max_sre2}, the exponential mapping in the domain $N_{52}'$ reads as follows:
\begin{align*}
&\sin(\th/2) = k \sin v, && \cos(\th/2) = - \sqrt{1 - k^2 \sin^2 v}, \\
&R_1 = 2 (F(v,k) - E(v,k)), && R_2 = 0.
\end{align*}
Thus 
\begin{align*}
&\Exp(N_{52}') \subset S = \{ q \in M \mid \th \in (\pi, 2 \pi), \ R_1 > 0, \ R_2 = 0 \},\\
&\Exp(n_1) = \Exp(n_2) = P_0  = \{ q \in M \mid \th =  2 \pi, \ R_1 = 0, \ R_2 = 0 \},\\
&\Exp(n_3) = \overline{M_{26}'} = \overline{\G_2}, \qquad \G_2 \eqdef M_{26}',\\
&\Exp(n_4) = \overline{M_{47}'} = \overline{\g_2}, \qquad \g_2 \eqdef M_{47}'.
\end{align*}
By Th.~\ref{th:conjC}, the mapping $\restr{\Exp}{N_{52}'}$ is nondegenerate, thus $\Exp(N_{52}')$ is an open connected domain in $S$, with $\partial \Exp(N_{52}') \subset \Exp(\partial N_{52}') = \overline{\G_2} \cup \overline{\g_2}$.

The curves $\overline{\G_2}$ and $\overline{\g_2}$ intersect one another at the point $P_0$. We show that they have no other intersection points. By contradiction, assume that the curves $\overline{\G_2}$ and $\overline{\g_2}$ have intersection points distinct from $P_0$, then the domain $\Exp(N_{52}')$ is bounded by finite arcs of the curves $\overline{\G_2}$ and $\overline{\g_2}$, i.e., there exists a point $P_1 \in \g_2 \cap \G_2$, $P_1 \neq P_0$, such that $\partial \Exp (N_{52}') = P_0 \g_2 P_1 \, \cup \, P_0 \G_2 P_1$. Then $\overline{\Exp(N_{52}')}$ does not contain the curves $\g_2$, $\G_2$. This is a contradiction to the diffeomorphic property of the mapping $\map{\Exp}{n_3 = N_{26}'}{\G_2 = M_{26}'}$ and 
$\map{\Exp}{n_4 = N_{47}'}{\g_2 = M_{47}'}$, see Lemmas~\ref{lem:N44'M44'} and~\ref{lem:N2'M2'} respectively.

Consequently, $\overline{\g_2} \cap \overline{\G_2} = P_0$, and the domain $\Exp(N_{52}')$ is bounded by the curves $\overline{\g_2}$, $\overline{\G_2}$.

The equalities $\ds\der{R_1^1}{\th}(0) = + \infty$, $\ds\der{R_1^2}{\th}(0) = 0$ (see Lemma~\ref{lem:Gam1} and Eq.~\eq{dR12}) imply that $R_1^2(\th) < R_1^1(\th)$ for sufficiently small $\th > 0$. Further, the representations
\begin{align*}
&\G_2 = M_{26}' =  \{q \in M \mid \th \in (\pi, 2 \pi), \ R_1 = R_1^1(2 \pi - \th), \ R_2 = 0 \}, \\
&\g_2 = M_{47}' =  \{q \in M \mid \th \in (\pi, 2 \pi), \ R_1 = R_1^2(2 \pi - \th), \ R_2 = 0 \}
\end{align*}
imply the required inequality
$$
R_1^2(\th) < R_1^1(\th), \qquad \th \in (0, \pi),
$$
and the equality $\Exp(N_{52}') = M_{52}'$.

Since the mapping $\map{\Exp}{N_{52}'}{M_{52}'}$ is nondegenerate and proper, and the domains $N_{52}'$, $M_{52}'$ are open, connected, and simply connected, it follows that this mapping is a diffeomorphism.
\end{proof}

The mutual disposition of the curves $\g_2 = M_{47}'$ and $\G_2 = M_{26}'$ is shown at Fig.~\ref{fig:gam2Gam2}.

\subsection{Decomposition of the set $M'$}
\label{subsec:decompM'}
Now we have the functions $R_1^2(\th) < R_1^1(\th)$ required for definition of the following decomposition:
\be{M'=Mi'}
M' = \cup_{i=1}^{58} M_i',
\ee
where the subsets $M_i'$ are defined by Table~\ref{tab:Mi'}. Notice that some of the sets $M_i'$ coincide between themselves, unlike the sets $N_i'$, see~\eq{N'=N_i'}.

\begin{table}
$$
\begin{array}{|c|c|c|c|c|c|c|}
\hline
M_i' & M_{1}' = M_{6}' & M_{2}' = M_{5}' & M_{3}' = M_{8}'& M_{4}' =  M_{7}'& M_9' =  M_{10}'& M_{11}' =  M_{12}'\\
\hline
\th  & \pi & \pi &  \pi &\pi &(\pi, 2 \pi) &(0, \pi)\\
\hline
R_1 & (0, +\infty) & (-\infty,0) & (-\infty,0) & (0, +\infty) & (R_1^1(2 \pi - \th), + \infty)& (-\infty, -R_1^1(\th))\\
\hline
R_2 & (-\infty, 0) & (-\infty,0)& (0, +\infty) & (0, +\infty) & 0& 0\\
\hline
\end{array}
$$
$$
\begin{array}{|c|c|c|c|c|c|}
\hline
M_i' & M_{13}' = M_{14}' & M_{15}' = M_{16}' & M_{17}' = M_{23}' & M_{18}' = M_{22}' & M_{19}' = M_{21}' \\
\hline
\th  & (0, \pi) & (\pi, 2 \pi) &  \pi & \pi & \pi \\
\hline
R_1 & (R_1^1(\th), +\infty) & (-\infty,-R_1^1(2 \pi - \th)) & (0, + \infty) & 0  & (-\infty,0)\\
\hline
R_2 & 0 & 0 & 0 & (-\infty,0)& 0 \\
\hline
\end{array}
$$
$$
\begin{array}{|c|c|c|c|c|c|c|c|}
\hline
M_i' & M_{20}' = M_{24}'& M_{25}' = M_{27}' & M_{26}' & M_{28}'  & M_{29}' = M_{31}'& M_{30}' & M_{32}' \\
\hline
\th  & \pi& 0 & (\pi, 2\pi) &  (0, \pi) & 0 &(0, \pi) & (\pi, 2\pi)  \\
\hline
R_1 & 0& (-\infty,0) & R_1^1(2 \pi - \th) & -R_1^1(\th) & (0, +\infty) & R_1^1(\th) & -R_1^1(2\pi -\th) \\
\hline
R_2 & (0, +\infty) & 0 & 0 & 0 & 0  & 0& 0  \\
\hline
\end{array}
$$
$$
\begin{array}{|c|c|c|c|c|c|c|}
\hline
M_i' & M_{33}' = M_{34}' & M_{35}' & M_{36}' & M_{37}'& M_{38}'& M_{39}'    \\
\hline
\th  & \pi & (\pi, 2\pi) & (\pi, 2\pi) &  (\pi, 2\pi) & (\pi, 2\pi) & (0,\pi)  \\
\hline
R_1 & 0 & (0,R_1^2(2 \pi - \th)) & 0 & (-R_1^2(2 \pi - \th),0) & 0 & (-R_1^2(\th),0)  \\
\hline
R_2 & 0 & 0 & (- \infty,0) & 0  & (0,+\infty)& 0  \\
\hline
\end{array}
$$
$$
\begin{array}{|c|c|c|c|c|c|c|c|c|}
\hline
M_i' & M_{40}' & M_{41}'  & M_{42}'  & M_{43}' & M_{44}'& M_{45}'  & M_{46}' & M_{47}'\\
\hline
\th  & (0,\pi) & (0, \pi) & (0, \pi) &  (\pi,  2\pi) & (0,  \pi) & 0 & 0 &  (\pi,  2\pi)\\
\hline
R_1 & 0 &  (0,R_1^2(\th)) & 0 & 0 & 0 & 0  & 0 &R_1^2(2 \pi - \th) \\
\hline
R_2 & (-\infty,0) & 0 & (0,+\infty) & 0  & 0& (0,+\infty) & (- \infty,0) & 0\\
\hline
\end{array}
$$
$$
\begin{array}{|c|c|c|c|c|c|c|}
\hline
M_i' & M_{48}' & M_{49}'  & M_{50}' & M_{51}' & M_{52}'  & M_{53}'   \\
\hline
\th  &  (\pi,  2\pi) & (0, \pi) & (0, \pi) &  0 & (\pi, 2 \pi) & 0  \\
\hline
R_1 &-R_1^2(2 \pi - \th) & -R_1^1(\th) & R_1^2(\th)& (- \infty,0) & (R_1^2(2 \pi- \th), R_1^1(2 \pi- \th)) &(- \infty,0) \\
\hline
R_2 & 0& 0 & 0 & (- \infty,0)  & 0& (0,+\infty)  \\
\hline
\end{array}
$$
$$
\begin{array}{|c|c|c|c|c|c|}
\hline
M_i' & M_{54}' & M_{55}'  & M_{56}'  & M_{57}' &  M_{58}'   \\
\hline
\th  & (0,  \pi)  & 0 & (0,  \pi) &  0 & (\pi, 2 \pi)   \\
\hline
R_1 & (-R_1^1(\th),-R_1^2(\th)) & (0,+\infty) & (R_1^2(\th), R_1^1(\th)) & (0,+\infty) & (-R_1^1(2 \pi - \th),-R_1^2(2 \pi -\th))  \\
\hline
R_2 & 0 & (-\infty, 0) & 0 & (0,+\infty) & 0   \\
\hline
\end{array}
$$
\caption{Definition of sets $M_i'$}\label{tab:Mi'}
\end{table}

The structure of decomposition~\eq{M'=Mi'} in the surfaces $\{\th = 0\}$, $\{\th = \pi \}$, $\{R_1 = 0\}$, $\{R_2=0\}$ is shown respectively at Figs.~\ref{fig:th=0}, \ref{fig:th=pi}, \ref{fig:R1=0}, \ref{fig:R2=0}.

\twofiglabelh
{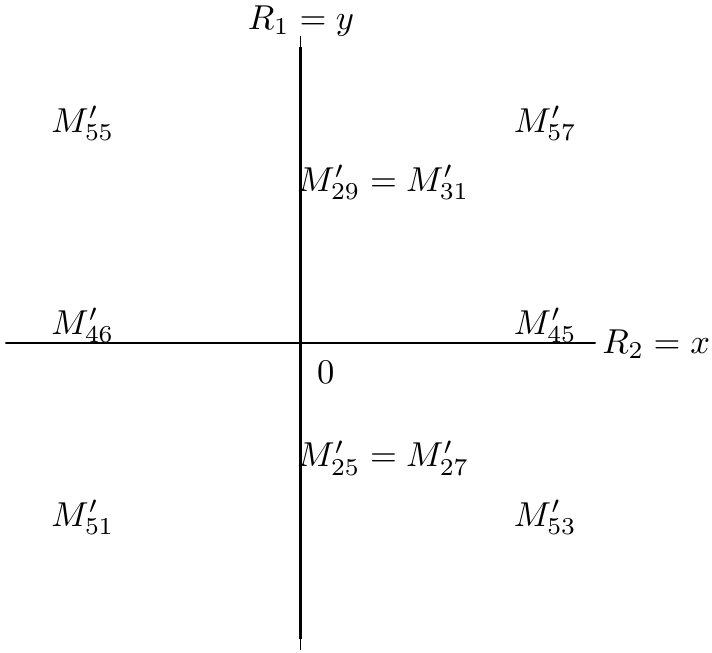}{Decomposition of surface $\{\th = 0\}$}{fig:th=0}
{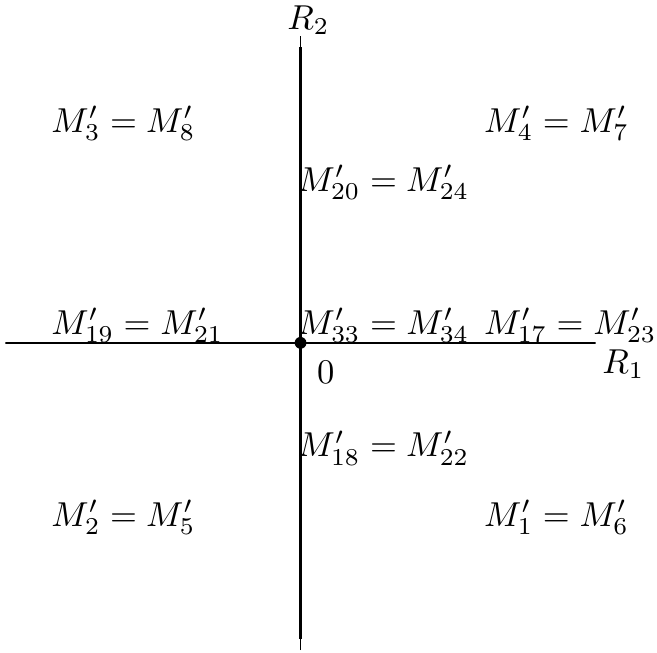}{Decomposition of surface $\{\th = \pi\}$}{fig:th=pi}
{6}

\onefiglabelsizen{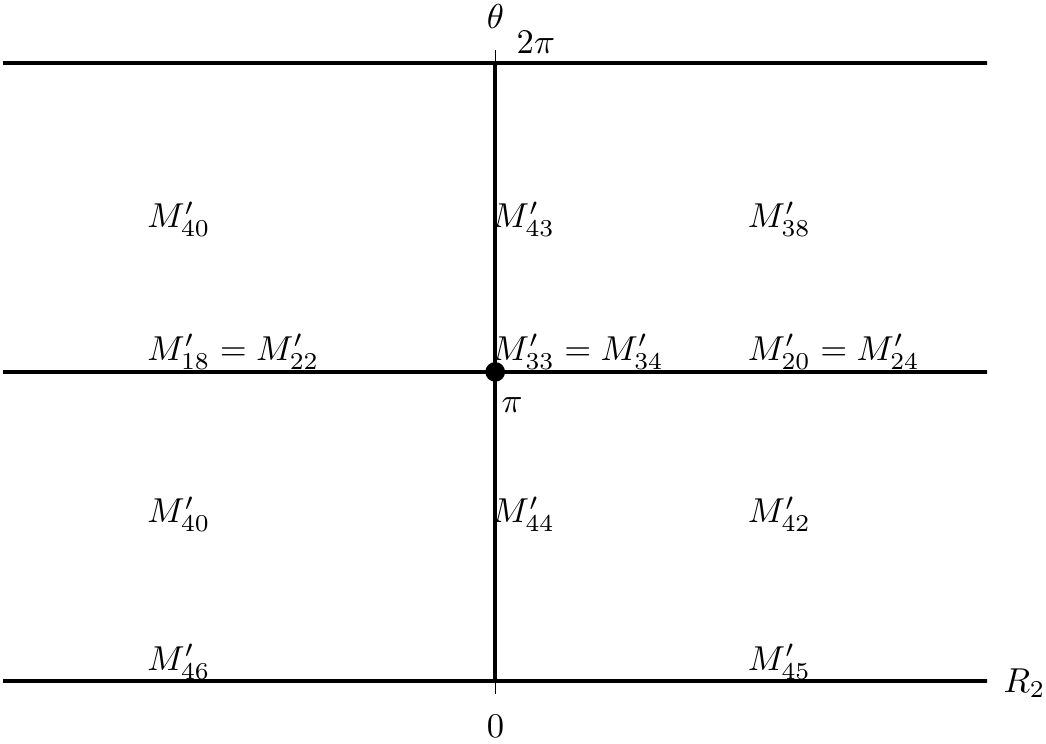}{Decomposition of surface $\{R_1 = 0\}$}{fig:R1=0}{7}

\onefiglabelsizen{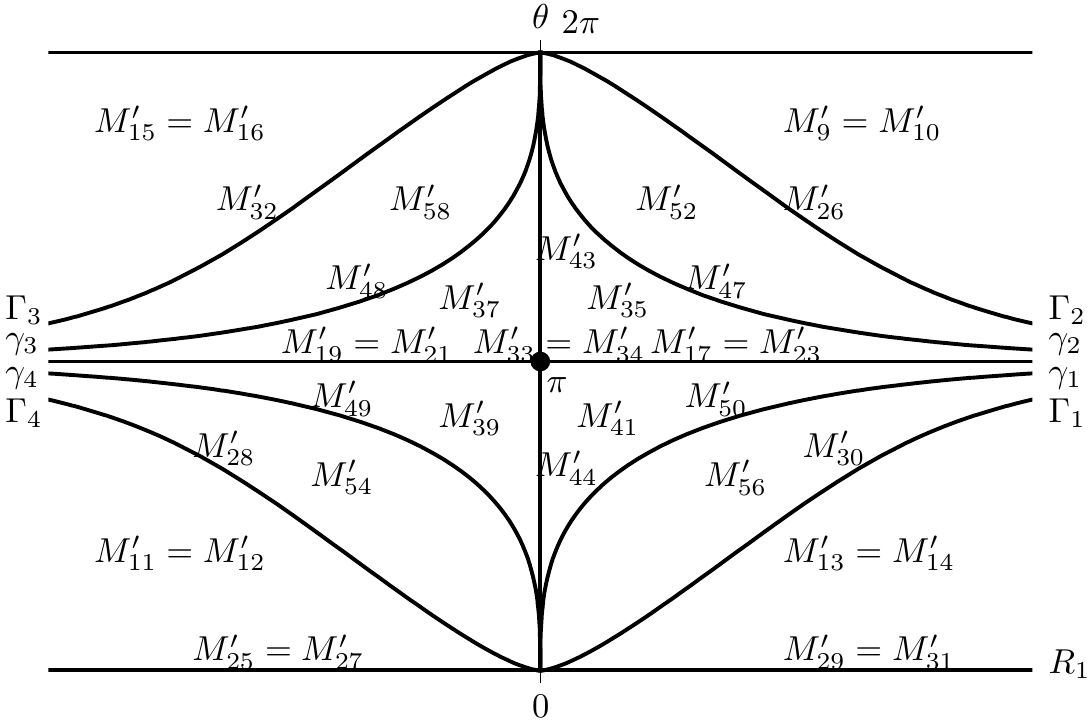}{Decomposition of surface $\{R_2 = 0\}$}{fig:R2=0}{7}

\subsection{Exponential mapping of the set $N_{36}'$}
\begin{lemma}
\label{lem:N17'M17'}
The mapping $\map{\Exp}{N_{36}'}{M_{36}'}$ is a diffeomorphism of 2-dimensional manifolds.
\end{lemma}
\begin{proof}
By formulas~(5.2)--(5.6)~\cite{max_sre2}, exponential mapping in the domain $N_{36}'$ reads as follows:
\begin{align*}
&\sin(\th/2) = \sqrt{1-k^2} \ss/\dd, && \cos(\th/2) = - \cc/\dd, \\
&R_1 = 0, && R_2 = - 2 f_2(p,k)/(k \dd),
\end{align*}
where $f_2(p,k) = k^2 \cc \ss + \dd(p-\E(p)) > 0$ by Lemma~5.2~\cite{max_sre2}, thus $\Exp(N_{36}') \subset M_{36}'$. 

In the topology of the manifold $\{R_1 = 0\}$, we have:
\begin{align*}
&\partial N_{36}' = \cup_{i=1}^4 n_i, \\
&n_1 = \{ \nu \in N_1^0 \mid \tau = K, \ p = 0, \ k \in [0, 1]\}, \\
&n_2 = \{ \nu \in N_1^0 \mid \tau = K, \ p \in  [0,\pi/2], \ k = 0\}, \\
&n_3 = \{ \nu \in N_1^0 \mid \tau = K, \ p = K, \ k \in [0, 1)\}, \\
&n_4 = \{ \nu \in N_1^0 \mid \tau = K, \ p \in [0, = \infty), \ k = 1\}, 
\end{align*}
see Fig.~\ref{fig:N17'}.

Further, we have $\Exp(n_i) = m_i$, $i = 1,\dots, 4$, where
\begin{align*}
&m_1 = \{ q \in M \mid \th = 2 \pi, \ R_1 = 0, \ R_2 = 0 \}, \\ 
&m_2 = \{ q \in M \mid \th \in[\pi, 2 \pi], \ R_1 = 0, \ R_2 =0 \}, \\ 
&m_3 = \{ q \in M \mid \th =  \pi, \ R_1 = 0, \ R_2 \in (-\infty, 0] \}, \\ 
&m_4 = \{ q \in M \mid \th = 2 \pi, \ R_1 = 0, \ R_2 \in (-\infty, 0] \}, 
\end{align*}
see Fig.~\ref{fig:M17'}.

\twofiglabel
{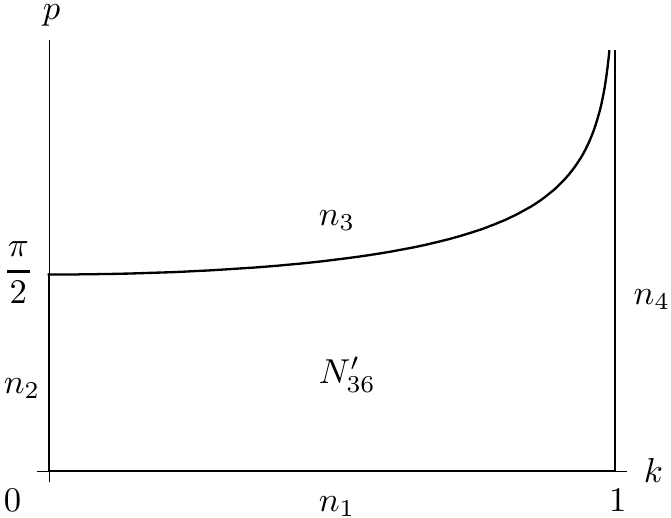}{Domain $N_{36}'$}{fig:N17'}
{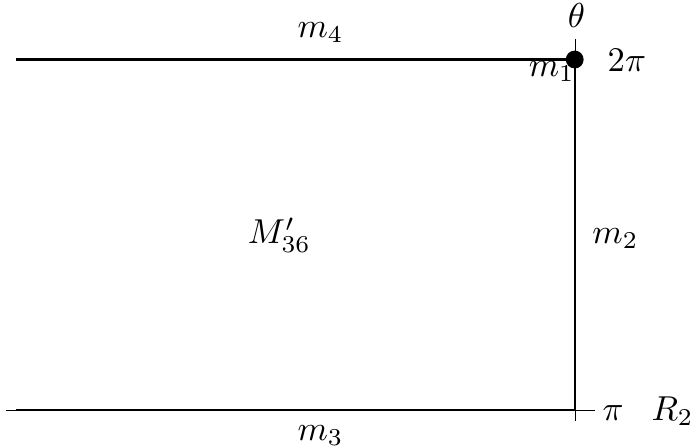}{Domain $M_{36}'$}{fig:M17'}

The mapping $\map{\Exp}{N_{36}'}{M_{36}'}$ is nondegenerate and proper, the domains $N_{36}'$, $M_{36}'$ are open (in the 2-dimensional topology), connected and simply connected, thus it is a diffeomorphism.
\end{proof}


\subsection{Exponential mapping of the set $N_{53}'$}
\begin{lemma}
\label{lem:N7'M7'}
The mapping $\map{\Exp}{N_{53}'}{M_{53}'}$ is a diffeomorphism of 2-dimensional manifolds.
\end{lemma}

\begin{proof}
The argument follows similarly to the proof of Lemma~\ref{lem:N17'M17'} via the following coordinate representation of the exponential mapping in the domain $N_{53}'$:
$$
\th = 0, \qquad R_1 = - 2 \sqrt{1-k^2} (p - \E(p))/\dd < 0, \qquad R_2 = - 2 k f_1(p,k)/\dd,
$$
where $f_1(p,k) = \cc (\E(p) - p) - \dd \ss < 0$ for $p \in (0, p_1^1)$, see Cor.~5.1~\cite{max_sre2}.
\end{proof}

\subsection{Exponential mapping of the set $N_{18}'$}
\begin{lemma}
\label{lem:N55'M55'}
The mapping $\map{\Exp}{N_{18}'}{M_{18}'}$ is a diffeomorphism of 1-dimensional manifolds.
\end{lemma}
\begin{proof}
By formulas~(5.2)--(5.6)~\cite{max_sre2}, we have in the set $N_{18}'$:
$$
\th = \pi, \qquad R_1 =   0, \qquad R_2 = - 2 /k (K(k) - E(k)),
$$
and the diffeomorphic property of $\restr{\Exp}{N_{18}'}$ follows as usual from its nondegeneracy and properness, and topological properties of the sets $N_{18}'$, ${M_{18}'}$.
\end{proof}

\subsection{Exponential mapping of the sets $N_{33}'$, $N_{34}'$}
\begin{lemma}
\label{lem:N3940}
The mappings $\map{\Exp}{N_{33}'}{M_{33}'}$ and $\map{\Exp}{N_{34}'}{M_{34}'}$ are diffeomorphisms of $0$-dimensional manifolds.
\end{lemma}
\begin{proof}
Obvious.
\end{proof}

\subsection{Exponential mapping of the set $N_{17}'$}
\begin{lemma}
\label{lem:N43'M43'}
The mapping $\map{\Exp}{N_{17}'}{M_{17}'}$ is a diffeomorphism of 1-dimensional manifolds.
\end{lemma}

\begin{proof}
The statement follows as in  the proof of Lemma~\ref{lem:N55'M55'} via the following coordinate representation of the exponential mapping in the domain $N_{17}'$:
$$
\th = \pi, \qquad R_1 = 2/k(K(k) - E(k)), \qquad R_2 = 0.
$$
\end{proof}

\subsection{Exponential mapping of the set $N_{27}'$}
\begin{lemma}
\label{lem:N47'M47'}
The mapping $\map{\Exp}{N_{27}'}{M_{27}'}$ is a diffeomorphism of 2-dimensional manifolds.
\end{lemma}
\begin{proof}
Formulas~(5.7)--(5.12)~\cite{max_sre2} yield:
$$
\th = \pi, \qquad R_1 =   - 2 \sqrt{1-k^2} \restr{(p - \E(p))/\dd}{p = p_1^1(k)}, \qquad R_2 = 0.
$$
Since $\tau = K$, then Lemma~\ref{lem:J2p=p11} gives $J < 0$, thus the mapping $\restr{\Exp}{N_{27}'}$ is nondegenerate.
Then the diffeomorphic property of $\map{\Exp}{N_{27}'}{M_{27}'}$ follows as usual.
\end{proof}

\subsection{Exponential mapping of the set $N_{1}'$}
\begin{lemma}
\label{lem:N21'M21'}
The mapping $\map{\Exp}{N_{1}'}{M_{1}'}$ is a diffeomorphism of 2-dimensional manifolds.
\end{lemma}
\begin{proof}
Formulas~(5.2)--(5.6)~\cite{max_sre2} yield:
\begin{align*}
&\th = \pi, \qquad R_1 =   2(K(k)-E(k))\tc/(k \td) > 0, \\ 
&R_2 = -2 \sqrt{1-k^2} (K(k) - E(k)) \ts/(k \td) < 0,
\end{align*}
and the statement follows as usual since 
 the mapping $\restr{\Exp}{N_{1}'}$ is nondegenerate.
\end{proof}

\subsection{Exponential mapping of the set $N_{10}'$}
\begin{lemma}
\label{lem:N25'M25'}
The mapping $\map{\Exp}{N_{10}'}{M_{10}'}$ is a diffeomorphism of 2-dimensional manifolds.
\end{lemma}
\begin{proof}
By formulas~(5.7)--(5.12)~\cite{max_sre2} we get:
\begin{align*}
&\sin(\th/2) = k \sn p_1^1 \tc /\sqrt{\Del}> 0, 
&&\cos(\th/2) = - \dn p_1^1 /\sqrt{\Del}< 0, \\
&R_1 = \restr{2(p - \E(p)) \td /\sqrt{\Del} }{p = p_1^1} > 0, 
&& R_2 = 0,
\end{align*}
where $\Del = 1 - k^2 \ssp \tsp$, and the statement follows by standard argument since 
 $\restr{\Exp}{N_{10}'}$ is nondegenerate and proper.
\end{proof}

\subsection{Action of the group of reflections in the preimage and image of the exponential mapping}
In order to extend the results of the preceding subsections to all 58 pairs $(N_i', M_i')$, we describe the action of the group of reflections $G = \{\Id, \eps^1, \dots, \eps^7\}$ on these sets.

\begin{table}[htbp]
$$
\begin{array}{|c|c|c|c|c|c|}
\hline
D & N_{35}' & N_{47}' & N_{52}' &  N_{17}' & N_{26}'    \\
\hline
\eps^1(D) & N_{37}' & N_{48}' & N_{58}' &  N_{19}' & N_{32}'    \\
\hline
\eps^4(D) & N_{39}' & N_{49}' & N_{54}' &  N_{21}' & N_{28}'    \\
\hline
\eps^5(D) & N_{41}' & N_{50}' & N_{56}' &  N_{23}' & N_{30}'    \\
\hline
 \end{array}
$$
\caption{Action of $\eps^1$, $\eps^4$, $\eps^5$ on $N_{35}'$, $N_{47}'$, $N_{52}'$, $N_{17}'$, $N_{26}'$}\label{tab:eps145N123}
\end{table}

\twotablabel
{$
\begin{array}{|c|c|c|}
\hline
D   & N_{36}' & N_{18}'    \\
\hline
\eps^2(D)   &  N_{38}' & N_{20}'    \\
\hline
\eps^4(D)   &  N_{40}' & N_{22}'    \\
\hline
\eps^6(D)   &  N_{42}' & N_{24}'    \\
\hline
 \end{array}
$}
{Action of $\eps^2$, $\eps^4$, $\eps^6$ on   $N_{36}'$, $N_{18}'$}
{tab:eps246N1755}
{$
\begin{array}{|c|c|c|}
\hline
D  &   N_{53}' & N_{27}'    \\
\hline
\eps^1(D)   &  N_{57}' & N_{31}'    \\
\hline
\eps^2(D)   &  N_{51}' & N_{25}'    \\
\hline
\eps^3(D)   &  N_{55}' & N_{29}'    \\
\hline
 \end{array}
$}
{Action of $\eps^1$, $\eps^2$, $\eps^3$ on   $N_{53}'$, $N_{27}'$}
{tab:eps123N747}

\begin{table}[htbp]
$$
\begin{array}{|c|c|c|c|c|c|}
\hline
D & M_{35}' & M_{47}' & M_{52}' &  M_{17}' & M_{26}'    \\
\hline
\eps^1(D) & M_{37}' & M_{48}' & M_{58}' &  M_{19}' & M_{32}'    \\
\hline
\eps^4(D) & M_{39}' & M_{49}' & M_{54}' &  M_{21}' & M_{28}'    \\
\hline
\eps^5(D) & M_{41}' & M_{50}' & M_{56}' &  M_{23}' & M_{30}'    \\
\hline
 \end{array}
$$
\caption{Action of $\eps^1$, $\eps^4$, $\eps^5$ on $M_{35}'$, $M_{47}'$, $M_{52}'$, $M_{17}'$, $M_{26}'$}\label{tab:eps145M123}
\end{table}

\twotablabel
{$
\begin{array}{|c|c|c|}
\hline
D   & M_{36}' & M_{18}'    \\
\hline
\eps^2(D)   &  M_{38}' & M_{20}'    \\
\hline
\eps^4(D)   &  M_{40}' & M_{22}'    \\
\hline
\eps^6(D)   &  M_{42}' & M_{24}'    \\
\hline
 \end{array}
$}
{Action of $\eps^2$, $\eps^4$, $\eps^6$ on   $M_{36}'$, $M_{18}'$}
{tab:eps246M1755}
{$
\begin{array}{|c|c|c|}
\hline
D  &   M_{53}' & M_{27}'    \\
\hline
\eps^1(D)   &  M_{57}' & M_{31}'    \\
\hline
\eps^2(D)   &  M_{51}' & M_{25}'    \\
\hline
\eps^3(D)   &  M_{55}' & M_{29}'    \\
\hline
 \end{array}
$}
{Action of $\eps^1$, $\eps^2$, $\eps^3$ on   $M_{53}'$, $M_{27}'$}
{tab:eps123M747}

\twotablabel
{$
\begin{array}{|c|c|c|}
\hline
D  &   N_{1}' & N_{10}'    \\
\hline
\eps^1(D)   &  N_{2}' & N_{15}'    \\
\hline
\eps^2(D)   &  N_{4}' & N_{9}'    \\
\hline
\eps^3(D)   &  N_{3}' & N_{16}'    \\
\hline
\eps^4(D)   &  N_{5}' & N_{12}'    \\
\hline
\eps^5(D)   &  N_{6}' & N_{13}'    \\
\hline
\eps^6(D)   &  N_{8}' & N_{11}'    \\
\hline
\eps^7(D)   &  N_{7}' & N_{14}'    \\
\hline
 \end{array}
$}
{Action of $\eps^1$, \dots, $\eps^7$ on   $N_{1}'$, $N_{10}'$}
{tab:eps17N2125}
{$
\begin{array}{|c|c|c|}
\hline
D  &   M_{1}' & M_{10}'    \\
\hline
\eps^1(D)   &  M_{2}' & M_{15}'    \\
\hline
\eps^2(D)   &  M_{4}' & M_{9}'    \\
\hline
\eps^3(D)   &  M_{3}' & M_{16}'    \\
\hline
\eps^4(D)   &  M_{5}' & M_{12}'    \\
\hline
\eps^5(D)   &  M_{6}' & M_{13}'    \\
\hline
\eps^6(D)   &  M_{8}' & M_{11}'    \\
\hline
\eps^7(D)   &  M_{7}' & M_{14}'    \\
\hline
 \end{array}
$}
{Action of $\eps^1$, \dots, $\eps^7$ on   $M_{1}'$, $M_{10}'$}
{tab:eps17M2125}

\begin{theorem}
\label{th:epsNM}
Tables~\ref{tab:eps145N123}, \ref{tab:eps246N1755}, \ref{tab:eps123N747}, \ref{tab:eps17N2125} and~\ref{tab:eps145M123}, \ref{tab:eps246M1755}, \ref{tab:eps123M747}, \ref{tab:eps17M2125} define diffeomorphisms between the corresponding manifolds $N_i'$ and $M_i'$.
\end{theorem} 
\begin{proof}
Follows from definitions of the manifolds $N_i'$ and $M_i'$ (Subsections~\ref{subsec:decompN'} and~\ref{subsec:decompM'}) and Propositions~4.3, 4.4~\cite{max_sre2} describing action of the reflections $\eps^i \in G$ in the image and preimage of the exponential mapping. Moreover, in the coordinates $(\th, R_1, R_2)$ action of the reflections is described by Table~\ref{tab:eps17R12th}.
\end{proof}

\begin{table}[htbp]
$$
\begin{array}{|c|c|c|c|c|c|c|c|}
\hline
  & \eps^1 & \eps^2 & \eps^3 &  \eps^4 & \eps^5 & \eps^6 & \eps^7   \\
\hline
R_1 & -R_1 & R_1 & -R_1 &  -R_1 & R_1 & -R_1 & R_1   \\
\hline
R_2 & R_2 & -R_2 & -R_2 &  R_2 & R_2 & -R_2& -R_2   \\
\hline
\th & \th & \th & \th &  2 \pi - \th & 2 \pi - \th   & 2 \pi - \th & 2 \pi - \th  \\
\hline
 \end{array}
$$
\caption{Action of $\eps^1$, \dots, $\eps^7$ on $M = \{(R_1, R_2, \th)\}$}\label{tab:eps17R12th}
\end{table}

\subsection[The final result for exponential mapping of the sets $N_i'$]
{The final result for exponential mapping \\of the sets $N_i'$}

\begin{theorem}
\label{th:ExpNi'Mi'}
For any $i \in I$, the mapping $\map{\Exp}{N_i'}{M_i'}$ is a diffeomorphism of manifolds of appropriate dimension 2, 1, or 0.
\end{theorem}
\begin{proof}
For $i \in \{35, 47, 26, 52, 36, 53, 18, 17, 27, 1, 10\}$ the statement follows from Lemmas~\ref{lem:N1'M1'}, \ref{lem:N2'M2'}, \ref{lem:N44'M44'}, \ref{lem:N3'M3'}, \ref{lem:N17'M17'}, \ref{lem:N7'M7'}, \ref{lem:N55'M55'}, \ref{lem:N43'M43'}, \ref{lem:N47'M47'}, \ref{lem:N21'M21'}, \ref{lem:N25'M25'} respectively.

For $i \in \{33, 34\}$ the statement was proved in Lemma~\ref{lem:N3940}.

For all the rest $i$ the statement follows from the above lemmas and Th.~\ref{th:epsNM} since the reflections $\eps^i \in G$ are symmetries of the exponential mapping, see Propos.~4.5~\cite{max_sre2}.
\end{proof}

\subsection{Reflections $\eps^k$ as permutations}

In addition to the index sets $I$, $C$, $J$, $R$, $X$ introduced in~\eq{ICJ}, \eq{RX}, we will need also the set
$$
T = \{(i,j) \in I \times I \mid i < j, \ M_i' = M_j' \}.
$$
From the definition of the sets $M_i'$ in Subsec.~\ref{subsec:decompM'} we obtain the explicit representation:
\begin{multline*}
T = \{(1, 6), \ (2,5), \ (3, 8), \ (4, 7), \ (9, 10),  \ (11,12), \ (13, 14), \ (15, 16), \\
  \ (17, 23), \ (18,22), \  (19,21), \ (20,24), \ (25, 27), \ (29, 31),   \ (33, 34)\}.
\end{multline*}

Notice that $X = \{ i \in I \mid \exists j \in I \ : \ (i,j) \in T \text{ or } (j,i)\in T\}$.

Now we show that reflections $\eps^k \in G$ permute elements in any pair $(i,j) \in T$.

We will need  multiplication Table~\ref{tab:G} in the group $G$, which follows from definitions of the reflections $\eps^k$ (Sec.~4~\cite{max_sre2}). The lower diagonal entries of the table are not filled since $G$ is Abelian.

\begin{table}[htbp]
$$
\begin{array}{|c|c|c|c|c|c|c|c|}
\hline
 & \eps^1 & \eps^2 & \eps^3 &  \eps^4 & \eps^5 & \eps^6 & \eps^7   \\
\hline
\eps^1   & \Id & \eps^3 & \eps^2 &  \eps^5 & \eps^4 & \eps^7 & \eps^6   \\
\hline
\eps^2   &   & \Id & \eps^1 &  \eps^6 & \eps^7 & \eps^4 & \eps^5   \\
\hline
\eps^3   &    &   & \Id  &  \eps^7 & \eps^6 & \eps^5 & \eps^4   \\
\hline
\eps^4   &    &   &   &  \Id & \eps^1 & \eps^2 & \eps^3   \\
\hline
\eps^5   &    &   &   &  & \Id  & \eps^3 & \eps^2   \\
\hline
\eps^6   &    &   &  &   &  & \Id  & \eps^1  \\
\hline
\eps^7   &    &   &   &    &   &   & \Id   \\
\hline
 \end{array}
$$
\caption{Multiplication table in the group $G$}\label{tab:G}
\end{table}

\begin{lemma}
\label{lem:epskT}
For any $(i,j) \in T$ there exists a reflection $\eps^k \in G$ such that the following diagram is commutative:
$$
\xymatrix{
N_i' \ar[r]^{\Exp} \ar[d]_{\eps^k} & M_i'  \ar[d]^{\Id}\\
N_j' \ar[r]^{\Exp} & M_j'
}
\qquad\qquad\qquad
\xymatrix{
(\lam,t) \ar@{|->}[r]^{\Exp} \ar@{|->}[d]_{\eps^k} & q_t  \ar@{|->}[d]^{\Id}\\
(\lam^k,t)  \ar@{|->}[r]^{\Exp} & q_t
}
$$
\end{lemma}
\begin{proof}
From definitions of the sets $N_i'$ (Subsec.~\ref{subsec:decompN'}) and the reflections $\eps^k$ (Sec.~4~\cite{max_sre2}), Tables~\ref{tab:eps145N123}, \ref{tab:eps246N1755}, \ref{tab:eps123N747}, \ref{tab:eps17N2125}, \ref{tab:G} and Propos.~4.5~\cite{max_sre2}, we obtain the following indices $k$ of required symmetries $\eps^k$ for pairs $(i,j) \in T$:
\begin{align*}
&(i,j) \in \{(1,6), \ (2, 5), \ (3, 8), \ (4, 7), \ (17, 23), \ (19, 21)\} \then k = 5, \\
&(i,j) \in \{(9,10), \ (15,16), \ (11,12), \ (13,14),\ (25,27), \ (29,31)\} \then  k = 2, \\
&(i,j) \in \{(33,34), \ (18,22), \ (20, 24) \} \then  k = 4.
\end{align*}    
\end{proof}

\section{Solution to optimal control problem}
\label{sec:sol}
In this section we present the final results of this study of the sub-Riemannian problem on $\SE(2)$.

\subsection{Global structure of the exponential mapping}

We say that a mapping $\map{F}{X}{Y}$ is double if   any point $y \in Y$  has exactly two preimages:
$$
\forall \ y \in Y \qquad F^{-1}(y) = \{ x_1, x_2\}, \qquad x_1 \neq x_2.
$$

\begin{theorem}
\label{th:Expglob}
\begin{itemize}
\item[$(1)$]
There is the following decomposition of preimage of the exponential mapping $\map{\Exp}{\hN}{\hM}$:
\begin{align*}
&\hN = \tN \sqcup N', \\
&\tN = \sqcup_{i=1}^8 D_i, \\
&N' = \NMax \sqcup \Nconj \sqcup \Nrest, \\
&\NMax = \sqcup_{i\in X} N_i', \\
&\Nconj = \sqcup_{i\in J} N_i', \\
&\Nrest = \sqcup_{i\in R} N_i ', 
\intertext{and in the image of the exponential mapping:}
&\hM = \tM \sqcup M', \\
&\tM = \sqcup_{i=1}^8 M_i, \\
&M' = \MMax \sqcup \Mconj \sqcup \Mrest, \\
&\MMax = \cup_{i\in X} M_i', \\
&M_i' \cap M_j' \neq \emptyset, \ i < j \then (i,j) \in T, \quad \{i,j\} \subset X,\\
&(i,j) \in T \then M_i' = M_j', \\
&\Mconj = \sqcup_{i\in J} M_i', \\
&\Mrest = \sqcup_{i \in R} M_i'.  
\end{align*}
\item[$(2)$]
In terms of these decompositions the exponential mapping $\map{\Exp}{\hN}{\hM}$ has the following structure:
\begin{align}
&\map{\Exp}{D_i}{M_i} \text{ is a diffeomorphism $\forall \ i = 1, \dots, 8$,} \label{ExpDiMi}\\
&\map{\Exp}{N_i'}{M_i'} \text{ is a diffeomorphism $\forall \ i \in I$.} \label{ExpNi'Mi'}
\end{align}
Thus
\begin{align}
&\map{\Exp}{\tN}{\tM}  \text{ is a bijection,} \label{ExpNMtilde}\\
&\map{\Exp}{\NMax}{\MMax}  \text{ is a double mapping,} \label{ExpNMMax}\\
&\map{\Exp}{\Nconj}{\Mconj}  \text{ is a bijection,} \label{ExpNMconj}\\
&\map{\Exp}{\Nrest}{\Mrest}  \text{ is a bijection.} \label{ExpNMrest}
\end{align}
\item[$(3)$]
Any point $q \in \tM$ ($q \in \Mconj$, $q \in \Mrest$) has a unique preimage $\nu = \Exp^{-1}(q)$ for the mapping $\map{\restr{\Exp}{\hN}}{\hN}{\hM}$. Moreover, $\nu \in \tN$ (resp., $\nu \in \Nconj$, $\nu \in \Nrest$). 
\item[$(4)$]
Any point $q \in \MMax$   has exactly two preimages $\{\nu', \nu''\} = \Exp^{-1}(q)$ for the mapping $\map{\restr{\Exp}{\hN}}{\hN}{\hM}$. Moreover, $\nu', \nu'' \in \NMax$   and $\nu'' = \eps^k(\nu')$ for some $\eps^k \in G$.
\end{itemize}
\end{theorem}
\begin{proof}
Equalities in item (1) follow immediately from definitions of the corresponding decompositions.

(2) Property~\eq{ExpDiMi} was proved in Th.~\ref{th:ExpDidiffeo}, and property~\eq{ExpNMtilde} is its corollary, with account of item (1).
Property~\eq{ExpNi'Mi'} was proved in Th.~\ref{ExpNi'Mi'}, and properties~\eq{ExpNMMax}--\eq{ExpNMrest} are its corollaries, with account of item (1).

(3) The statement follows from~\eq{ExpNMtilde}, \eq{ExpNMMax}, \eq{ExpNMconj}, \eq{ExpNMrest}, with account of item (1).

(4) The statement follows from \eq{ExpNMMax}, \eq{ExpNMconj}, \eq{ExpNMrest}, and Lemma~\ref{lem:epskT}.
\end{proof}

\subsection{Optimal synthesis}

\begin{theorem}
\label{th:opt_synth}
Let $q \in \hM$.
\begin{itemize}
\item[$(1)$]
Let $q \in \tM \cup \Mconj \cup \Mrest = \hM \setminus \MMax$. Denote $\nu = (\lam,t) = \Exp^{-1}(q) \in \tN \cup \Nconj \cup \Nrest = \hN \setminus \NMax$. Then $q_s = \Exp(\lam,s)$, $s \in [0,t]$, is the unique optimal trajectory connecting $q_0$ with $q$. If $q \in \tM \cup \Mrest$, then $t < \tt(\lam)$; if $q \in \Mconj$, then $t = \tt(\lam) = \tconj(\lam)$.
\item[$(2)$]
Let $q \in \MMax$. Denote $\{\nu', \nu''\} = \Exp^{-1}(q) \subset \NMax$, $\nu' = (\lam',t) \neq \nu'' = (\lam'',t)$. Then there exist exactly two distinct optimal trajectories connecting $q_0$ and $q$; namely, $q_s' = \Exp(\lam',s)$ and $q_s'' = \Exp(\lam'',s)$, $s \in [0,t]$. Moreover, $t = \tt(\lam)< \tconj(\lam)$. 
\item[$(3)$]
An optimal trajectory $q_s = \Exp(\lam, s)$ is generated by the optimal controls 
$$
u_i(s) = h_i(\lam_s), \qquad \lam_s= e^{s \vh}(\lam), \qquad i = 1,2.
$$ 
Thus
$$
u_1(s) = \sin(\g_s/2), \qquad u_2(s) = -\cos(\g_s/2),
$$
where $\g_s$ is the solution to the equation of pendulum $\ddot{\g}_s = - \sin \g_s$ with the initial condition $(\g_0, \dot \g_0) = \lam$.
\end{itemize}
\end{theorem}
\begin{proof}
For any point $q \in \hM$ there exists an optimal trajectory $q_s = \Exp(\lam,s)$, $s \in [0,t]$, $\nu = (\lam,t) \in N$, such that $q_t = q$ and $t \leq \tcut(\lam)$. By Th.~5.4~\cite{max_sre2}, we have $t \leq \tt(\lam)$, thus $\nu \in \hN$.

(1) If $q \in \tM \sqcup \Mconj \sqcup \Mrest$, then by Th.~\ref{th:Expglob}, there exists a unique $\nu = (\lam,t) \in \hN$ such that $q = \Exp(\nu)$, moreover, $\nu \in \tN \sqcup \Nconj \sqcup \Nrest$. Consequently, $q_s = \Exp(\lam,s)$, $s \in [0,t]$, is a unique optimal trajectory connecting $q_0$ with $q$.

The inequality $t < \tt(\lam)$ for $\nu = (\lam,t) \in \tN \sqcup \Nrest$, and the equality $t = \tt(\lam) = \tconj(\lam)$ for $\nu \in \Nconj$ follow  from definitions of the sets $\tN$, $\Nrest$, $\Nconj$.

(2) 
If $q \in \MMax$, then the statement follows similarly to item (1) from Th.~\ref{th:Expglob} and definition of the set $\NMax$.

(3) 
The expressions for optimal controls were obtained in Sec.~2~\cite{max_sre2}.
\end{proof}

It follows from the definition of cut time that for any $\lam \in C$ and $t \in (0, \tcut(\lam))$, the trajectory $q(s) = \Exp(\lam,s)$ is optimal at the segment $s \in [0, t]$. For the case of finite $\tcut(\lam)$, we obtain a similar statement for $t = \tcut(\lam)$.

\begin{theorem}
\label{th:tcut}
If $\tcut(\lam) < + \infty$, then the extremal trajectory $\Exp(\lam,s)$ is optimal for $s \in [0, \tcut(\lam)]$.
\end{theorem}
\begin{proof}
Let $\tcut(\lam) = \tt(\lam) < + \infty$, i.e., $\lam \in C_1 \cup C_2 \cup C_4$, and let $t = \tcut(\lam)$. Then $(\lam,t) \in \NMax$, and the statement follows from item (2) of Th.~\ref{th:opt_synth}.
\end{proof}

\subsection{Cut locus}

Now we are able to describe globally the first Maxwell set
\begin{align*}
&\Max = \{ q \in M \mid \exists t > 0, \ \exists \text{ optimal trajectories }
q_s \not\equiv q_s', \ s \in [0,t],\\ 
&\hspace{8cm}
 \text{ such that } q_t = q_t' = q \},\\
\intertext{the cut locus}
&\Cut  = \{ \Exp(\lam,t) \mid \lam \in C, \ t = \tcut(\lam) \},  
\intertext{and its intersection with caustic (the frst conjugate locus)}
&\Conj = \{ \Exp(\lam,t) \mid \lam \in C, \ t = \tconj(\lam) \}.
 \end{align*}

\begin{theorem}
\label{th:MMaxConj}
\begin{itemize}
\item[$(1)$]
$\Max = \MMax$,
\item[$(2)$]
$\Cut = \Mcut$,
\item[$(3)$]
$\Cut \cap \Conj = \Mconj$.
\end{itemize}
\end{theorem}
\begin{proof}
Items (1), (2) follow from Theorem~\ref{th:opt_synth} and Corollary~\ref{th:tcut}.

(3) Let $q \in \Mconj$, and let $q_s = \Exp(\lam,s)$, $s \in [0, \tt(\lam)$, be the optimal trajectory connecting $q_0$ with $q$. Thus there are no conjugate points at the interval $(0, \tt(\lam))$. We show that $\tt(\lam)$ is a conjugate time. 

Since $q = \Exp(\lam, \tt(\lam)) \in \Mconj$, then $(\lam,\tt(\lam)) \in \Nconj$, thus $\lam \in C_2$ and $\ts = 0$, $p = p_1^1(k)$, $(k \in (0,1)$. Then Propos.~\ref{propos:conjC2ts=0} states that $\tt(\lam) = 2 k p_1^1$ is a conjugate time.

We proved that $\tconj(\lam) = \tt(\lam)$. Thus $\Mconj \subset \Conj$, and in view of item (2) of this theorem we get $\Mconj \subset \Cut \cap \Conj$.
Now we prove that $\Cut \cap \Conj \subset \Mconj  $, i.e., $\Mcut \cap \Conj \subset \Mconj$. 

Fix any point $q \in \Mcut$. Then $q = \Exp(\lam,t)$ for some $(\lam,t) \in \Ncut = \NMax \sqcup \Nconj$. In order to complete the proof, we assume that $(\lam,t) \in \NMax$ and show that $q \notin \Conj$. Since $(\lam,t) \in \NMax$, then $t = \tt(\lam)$. We prove that $t < \tconj(\lam)$. 

If $\lam \in C_1 \cup C_4$, then $\tconj (\lam) = + \infty$ by Th.~\ref{th:conjC}.

Let $\lam \in C_2$. If $\ts = 0$, then $(\lam,t) \in \Nconj$, which is impossible since $(\lam,t) \in \NMax$. And if $\ts \neq 0$, then $t < \tconj(\lam)$ by Propos.~\ref{propos:conjC2ts=0}. 

The inclusion $\Cut \cap \Conj \subset \Mconj$ follows.
\end{proof}

\begin{theorem}
\label{th:cut_glob}
The cut locus has 3 connected components:
\begin{align}
&\Cut = \Cut_{\glob} \sqcup \Cut_{\loc}^+ \sqcup \Cut_{\loc}^-, \label{Cutglobloc} \\
&\Cut_{\glob} = \{ q \in M \mid \th = \pi\}, \label{Cutglob=}\\
&\Cut_{\loc}^+ = \{ q \in \hM \mid \th \in (- \pi, \pi), \ R_2 = 0, \ R_1 > R_1^1(|\th|\}, \label{Cutloc+=}\\
&\Cut_{\loc}^- = \{ q \in \hM \mid \th \in (- \pi, \pi), \ R_2 = 0, \ R_1 < -  R_1^1(|\th|\}, \label{Cutloc-=}
\end{align}
where the function $R_1^1$ is defined by Eq.~\eq{R11}.
The initial point $q_0$ is contained in the closure of the components $\Cut_{\loc}^+$, $\Cut_{\loc}^-$, and is separated from the component $\Cut_{\glob}$.
\end{theorem}
\begin{proof}
By Theorems~\ref{th:MMaxConj}, \ref{th:Expglob} and Lemma~\ref{lem:decompN'}, we have
$$
\Cut = \Mcut = \Exp(\Ncut) = \cup_{i \in C} M_i'.
$$
Denote
\begin{align*}
&C_{\glob} = \{ 1, \dots, 8, 17, \dots, 24, 33, 34\}, \\
&C_{\loc}^+ = \{13, \dots, 16, 29, \dots, 32\}, \\
&C_{\loc}^- = \{ 9, \dots, 12, 25, \dots, 28\}.
\end{align*} 
Then
$$
\cup_{i \in C_{\glob}} M_i' = \Cut_{\glob}, \qquad
\cup_{i \in C_{\loc}^+} M_i' = \Cut_{\loc}^+, \qquad
\cup_{i \in C_{\loc}^-} M_i' = \Cut_{\loc}^-, 
$$ 
and decomposition~\eq{Cutglobloc}--\eq{Cutloc-=} follows.

The topological properties of  $\Cut_{\glob}$, $\Cut_{\loc}^\pm$ follow from equalities \eq{Cutglobloc}--\eq{Cutloc-=}.
\end{proof}

The curve $\G = \Cut \cap \Conj$ has the following asymptotics near the initial point $q_0$:
$$
R_1 = R_1^1(\th) = \sqrt[3]{\pi}/2 \, \th^{2/3} + o (\th^{2/3}), \quad \th \to 0, \qquad R_2 = 0,
$$
see item (5) of Lemma~\ref{lem:Gam1}. This agrees with the result on asymptotics of conjugate locus for contact sub-Riemannian structures in $\R^3$ obtained by A.~Agrachev~\cite{agrachev_contact}, and J.-P.~Gauthier et al~\cite{gauthier_contact}.

\twofiglabel
{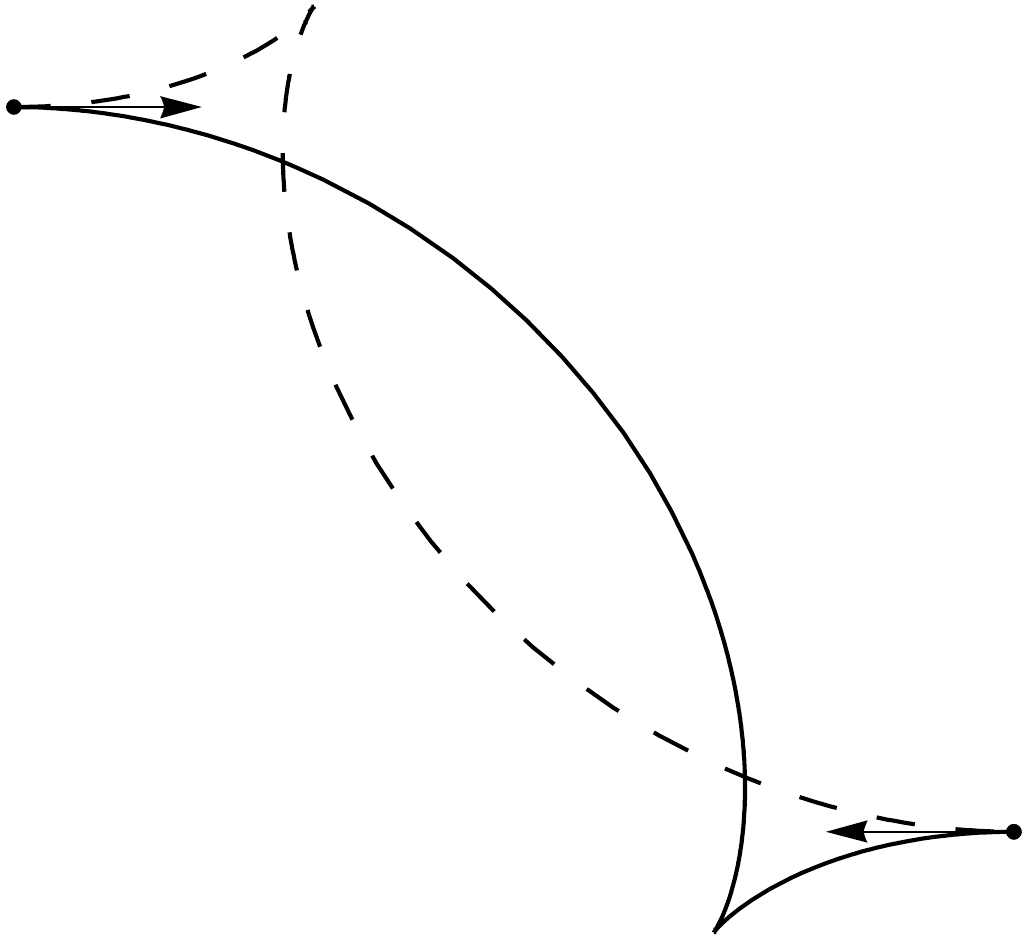}{Cut point for $\lam \in C_1$, generic case (Optimal solutions for $\th_1 = \pi$)}{fig:cutC12n}
{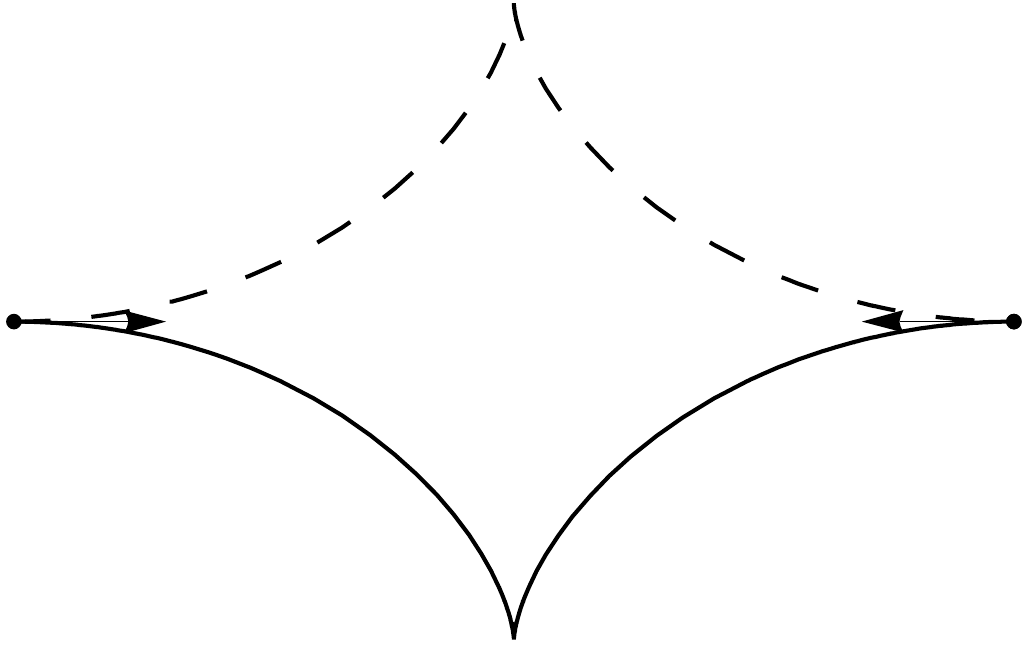}{Cut point for $\lam \in C_1$, symmetric case with cusp (Optimal solutions for $x_1 \neq 0$, $y_1 = 0$, $\th_1 = \pi$)}{fig:cutC14n}

\onefiglabelrotate
{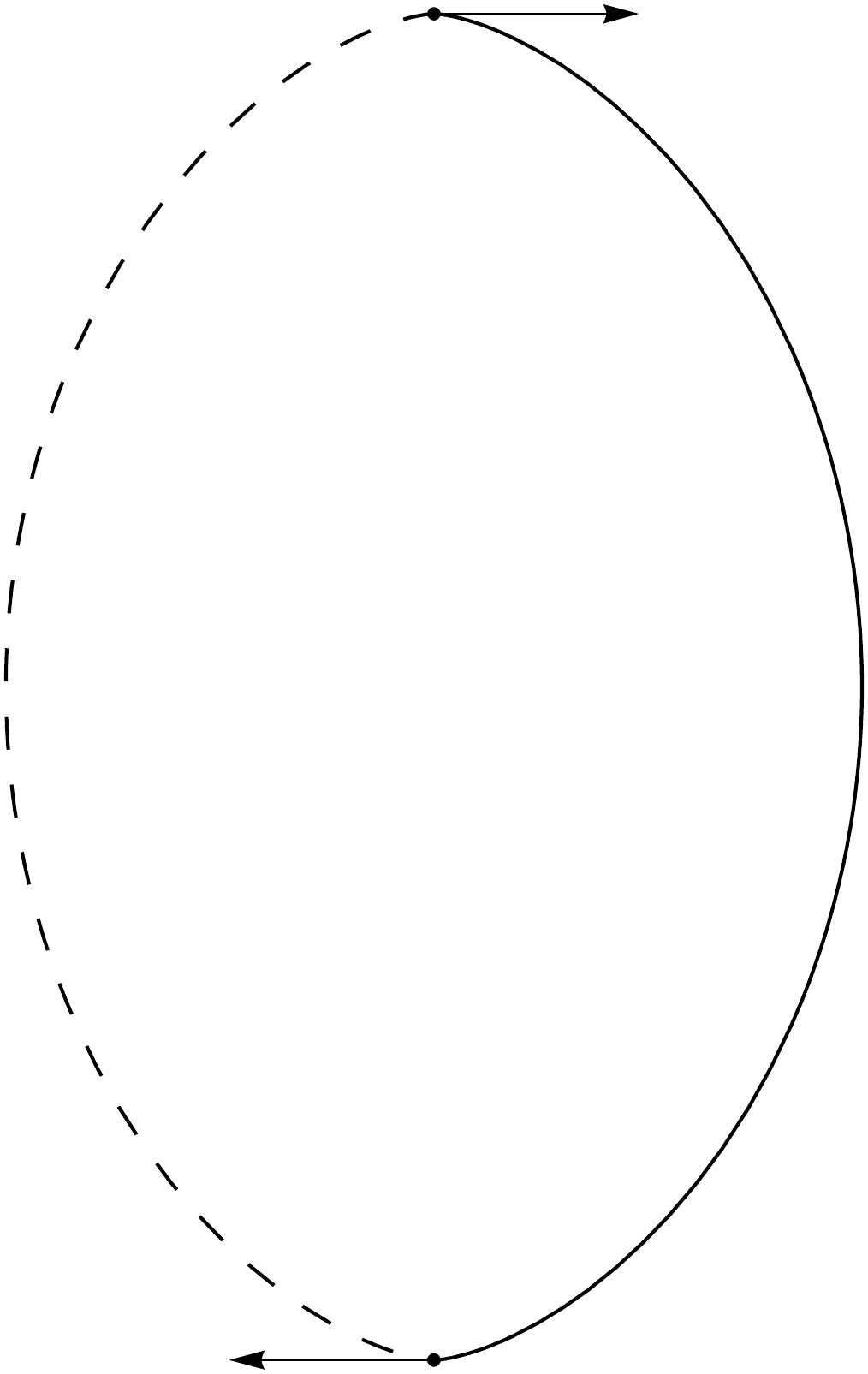}{Cut point for $\lam \in C_1$, symmetric case without cusp (Optimal solutions for $x_1 = 0$, $y_1 \neq 0$, $\th_1 = \pi$)}{fig:cutC15n}{0}

\twofiglabel
{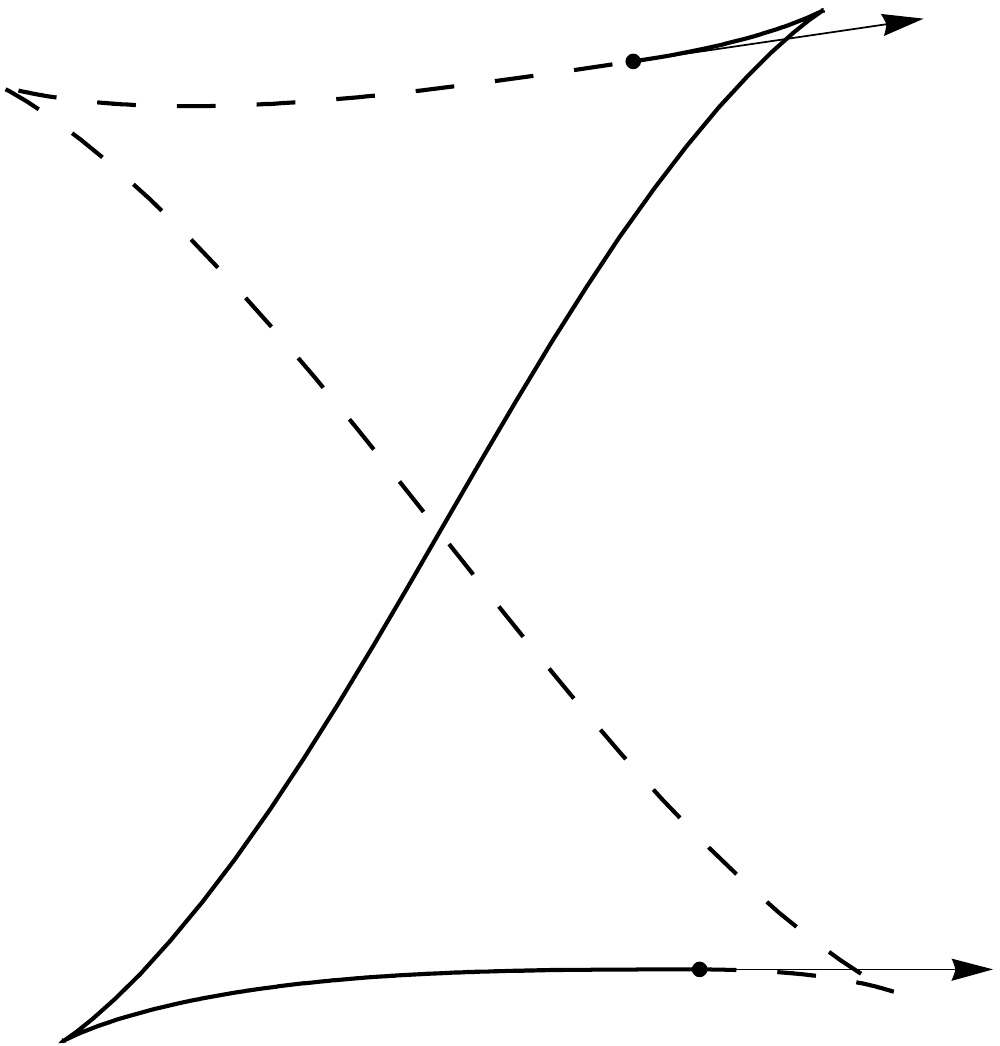}{Cut point for $\lam \in C_2$, generic case}{fig:cutC22n}
{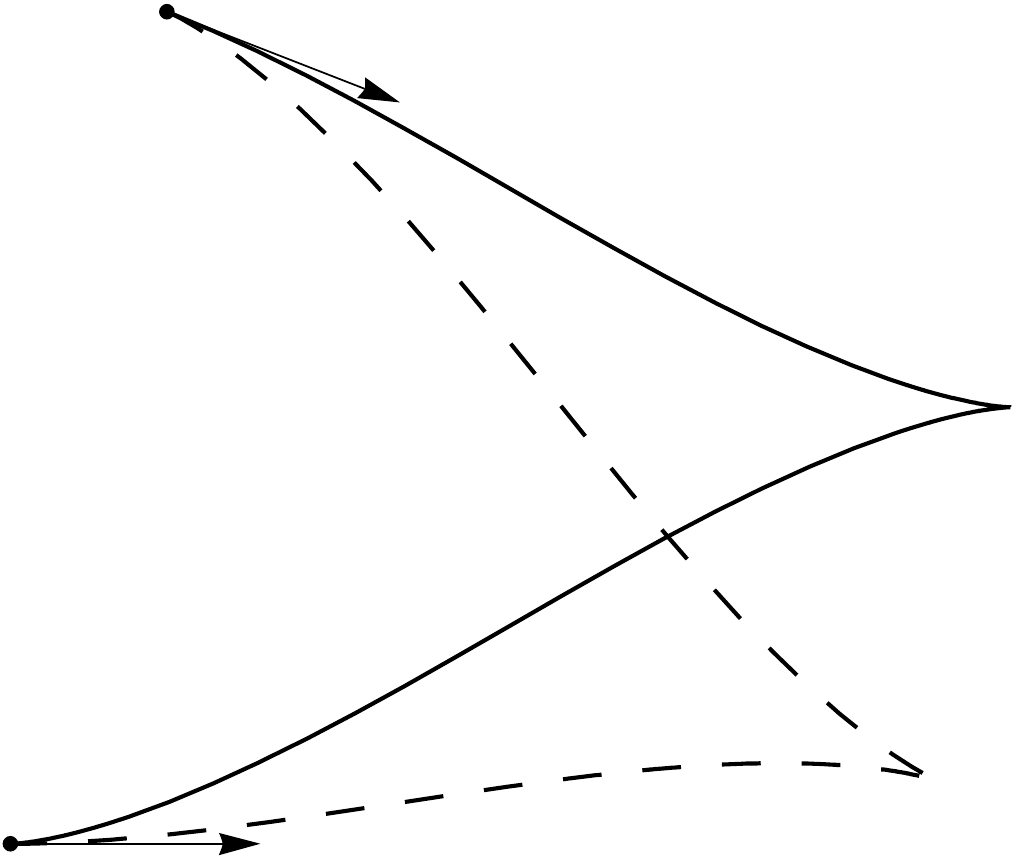}{Cut point for $\lam \in C_2$, special case with one cusp}{fig:cutC21n}

\twofiglabel
{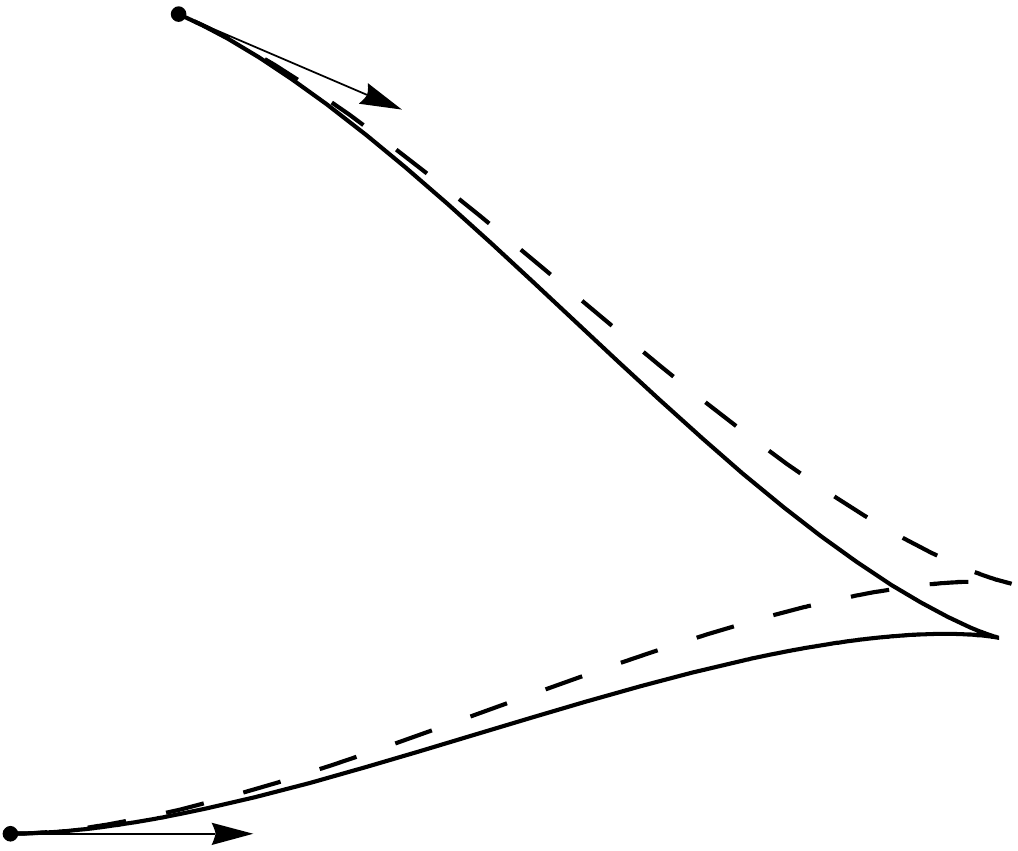}{Cut point for $\lam \in C_2$ approaching conjugate point}{fig:cutC24n}
{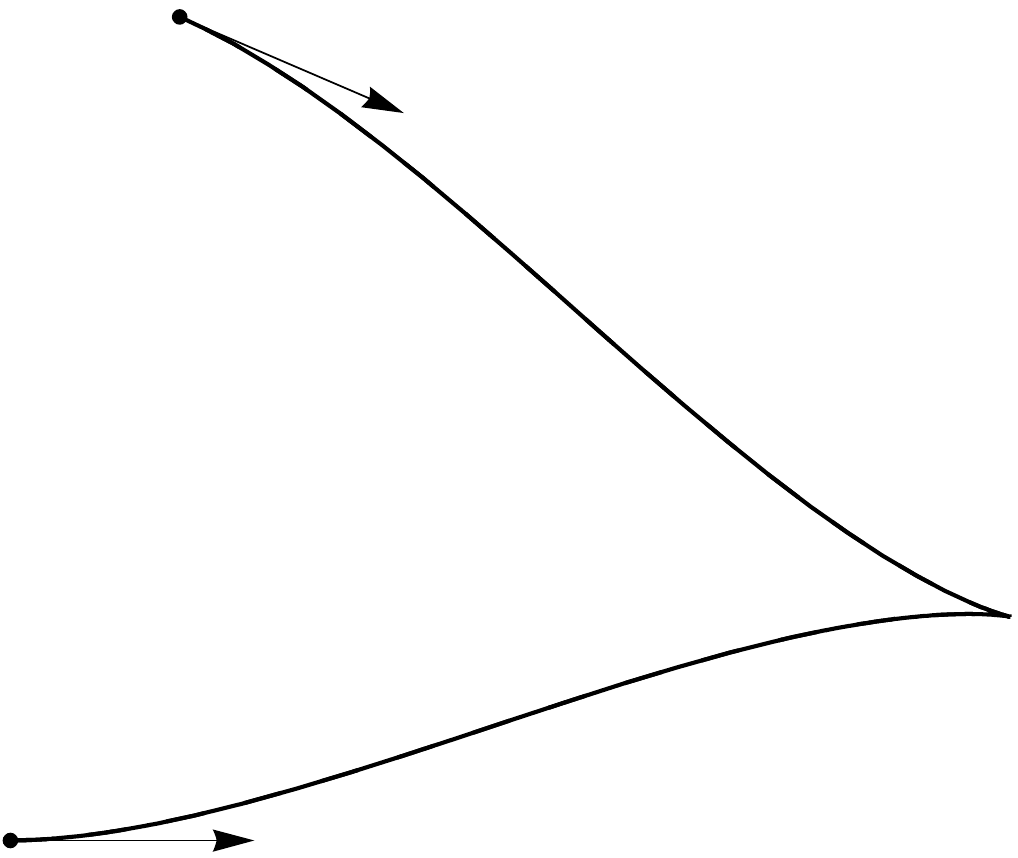}{Cut point for $\lam \in C_2$ coinciding with conjugate point}{fig:cutC23n}

\onefiglabel
{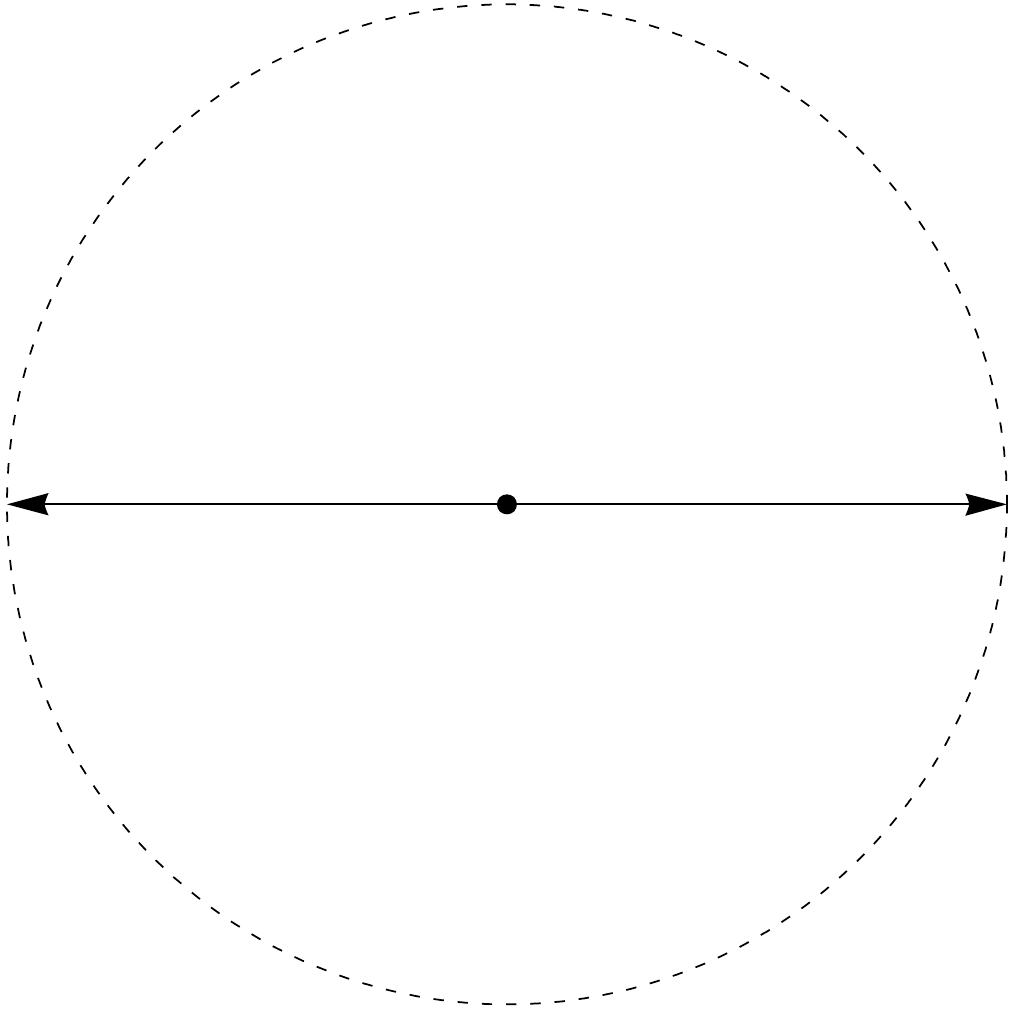}{Cut point for $\lam \in C_4$ (Optimal solutions for $x_1 = y_1 = 0$, $\th_1 = \pi$)}{fig:cutC4n}

\subsection{Explicit optimal solutions for special terminal points}
In this subsection we describe optimal solutions for particular  terminal points $q_1 = (x_1, y_1, \th_1)$. Where applicable, we interpret the optimal trajectories in terms of the corresponding optimal motion of a car in the plane.

For generic terminal points, we developed a software in computer system Mathematica~\cite{math} for numerical evaluation of  solutions to the problem. 

\subsubsection{$x_1 \neq 0$, $y_1 = 0$, $\th_1 = 0$}
In this case $\nu \in N_5$, and the optimal trajectory is
$$
x_t = t \sgn x_1, \quad y_t = 0, \quad \theta_t = 0, \qquad t \in [0, t_1], \quad t_1 = |x_1|,
$$
and the car moves uniformly forward or backward along a segment.

\subsubsection{$x_1 = 0$, $y_1 = 0$, $|\th_1| \in ( 0, \pi)$}
We have $\nu \in N_4$,  and the optimal solution is 
$$
x_t = 0, \quad y_t = 0, \quad \theta_t = t \sgn \th_1, \qquad t \in [0, t_1], \quad t_1 = |\th_1|,
$$
the car rotates uniformly around itself by the angle $\th_1$. 

\subsubsection{$x_1 = 0$, $y_1 = 0$, $\th_1 =  \pi$}
We have $\nu \in N_4$,  and there are two  optimal solutions: 
$$
x_t = 0, \quad y_t = 0, \quad \theta_t = \pm t, \qquad t \in [0, t_1], \quad t_1 = \pi,
$$
the car rotates uniformly around itself clockwise or counterclockwise by the angle $\pi$, see Fig.~\ref{fig:cutC4n}. 

\subsubsection{$x_1 \neq 0$, $y_1 = 0$, $\th_1 =  \pi$}
\label{subsubsec:n0pi}
There are two  optimal solutions: 
\begin{align*}
&x_t = (\sgn x_1)/k(t + E(k) - E(K + t, k)), \quad 
y_t = (s/k) (\sqrt{1-k^2} - \dn(K + t, k)), \\
&\theta_t = s \sgn x_1 (\pi/2 - \am(K + t, k)), \quad s = \pm 1, \qquad t \in [0, t_1], \quad t_1 = 2 K, \\
\intertext{and $k \in (0,1)$ is the root of the equation}
&(2/k) (K(k) - E(k)) = |x_1|,
\end{align*}
see Fig.~\ref{fig:cutC14n}.

\subsubsection{$x_1 = 0$, $y_1 \neq 0$, $\th_1 =  \pi$}
\label{subsubsec:0npi}
There are two  optimal solutions: 
\begin{align*}
&x_t = s(1 - \dn(t,k))/k, \quad 
y_t = (\sgn y_1/k) (t - E(t,k)), \\
&\theta_t = s \sgn y_1   \am( t, k), \quad s = \pm 1, \qquad t \in [0, t_1], \quad t_1 = 2 K, \\
\intertext{and $k \in (0,1)$ is the root of the equation}
&(2/k) (K(k) - E(k)) = |y_1|,
\end{align*}
see Fig.~\ref{fig:cutC15n}.

\subsubsection{$x_1 = 0$, $y_1 \neq 0$, $\th_1 =  0$}
\label{subsubsec:0n0}
There are two  optimal solutions given by formulas for $(x_t, y_t, \th_t)$ for the case $\lam \in C_2$ in Subsec.~3.3~\cite{max_sre2} for the following values of parameters:
\begin{align*}
& t \in [0, t_1], \quad t_1 = 2 k p_1^1(k), \\
\intertext{with the function $p_1^1(k)$ defined in Lemma 5.3 \cite{max_sre2},}
&s_2 = - \sgn y_1, \qquad \psi = \pm K(k) - p_1^1(k),\\
\intertext{and $k \in (0,1)$ is the root of the equation}
&2 (p_1^1(k) - E(p_1^1(k),k) \sqrt{1-k^2}/\dn(p_1^1(k),k)) = |y_1|,
\end{align*}
see Fig.~\ref{fig:sol0n0}.

\onefiglabelsizen
{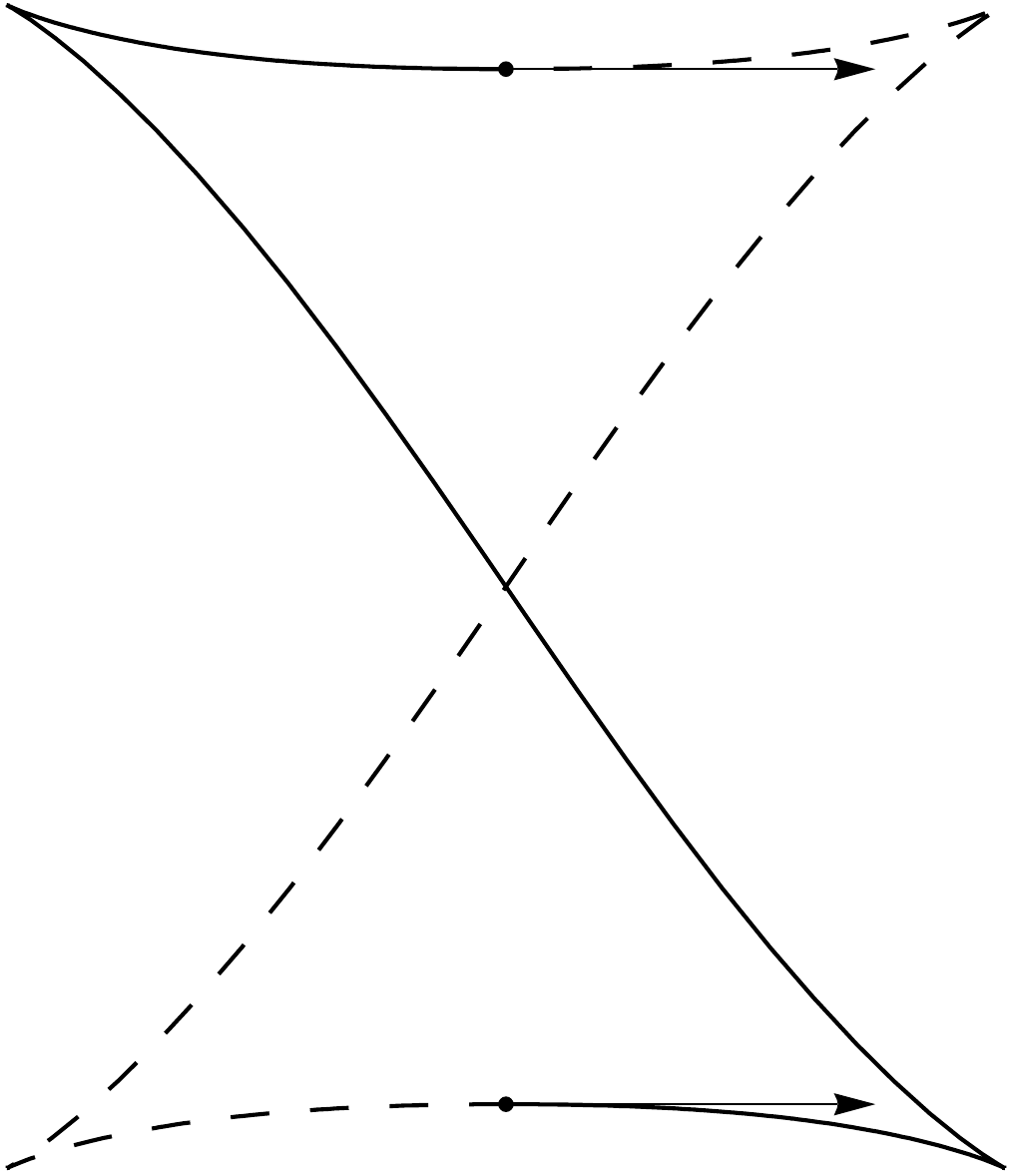}{Optimal solutions for $x_1 = 0$, $y_1 \neq 0$, $\theta_1 = 0$}{fig:sol0n0}{5}

\subsubsection{$(x_1,y_1) \neq 0$, $\th_1 =  \pi$}
Introduce the polar coordinates $x_1 = \rho_1 \cos \chi_1$, $y_1 = \rho_1 \sin \chi_1$.
There are two  optimal solutions given by formulas for $(x_t, y_t, \th_t)$ for the case $\lam \in C_1$ in Subsec.~3.3~\cite{max_sre2} for the following values of parameters:
\begin{align*}
& t \in [0, t_1], \quad t_1 = 2 K(k), \\
\intertext{and $k \in (0,1)$ is the root of the equation}
&2 (p_1^1(k) - E(p_1^1(k),k) \sqrt{1-k^2}/\dn(p_1^1(k),k)) = \rho_1,\\
&s_1 = \pm 1, \qquad \f = s_1 F(\pi/2 - \chi_1, k),
\end{align*}
see Figs.~\ref{fig:cutC12n}, \ref{fig:cutC14n}, \ref{fig:cutC15n}. In the cases $y_1 = 0$ and $x_1 = 0$ we get  respectively the cases considered in Subsubsecs.~\ref{subsubsec:n0pi} and \ref{subsubsec:0npi}

\subsubsection{$y_1 \neq 0$, $\th_1 =  0$}
There is a unique  optimal solution given by formulas for $(x_t, y_t, \th_t)$ for the case $\lam \in C_2$ in Subsec.~3.3~\cite{max_sre2} for the following values of parameters:
\begin{align*}
&s_2 = - \sgn y_1, \\
\intertext{$k \in (0,1)$ and $p \in (0, p_1^1(k)]$ are solutions to the system of equations}
&s (\sgn y_1) 2 k f_1(p,k)/\dn(p,k) = x_1, \qquad s = \pm 1, \\
&2 (p - \E(p)) \sqrt{1-k^2}/\dn(p,k) = |y_1|, \\
\intertext{and}
&t \in [0, t_1], \quad t_1 = 2 k p, \qquad \psi = s K(k) - p. 
\end{align*}

\subsection{Plots of caustic, cut locus, and sub-Riemannian spheres}

Here we collect 3-dimensional plots of some essential objects in the sub-Riemannian problem on $\SE(2)$.


Figure~\ref{fig:causR1R2th} shows the sub-Riemannian caustic in the rectifying coordinates $(R_1, R_2, \th)$.

\onefiglabelsizen
{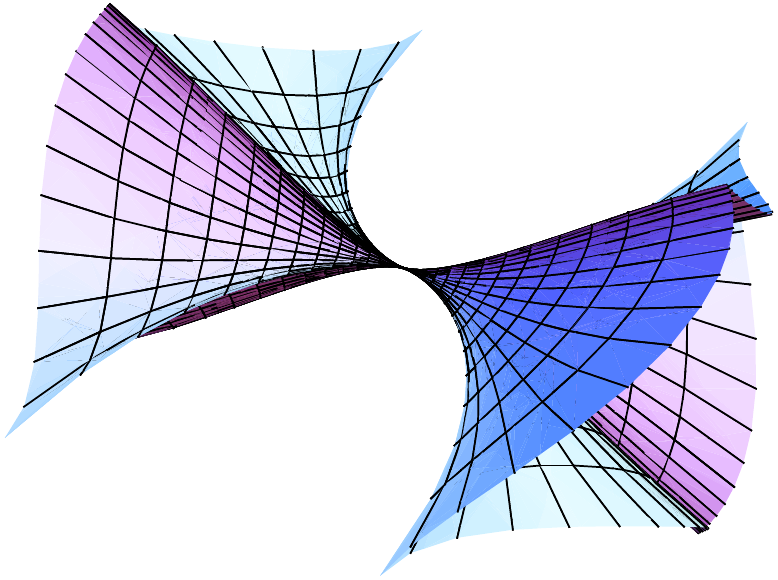}{Sub-Riemannian caustic}{fig:causR1R2th}{7}

The cut locus in rectifying coordinates $(R_1, R_2, \theta)$ is presented at Fig.~\ref{fig:max1}, notice that here the horizontal planes $\th = 0$ and $\th = 2 \pi$ should be identified. 
Global embedding of the cut locus to the solid torus (diffeomorphic image of the state space $M = \SE(2)$) is shown at Fig.~\ref{fig:max2}.
 
\onefiglabelsizen
{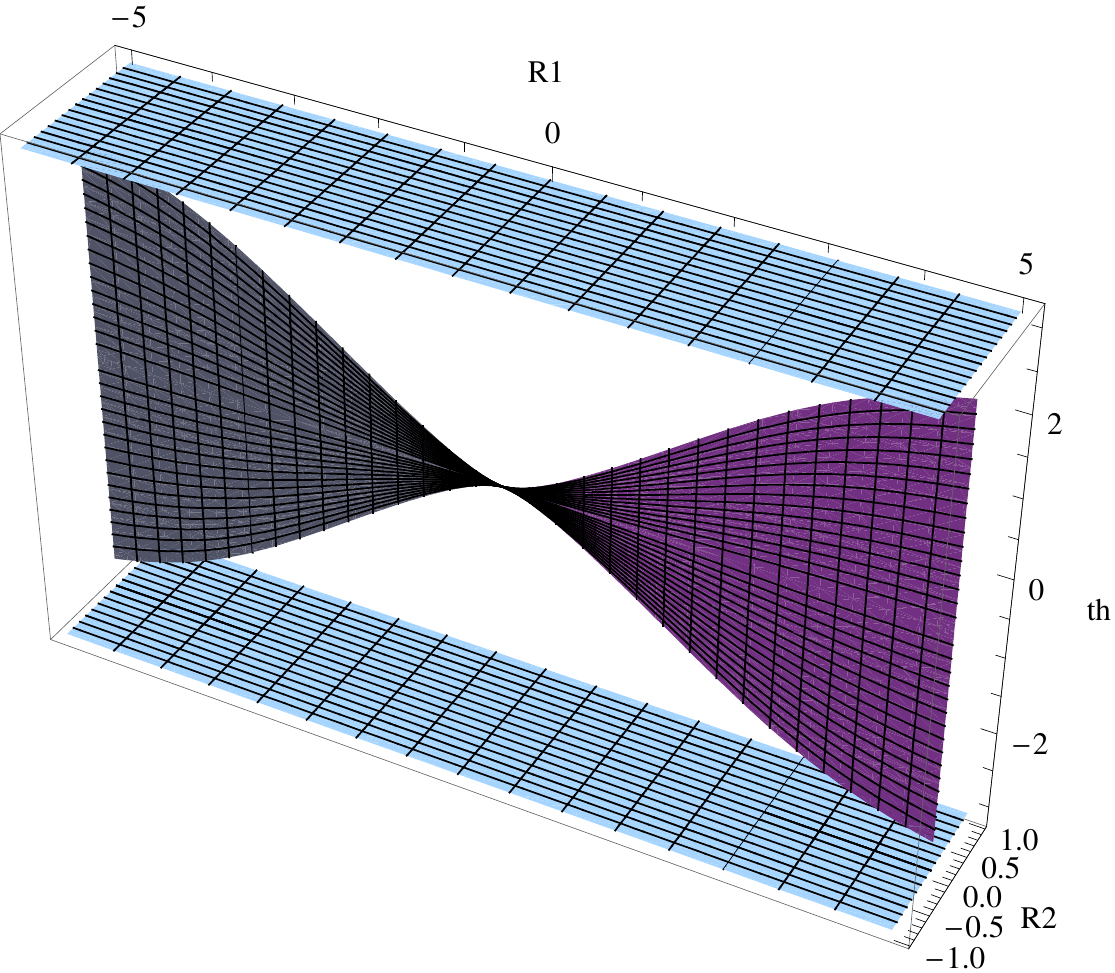}{Cut locus in rectifying coordinates $(R_1, R_2, \theta)$}{fig:max1}{10}

\onefiglabelsizen
{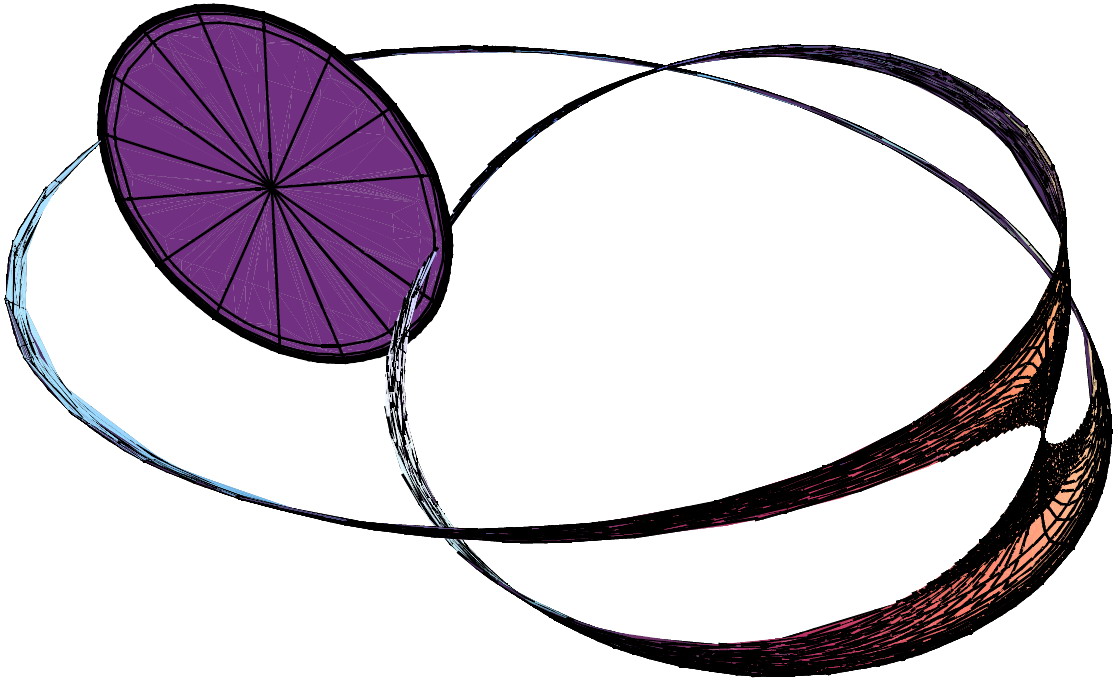}{Cut locus:  global view}{fig:max2}{8}


Figures~\ref{fig:sphereRpi2}--\ref{fig:sphereR32pitorus} present sub-Riemannian spheres
$$
S_R = \{ q \in M \mid d(q_0, q) = R\}
$$
of different radii $R$,
where
$$
d(q_0, q_1) = \inf \{ l(q(\cdot)) \mid q(\cdot) \text{ trajectory of \eq{sys1}, } q(0) = q_0, \ q(t_1) = q_1\}
$$
is the sub-Riemannian distance --- the cost function in the sub-Riemannian problem~\eq{sys1}--\eq{J}.

In this problem sub-Riemannian spheres    can be of three different topological classes. If 
$R \in (0, \pi)$, then $S_R$ is homeomorphic to the standard 2-dimensional Euclidean sphere   $S^2$, see Fig.~\ref{fig:sphereRpi2torus}. For
$R = \pi$ the sphere $S_R$ is homeomorphic to    the sphere $S^2$ with its north pole $N$ and south pole $S$ identified:  $S_{\pi} \cong S^2/ \{N = S\}$, see Fig.~\ref{fig:sphereRpitorus}. And if 
$R > \pi$,  then $S_R$ is homeomorphic to the 2-dimensional torus, see Fig.~\ref{fig:sphereR32pitorus}. 

Figure~\ref{fig:sphereRpi2} shows the sphere  $S_{\pi/2}$ in rectifying coordinates $(R_1,R_2,\th)$.
Figure~\ref{fig:sphereRpi2cut} represents the same sphere with a cut-out opening the singularities of the sphere: the sphere intersects itself at the local components of the cut locus $\Cut_{\loc}^\pm$.
Figure~\ref{fig:sphereRpi2torus} shows embedding of the same sphere to the solid torus.
Sub-Riemannian spheres  of small radii resemble the well-known apple-shaped sub-Riemannian sphere in the Heisenberg group~\cite{versh_gersh}. Although, there is a major difference: the sphere in the Heisenberg group has a one-parameter family of symmetries (rotations), but the sphere in $\SE(2)$ has only a discrete group of symmetries $G = \{\Id, \eps^1, \dots, \eps^7\}$ (reflections).

Figures~\ref{fig:sphereRpi}--\ref{fig:sphereRpitorus} represent similarly the sub-Riemannian sphere of the critical radius $\pi$. In addition to self-intersections at the local components $\Cut_{\loc}^\pm$, the sphere $S_{\pi}$ has one self-intersection point at the global component $\Cut_{\glob}$.

Figures~\ref{fig:sphereR32pi}--\ref{fig:sphereR32pitorus} show similar images of the sphere of radius $3 \pi/2$.  For $R > \pi$ the sphere $S_R$ has two topological segments of self-intersection points at $\Cut_{\loc}^\pm$ respectively, and a topological circle of self-intersection points at $\Cut_{\glob}$.
  
Figures~\ref{fig:half1sphere}, \ref{fig:half1sphere} shows intersections of the spheres $S_{\pi/2}$, $S_{\pi}$, $S_{3\pi/2}$ with the half-spaces  $\th < 0$ and $R_2 > 0$ respectively.

Figure~\ref{fig:frontRpi} shows self-intersections of the wavefront
$$
W_R = \{ \Exp(\lam, R) \mid \lam \in C \}
$$
for $R = \pi$.

\twofiglabelsize
{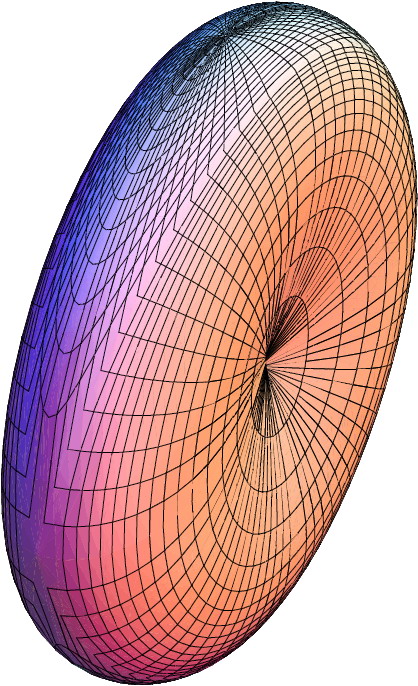}{Sub-Riemannian sphere $S_{\pi/2}$}{fig:sphereRpi2}{7}
{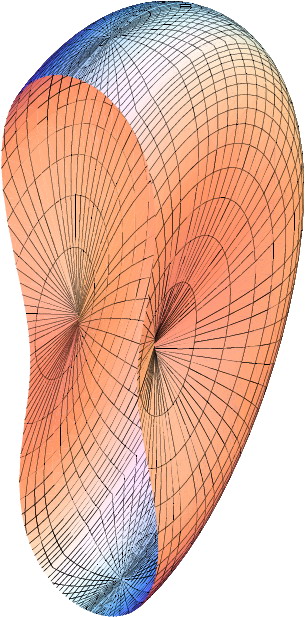}{Sub-Riemannian  sphere $S_{\pi/2}$ with cut-out  }{fig:sphereRpi2cut}{7}

\onefiglabelsizen
{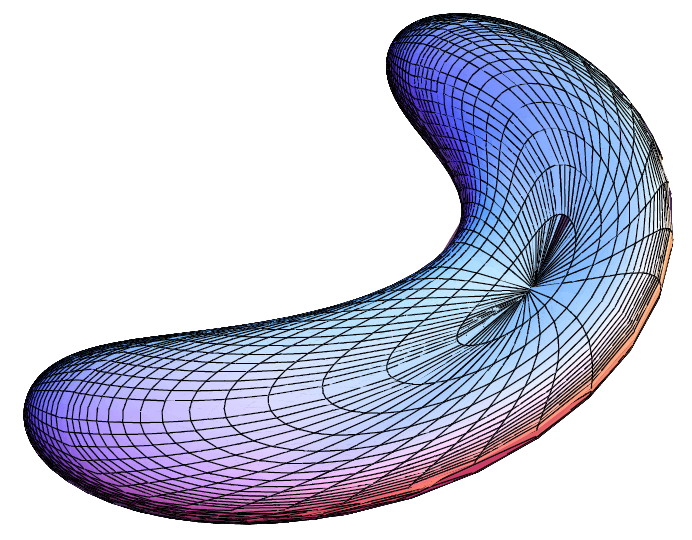}{Sub-Riemannian sphere $S_{\pi/2}$, global view}{fig:sphereRpi2torus}{7}

\twofiglabelsize
{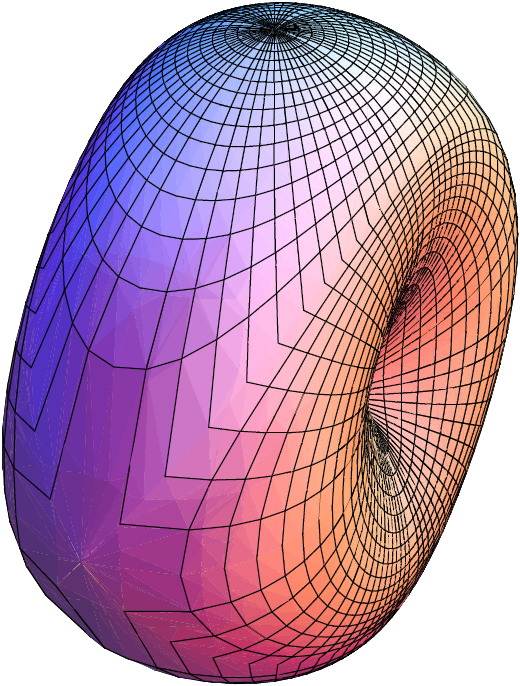}{Sub-Riemannian sphere $S_{\pi}$}{fig:sphereRpi}{7}
{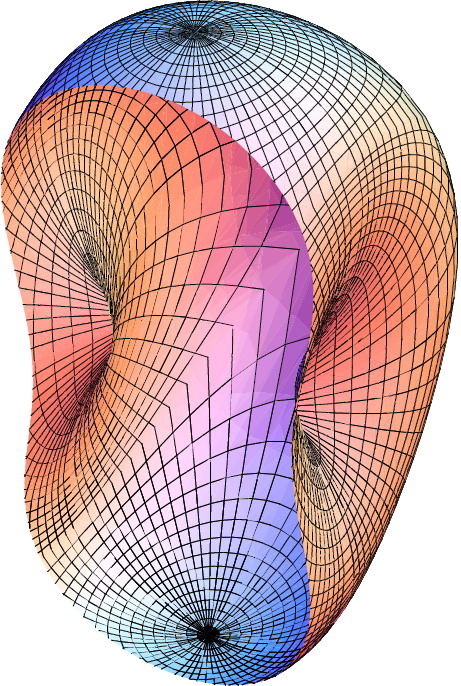}{Sub-Riemannian  sphere $S_{\pi}$ with cut-out  }{fig:sphereRpicut}{7}

\onefiglabelsizen
{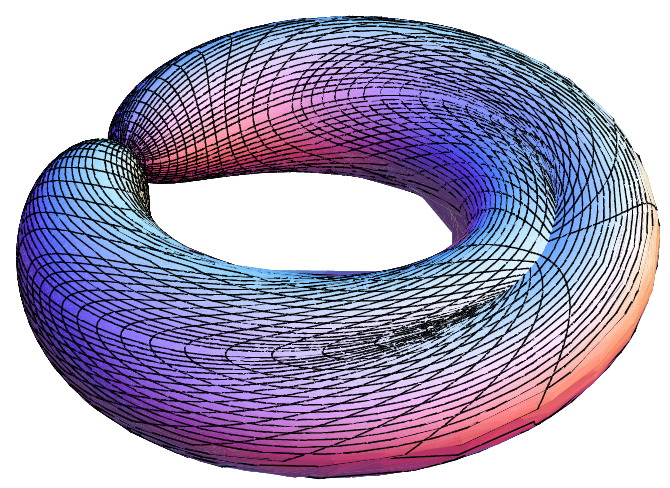}{Sub-Riemannian sphere $S_{\pi}$, global view}{fig:sphereRpitorus}{7} 

\twofiglabelsize
{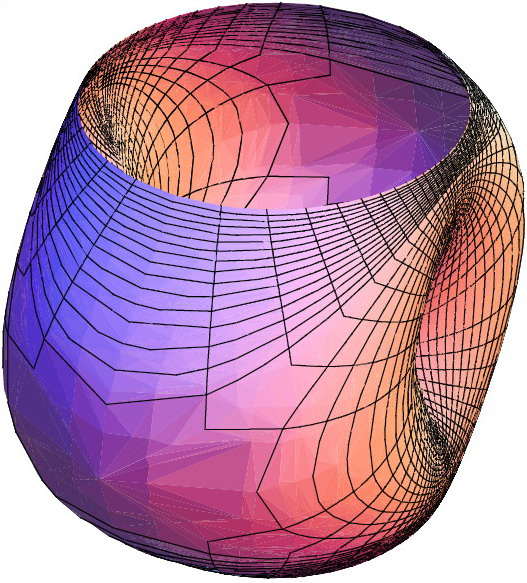}{Sub-Riemannian sphere $S_{3\pi/2}$}{fig:sphereR32pi}{7}
{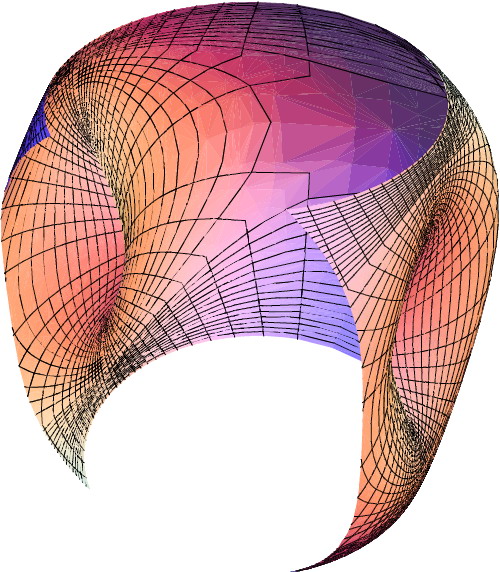}{Sub-Riemannian  sphere $S_{3\pi/2}$ with cut-out  }{fig:sphereR32picut}{7}

\onefiglabelsizen
{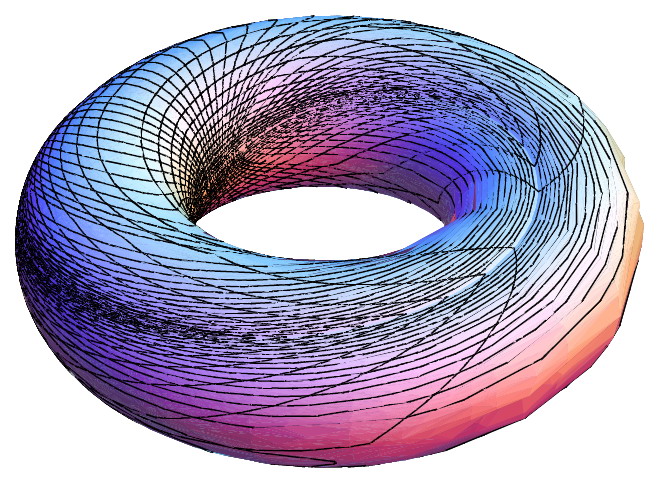}{Sub-Riemannian sphere $S_{3\pi/2}$, global view}{fig:sphereR32pitorus}{7}

\onefiglabelsizen
{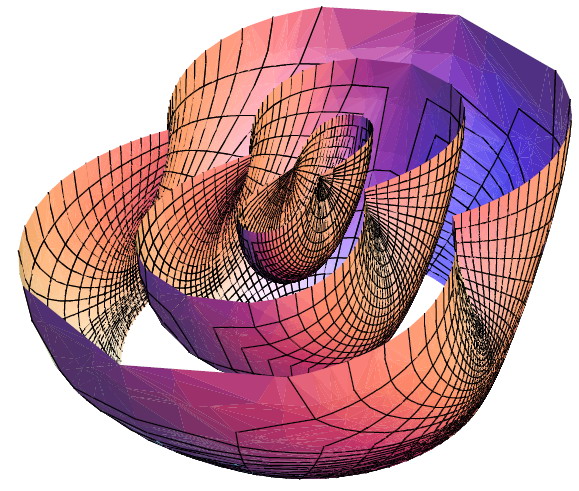}{Sub-Riemannian hemi-spheres $S_{\pi/2}$, $S_{\pi}$, $S_{3\pi/2}$ for  $\th < 0$}{fig:half1sphere}{7}

\onefiglabelsizen
{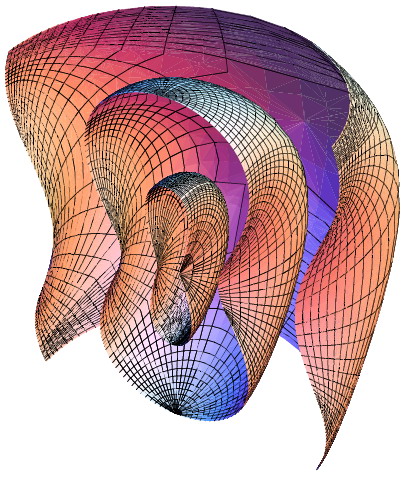}{Sub-Riemannian hemi-spheres $S_{\pi/2}$, $S_{\pi}$, $S_{3\pi/2}$ for  $R_2 > 0$}{fig:half2sphere}{7}

\onefiglabelsizen
{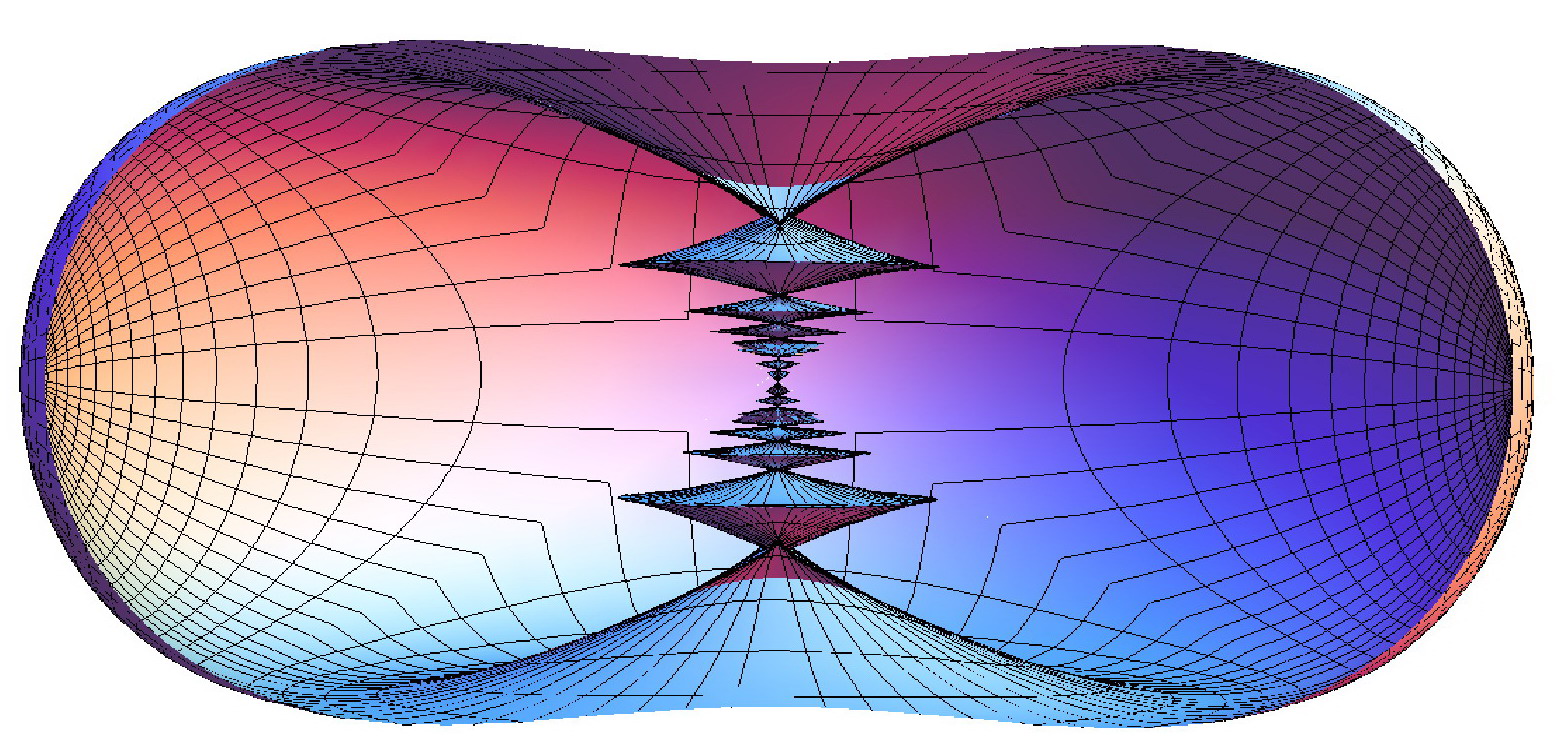}{Wavefront $W_{\pi}$ with cut-out }{fig:frontRpi}{5}


\section{Conclusion}
\label{sec:concl}

The solution to the sub-Riemannian problem on $\SE(2)$ obtained in this paper and the previous one~\cite{max_sre2} is based on a detailed study of the action of the discrete group of symmetries of the exponential mapping. This techniques was already partially developed in the study of   related optimal control problems (nilpotent sub-Riemannian problem with the growth vector (2,3,5)~\cite{dido_exp, max1, max2, max3} and Euler's elastic problem~\cite{el_max, el_conj}). The sub-Riemannian problem on $\SE(2)$ is the first problem in this series, where a complete solution was obtained (local and global optimality, cut time and cut locus, optimal synthesis). We believe that our approach based on the study of symmetries will provide such complete results in other symmetric invariant problems, such as the nilpotent sub-Riemannian problem with the growth vector (2,3,5),  the ball-plate problem~\cite{jurd_book}, and others. 

The sub-Riemannian problem on the group of rototranslations $\SE(2)$ appeared recently as an important model in robotics~\cite{laumond} and vision~\cite{petitot, citti}. We expect that our results   can be applied to these domains.

\section*{Acknowledgements}

The author is grateful to Prof. Andrei Agrachev for bringing the sub-Riemannian problem on $\SE(2)$ to authors' attention.

The author wishes to thank Prof. Arrigo Cellina for hospitality and excellent conditions for work at the final version of this paper.

\newpage
\listoffigures\addcontentsline{toc}{section}{\listfigurename}

\listoftables

\end{document}